\documentclass[
a4paper,
oneside,
DIV=12,
pagesize,
]{scrreprt}

\makeatletter{}
\usepackage[utf8]{inputenc}
\usepackage[sc, osf]{mathpazo}
\usepackage[T1]{fontenc}
\usepackage{textcomp}
\usepackage{bm}

\usepackage{amsmath,amssymb,amsthm,graphicx,float,url,mathrsfs}
\usepackage[hang]{subfigure}

\usepackage{mathtools}

\usepackage[ngerman,english]{babel}

\usepackage{tikz}
\usepackage{pgfplots}

\makeatletter{}
\addtokomafont{disposition}{\rmfamily}

\newcounter{subeq}

\usepackage[
  headsepline,
  ]{scrpage2}
\automark[section]{chapter}

\clearscrheadfoot
\ihead[]{\headmark}
\ohead[]{\pagemark}
\KOMAoptions{headinclude=true,footinclude=false}

\usepackage{color} \definecolor{darkblue}{rgb}{0,0,.5}
\definecolor{darkred}{rgb}{.5,0,0}
\definecolor{indigo}{rgb}{.29,0,.55}

\usepackage{pdfsync}

\newtheorem{theorem}{Theorem}[chapter]
\newtheorem{proposition}[theorem]{Proposition}
\newtheorem{corollary}[theorem]{Corollary}
\newtheorem{definition}[theorem]{Definition}
\newtheorem{lemma}[theorem]{Lemma}
\newtheorem{example}[theorem]{Example}

\theoremstyle{definition}
\newtheorem{remark}[theorem]{Remark}

\usepackage{enumitem}
\setlist{parsep=0em} \setenumerate{label=(\roman*),topsep=0.5\baselineskip}

\newlist{enumarab}{enumerate}{1}
\setlist[enumarab,1]{
  label=(\arabic*),
  ref=(\arabic*),
  fullwidth,
  listparindent=\parindent, 
  itemindent=\parindent,
  leftmargin=0em,
  itemsep=0em
}

\newlist{enumalph}{enumerate}{1}
\setlist[enumalph,1]{
  label=(\alph*),
  ref=(\alph*),
  fullwidth,
  listparindent=\parindent,
  itemindent=\parindent,
  leftmargin=0em
}

\newlist{enumrom}{enumerate}{1}
\setlist[enumrom,1]{
  label=(\roman*),
  ref=(\alph*),
  fullwidth,
  listparindent=\parindent,
  itemindent=\parindent,
  leftmargin=0em
}

\newlist{enumpara}{enumerate}{1}
\setlist[enumpara,1]{
  label=\emph{\arabic*}.,
  ref=\arabic*,
  wide,
  nosep
}

\newlist{enumparaalph}{enumerate}{1}
\setlist[enumparaalph,1]{
  label=\alph*),
  ref=\alph*),
  wide,
  nosep
}

\makeatletter \newcommand{\thmenumhspace}[1]{\sbox{\@labels}{\unhbox\@labels\hskip#1}} \makeatother

\renewcommand{\i}{\mathrm{i}} 
\newcommand{\dd}{\, \mathrm{d}} \newcommand{\diff}{\mathop{}\!\mathrm{d}}

\newcommand{\argmin}{\operatornamewithlimits{arg\,min}} \newcommand{\sgn}{\operatorname{sgn}}
\newcommand{\supp}{\operatorname{supp}}
\newcommand{\e}{\mathrm{e}}
\renewcommand\Re{\operatorname{Re}}

\usepackage{xspace}
\makeatletter
\newcommand\etc{etc\@ifnextchar.{}{.\@\xspace}}
\newcommand\ie{i.e.\@ifnextchar,{}{\@\xspace}}
\newcommand\eg{e.g.\@ifnextchar,{}{\@\xspace}}
\newcommand\wrt{w.r.t.\@ifnextchar,{}{\@\xspace}}
\makeatother

\newcommand{\leadeq}[2][4]{\MoveEqLeft[#1] #2 \nonumber}
\newcommand{\leadeqnum}[2][4]{\MoveEqLeft[#1] #2}

\KOMAoptions{DIV=last}

\begin{document}
\flushbottom

\title{Minimizers and Gradient Flows of Attraction-Repulsion Functionals with Power Kernels and Their Total Variation Regularization} 
\subtitle{Minimierer und Gradientenflüsse von Anziehungs-Abstoßungs-Funktionalen mit Potenz-Kernen und ihre Regularisierung mittels totaler Variation}
\titlehead{Technische Universität München\\
  Department of Mathematics}
\publishers{
  \begin{tabular}{rl}
    Supervisor: & Prof. Dr. Massimo Fornasier\\
    Advisor: & Prof. Dr. Daniel Matthes
  \end{tabular}
}
\author{Master’s Thesis by Jan-Christian Hütter}
\date{Submission date: May 31, 2013}

\pagenumbering{roman}
\maketitle
\pagestyle{scrheadings}
\clearpage{}\makeatletter{}\noindent I assure the single handed composition of this Master's thesis only supported by declared resources.

\vspace{3\baselineskip}
\noindent{}Garching, May 31, 2013,
\hfill
\begin{minipage}[t]{9cm}
  \centering \hrulefill \\
  Jan-Christian Hütter
\end{minipage}
\cleardoublepage

\selectlanguage{ngerman}
\section*{\abstractname}

Wir untersuchen Eigenschaften eines Anziehungs-Abstoßungs-Funktionals, welches durch Potenz-Kerne gegeben ist und das zum Halftonen von Bildern eingesetzt werden kann. Im ersten Teil dieser Arbeit untersuchen wir die Existenz und das Verhalten von Minimieren des Funktionals und dessen Regularisierung mittels der totalen Variation, wobei wir auf variationelle Konzepte für Wahrscheinlichkeitsmaße zurückgreifen. Darüberhinaus führen wir Partikelapproximationen sowohl zum Funktional als auch zu seiner Regularisierung ein und beweisen ihre Konsistenz im Sinne von \( \Gamma \)-Konvergenz, die wir zusätzlich durch numerische Beispiele verdeutlichen. Im zweiten Teil betrachten wir den Gradientenfluss des Funktionals im \( 2 \)-Wasserstein Raum und beweisen Aussagen über sein asymptotisches Verhalten für große Zeiten, wofür wir die Technik der Pseudo-Inversen eines Wahrscheinlichkeitsmaßes in 1D verwenden. Abhängig von den gewählten Parametern beinhaltet dies Konvergenz einer Teilfolge gegen einen Gleichgewichtszustand oder sogar Konvergenz der gesamten Trajektorie gegen einen explizit bestimmbaren Grenzwert. Ein wichtiger Bestandteil der Argumentation ist in beiden Teilen dieser Arbeit die verallgemeinerte Fouriertransformation, mit deren Hilfe die konditionelle Positiv-Definitheit des Interaktionskerns im Falle übereinstimmender Anziehungs- und Abstoßungs-Exponenten nachgewiesen werden kann.
 
\selectlanguage{english}
\section*{\abstractname}

We study properties of an attractive-repulsive energy functional based on power-kernels, which can be used for halftoning of images. In the first part of this work, using a variational framework for probability measures, we examine existence and behavior of minimizers to the functional and to a regularization of it by a total variation term. Moreover, we introduce particle approximations to the functional and to its regularized version and prove their consistency in terms of \( \Gamma \)-convergence, which we additionally illustrate by numerical examples. In the second part, we consider the gradient flow of the functional in the \( 2 \)-Wasserstein space and prove statements about its asymptotic behavior for large times, for which we employ the pseudo-inverse technique for probability measures in 1D. Depending on the parameter range, this includes existence of a subsequence converging to a steady state or even convergence of the whole trajectory to a limit which we can specify explicitly. For both parts of the work, a key ingredient is the generalized Fourier transform, which allows us to verify the conditional positive definiteness of the interaction kernel for coinciding attractive and repulsive exponents.

\clearpage
\selectlanguage{english}
\section*{Acknowledgments}
I would like to thank all the people who supported me in writing this thesis. In particular, I highly appreciate all the time and effort Massimo Fornasier and Daniel Matthes put into advising, teaching and encouraging me. Moreover, I am very thankful to Marco DiFrancesco and José Antonio Carrillo for the inspiring discussions we had and for kindly hosting me in Bath and London, respectively. Also, I thank the \textsc{start} project ``Sparse Approximation and Optimization in High Dimensions'' for its financial support.

Finally, I'm much obliged to my family, especially my mother, Ingeborg Egel-Hütter, and all my friends, including Bernhard Werner, Felix Rötting, Jens Wolter, Thomas Höfelsauer, Vroni Auer and Wahid Khosrawi-Sardroudi, for their encouragement, company, patience and the inspiration they gave to me.

\clearpage{}
\addtocontents{toc}{\protect\markboth{}{}}
\tableofcontents
\cleardoublepage

\pagenumbering{arabic}
\clearpage{}\makeatletter{}
\chapter{Introduction}
\label{cha:introduction}

\section{Problem statement and related work}
\label{sec:problem-statement}

\begin{figure}[t]
  \centering
  \subfigure[Original image]{\includegraphics[width=6cm]{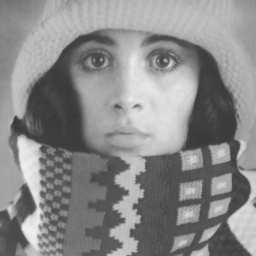}}
  \hspace{0.5cm}
  \subfigure[Dithered image]{\includegraphics[width=6cm]{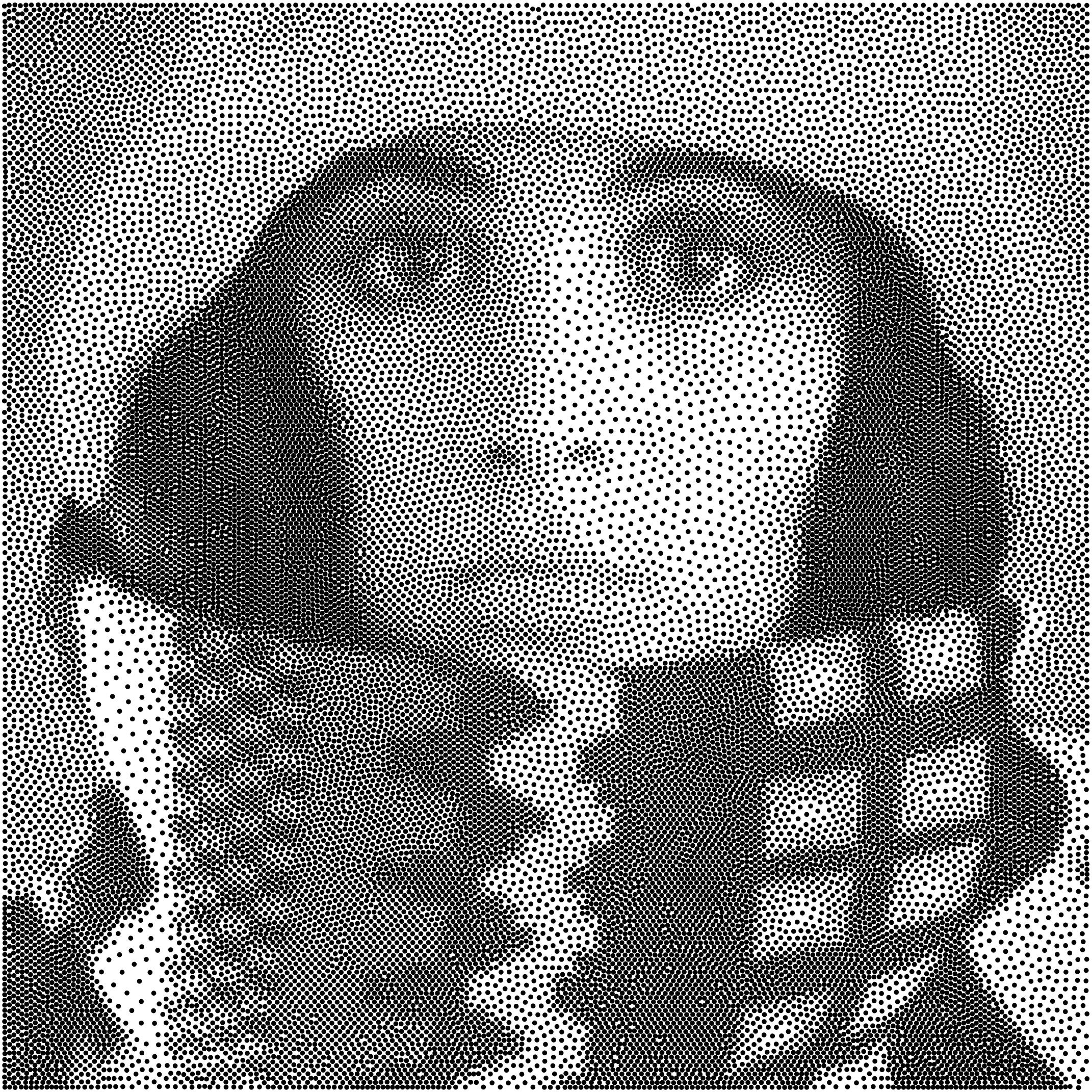}}
  \caption{Undithered and dithered image}
  \label{fig:figure2}
\end{figure}

In \cite{FHS12}, the authors proposed to use an attraction-repulsion functional to measure the quality of a point-approximation to a picture: If we interpret a black-and-white picture as a probability measure \( \omega \) on a compact set \( \Omega \), we are looking for points \( x_1,\ldots,x_N \in \mathbb{R}^2 \in\mathbb{N} \) such that the corresponding point measure \( \frac{1}{N} \sum_{i = 1}^{N} \delta_{x_i} \) approximates \( \omega \) well. While there are many ways to determine the proximity of those two probability measures (for a brief summary over some important ones, see \cite{07_Carillo_Toscani_prob-metrics}), the interesting idea in \cite{FHS12} consists of employing kinetic principles for that purpose. Namely, we consider the points \( x=(x_i)_{i=1,\ldots,N} \) to be attracted by the picture by introducing an attraction potential
\begin{equation}
  \label{eq:5}
  \mathcal{V}_N(x) := \frac{1}{N} \sum_{i=1}^{N} \int_{\mathbb{R}^2} \left| x_i - y \right| \diff \omega(y)
\end{equation}
which is to be minimized. If left as it is, this will most certainly not suffice to force the points into an intuitively good position: The
minimizer would consist of all the points being at the median of \( \omega \). Hence, we would like to enforce a spread of the points by adding a pairwise repulsion term
\begin{equation}
  \label{eq:6}
  \mathcal{W}_N(x) := - \frac{1}{2N^2} \sum_{i,j=1}^{N} \left| x_i-x_j \right|,
\end{equation}
leading to the minimization of the composed functional
\begin{equation}
  \label{eq:7}
  \mathcal{E}_N(x) := \mathcal{V}_N(x) + \mathcal{W}_N(x),
\end{equation}
which produces visually appealing results as in Figure \ref{fig:figure2} (taken from \cite{FHS12}).

Furthermore, an attraction-repulsion functional like this one admits more than one interpretation: one could also consider the particles as a population which is attracted to a food source \( \omega \), modeled by \( \mathcal{V} \), while at the same time being repulsed by competition with each other, modeled by \( \mathcal{W} \). 

Now, if we write \( \mu_x = \frac{1}{N} \sum_{i=1}^{N} \delta_{x_i}  \) instead of \( x \), we see that the above functionals can be expressed independently of \( N \), 
\begin{equation}
  \label{eq:8}
  \mathcal{E}[\mu_x] = \int_{\mathbb{R}^d\times\mathbb{R}^d} \left| x-y \right| \diff \omega(x) \diff \mu_x(y) - \frac{1}{2}\int_{\mathbb{R}^d\times\mathbb{R}^d} \left| x-y \right|  \diff \mu_x(x) \diff \mu_x(y).
\end{equation}
Generalizing further, in the following we also want to consider a different domain than \( \mathbb{R}^2 \) and a slightly larger class of interaction kernels than \( x \mapsto \left| x \right| \), as well as possibly allowing different kernels for attraction and repulsion. So, if we write
\begin{itemize}
\item \( \Omega \subseteq \mathbb{R}^d \) with \( d \in \mathbb{N} \) for the \emph{domain},
\item \( \psi_a \colon \mathbb{R}^d \rightarrow \mathbb{R} \) for the (radially symmetric) \emph{attraction kernel},
\item \( \psi_r \colon \mathbb{R}^d \rightarrow \mathbb{R} \) for the (also radially symmetric) \emph{repulsion kernel},
\item \( \omega \in \mathcal{P}(\Omega) \) for the \emph{datum}, where \( \mathcal{P}(\Omega) \) denotes the set of probability measures on \( \Omega \),
\end{itemize}
the functional of interest becomes
\begin{equation}
  \label{eq:9}
  \mathcal{E}[\mu] := \int_{\Omega} \psi_a(x-y) \diff \omega(x) \diff \mu(y) - \frac{1}{2} \int_{\Omega} \psi_r(x-y)  \diff \mu(x) \diff \mu(y).
\end{equation}
Additionally, we will shall consider a regularization of \( \mathcal{E} \) by a total variation term,
\begin{equation}
  \label{eq:355}
  \mathcal{E}^\lambda[\mu] := \mathcal{E}[\mu] + \lambda \left| D\mu \right|(\Omega),
\end{equation}
where \( \lambda > 0 \) and \( \mu \) is assumed to be in \( L^1(\Omega) \) and to have a distributional derivative \( D\mu \) which is a finite Radon measure with total variation \( \left| D\mu \right| \).
\\

\noindent Variational functionals like the one above, being composed of a quadratic and a linear integral term, arise in many models in biology, physics and mathematics as the limit of particle models. In particular the quadratic term, in our case denoted by \( \mathcal{W} \), corresponding to the self-interaction between particles, is of great interest in modeling physical or biological phenomena, see for example \cite{03-Camazine-Deneubourg-Franks-self-organization, 09-Couzin-effective-leadership, 08-Li-school-formation, 09-Vicsek-novel-type-of-phase-transition}.

The range of mathematical questions when investigating such models is diverse: Firstly, one can study the continuous functional to find conditions for the existence of (local or global) minimizers and afterwards determine some of their properties. Examples for this are the so called non-local isoperimetric problem studied in \cite{2012-Muratov} and \cite{13-Cicalese-droplet-minimizers}, where a total variation term as in \eqref{eq:355} not only appears but is in fact critical for the model, and the non-local Ginzburg-Landau energies for diblock polymer systems as in \cite{12-Serfaty-DropletI,12-Serfaty-DropletII}.

Secondly, one can consider the associated gradient flow of the energy functional, where some of the arising problems are its well-definedness, its asymptotic behavior for large times (\eg convergence to a steady state or pattern formation) and the relationship between the gradient flow of a particle approximation and the gradient flow of the limit functional, called the mean-field limit. One major breakthrough in the development of the theory of gradient flows in Wasserstein spaces was \cite{98-JKO} and a recent and thorough treatment of it can be found in \cite{AGS08}. For an introduction in particular to the mean field limit, see \cite{13-Carrillo-Choi-Hauray-MFL} and the references therein. Additionally, we refer to Section \ref{sec:intro} for a more in-depth review of results which are connected to the gradient flow of the functional in question.

With respect to our particular problem and the static setting, see \cite{testgwscwe11} for efficient optimization algorithms to find local minima of \( \mathcal{E} \) and \cite{grpost11} for the relationship of minimizers of \( \mathcal{E} \) and the error of quadrature formulas, also highlighting the connection between those minimizers and the problem of optimal quantization of measures (see \cite{00_Graf_Luschgy_quantization, 04-Gruber-Quantization}). As for the gradient flow, see for example \cite{BCLR} for the analysis of symmetric steady states for the gradient flow of interaction functionals similar to \( \mathcal{W} \), but being composed of the sum of an attractive and a repulsive power function.
\\

\noindent In the scope of this work, we shall limit our attention to the special case of power kernels,
\begin{equation}
  \label{eq:10}
  \psi_a(x) := \left| x \right|^{q_a}, \quad \psi_r(x) := \left| x \right|^{q_r}, \quad x \in \mathbb{R}^d,
\end{equation}
with \( q_a,q_r \in [1,2] \). The topics we would like to address are:

\begin{itemize}
\item Conditions for the well-definedness of the above expression (Section \ref{sec:prel-obs} and later on leading to the theory developed in Section \ref{sec:prop-funct-mathbbrd})
\item Existence/non-existence of minimizers (Section \ref{sec:prel-obs})
\item Convergence of minimizers of the functionals \( \mathcal{E}_N \) to minimizers of \( \mathcal{E} \) (Section \ref{sec:part-appr})
\item Compactness properties of the sub-levels of \( \mathcal{E} \) (Section \ref{sec:moment-bound-symm})
\item Regularization of the functional with an additional total variation term and a discretization in terms of point masses (Section \ref{sec:tv-reg} for the theory, Section \ref{sec:numer-exper} for numerical experiments)
\item Existence and asymptotic behavior of the associated gradient flow of \( \mathcal{E} \) in the space of probability measures with existing second moments, \( \mathcal{P}_2(\mathbb{R}) \), endowed with the \( 2 \)-Wasserstein metric (Section \ref{cha:gradient-flow-1d})
\end{itemize}

To our knowledge, all the results contained in this thesis, except for a few ones recalled from other sources, are original. Additionally, the mathematical tools used to prove them are diverse, such as variational calculus in spaces of probability measures (including in particular \( \Gamma \)-convergence, BV-functions and compactness arguments), harmonic analysis with generalized Fourier-transforms in Section \ref{cha:vari-prop-symm}, as well as fixed point arguments and the pseudo-inverse technique for gradient flows in Wasserstein spaces in Section \ref{cha:gradient-flow-1d}.

\section{Overview of the chapters}
\label{sec:organisation-thesis}

\subsection{Variational properties}
\label{sec:vari-prop}

In Section \ref{sec:prel-obs}, we start with a few theoretical preliminaries, followed by examples for and counterexamples to the existence of minimizers for \( \mathcal{E} \) in the case of power potentials, depending on the powers and the domain \( \Omega \), where elementary estimates for the behavior of the power functions are used in conjunction with appropriate notions of compactness for probability measures, \ie, uniform integrability of moments and moment bounds.

Beginning from Section \ref{sec:prop-funct-mathbbrd}, we study limiting case of coinciding powers for attraction and repulsion, where there is no longer an obvious confinement property given by the attraction term. To regain compactness and lower semi-continuity, we pass to the lower semi-continuous envelope of our functional, which can be proven to coincide with a Fourier representation of the functional \( \mathcal{E} \), see Corollary \ref{cor:lower-semi-cont} in Section \ref{sec:extension-p}, which is at first derived on \( \mathcal{P}_2(\mathbb{R}^d) \) in Section \ref{sec:fourier-formula-p2}. The main ingredient to find this representation is the generalized Fourier transform in the context of the theory of conditionally positive definite functions, which we briefly recapitulated in Appendix \ref{cha:cond-posit-semi}.

Having thus established a problem which is well-posed for our purposes, we proceed to prove one of our main results, namely the convergence of the minimizers of \( \mathcal{E}_N \) to \( \omega \), Theorem \ref{thm:cons-part-appr} in Section \ref{sec:part-appr}. This convergence will follow in a standard way from the \( \Gamma \)-convergence of the corresponding functionals. Furthermore, again applying the techniques of Appendix \ref{cha:cond-posit-semi} used to prove the Fourier representation, this allows us to derive a compactness property for the sublevels of \( \mathcal{E} \) in terms of a uniform moment bound in Section \ref{sec:moment-bound-symm}.

Afterwards, in Section \ref{sec:tv-reg}, we shall introduce the total variation regularization of \( \mathcal{E} \). Firstly, we prove consistency in terms of \( \Gamma \)-convergence for vanishing regularization parameter in Section \ref{sec:cons-regul-cont}. Then, in Section \ref{sec:discrete-version-tv-1}, we propose two ways of computing a version of it on the particle level and again prove consistency for \( N \to \infty \). One version consists of employing kernel density estimators, while in the other one each point mass is replaced by an indicator function extending up to the next point mass for the purpose of computing the total variation. In Section \ref{sec:numer-exper}, we illustrate the first approach by numerical experiments.

\subsection{Gradient flow in 1D}
\label{sec:vari-prop-1}

We begin with a more thorough summary of previously known results and connections to other works in Section \ref{sec:intro} and a reminder about the pseudo-inverse technique for Wasserstein gradient flows in Section \ref{sec:pseudo-inverse-techn}.

Section \ref{sec:existence-theory-linfty} then contains an global existence result for such a gradient flow associated to \( \mathcal{E} \) in the space \( L^\infty_c(\mathbb{R}) \), whose proof is based on a fixed point argument in the spirit of the Picard-Lindelöf theorem.

In Section \ref{sec:asymptotics}, we investigate some combinations of the parameters \( q_a \) and \( q_r \) for which we are able to prove statements about the asymptotic behavior. For \( q_a = q_r = 2 \), the solutions exhibit a traveling wave behavior, which can be seen elemetarily (Section \ref{sec:case-q-2}). For \( q_r = 1 \leq q_a \leq 2 \), the gradient flow converges to a convolution of \( \omega \) (Section \ref{sec:case-q-1}), which follows by the special structure of the repulsion term in this case together with the monotonicity of the attraction field.

Finally, in Section \ref{sec:conv-against-steady}, we show the existence of a convergent subsequence for large times in some of the remaining parameter range, namely the two cases \( 1 \leq  q_r = q_a < 4/3 \) and \( 1 \leq q_r < q_a \leq 2 \). For this, we use an energy-energy-dissipation inequality and draw on the compactness given by the moment bound proven in Section \ref{sec:moment-bound-symm}.

\clearpage{}
\clearpage{}\makeatletter{}\chapter{Variational properties}
\label{cha:vari-prop-symm}

In this section, we want to prove certain variational properties of the functional \( \mathcal{E} \) in order to prove consistency of particle approximations to it and its regularization by a total variation term.

We recall its definition:
\begin{equation}
  \label{eq:11}
    \mathcal{E}[\mu] := \int_{\Omega\times\Omega} \psi_a(x-y) \diff \omega(x) \diff \mu(y) - \frac{1}{2} \int_{\Omega\times\Omega} \psi_r(x-y)  \diff \mu(x) \diff \mu(y),
\end{equation}
for \( \omega \), \( \mu \in \mathcal{P}_2(\mathbb{R}^d) \) (at least for now) and
\begin{equation}
  \label{eq:12}
    \psi_a(x) := \left| x \right|^{q_a}, \quad \psi_r(x) := \left| x \right|^{q_r}, \quad x \in \mathbb{R}^d,
\end{equation}
with \( q_a \), \( q_r \in [1,2] \). Furthermore, denote for a vector-valued measure \( \nu \) its \emph{total variation} (which is a positive measure) by \( \left| \nu \right| \) and by \( BV(\mathbb{R}^d) \) the subset of distributions \( f \in L^1_{loc}(\mathbb{R}^d) \) whose distributional derivatives \( Df \) are finite Radon measures (see \cite[Definition 3.1]{00-Amb-Fusco-Pallara-BV}). Abusing terminology, we call \( \left| Df \right|(\Omega) \) the total variation of \( f \). Now, we define the \emph{total variation regularization} of \( \mathcal{E} \) by
\begin{equation}
  \label{eq:13}
  \mathcal{E}^\lambda[\mu] := \mathcal{E}[\mu] + \lambda \left| D\mu \right|(\Omega),
\end{equation}
where \( \mu \in \mathcal{P}_2(\mathbb{R}^d) \cap BV(\mathbb{R}^d) \).

\section{Preliminary observations}
\label{sec:prel-obs}

We shall briefly state some results which are in particular related to the asymmetric case of \( q_a \) and \( q_r \)  not necessarily being equal. 

\subsection{Narrow convergence and Wasserstein-convergence}
\label{sec:conv-p-pp}

We want to begin with a brief summary of measure theoretical results which will be needed in the following.

The first simple lemma is useful when switching the point of view and therefore also the involved topology from density functions to probability measures. For a brief introduction to the narrow topology, see \cite[Chapter 5.1]{AGS08}.

\begin{lemma}[\(L^{1}\)-convergence implies narrow convergence]
  \label{lem:1}
  Let \(\Omega \subseteq \mathbb{R}^{d}\) and \(f_{n} \in L^{1}(\Omega)\) be a sequence which converges to \(f \in L^{1}\) in \(L^{1}\). Then, \( f_{n} \rightarrow f \text{ narrowly}. \)
\end{lemma}

\begin{proof}
  Let \(g \in C_{b}(\Omega)\). Then,
  \begin{equation}
    \label{eq:14}
    \left| \int_{\Omega} g(x)\left( f_{n}(x) - f(x) \right) \diff x \right| \leq \left\| g \right\|_{\infty} \int_{\Omega} \left| f_{n}(x) - f(x) \right| \diff x \rightarrow 0.\qedhere
  \end{equation}
\end{proof}

On a complete metric space \( X \), the narrow topology can be characterized by countably many functions and if \( X \) is separable, it is compatible with building product measures.

\begin{lemma}[Metrizability of narrow convergence]
  \label{lem:25}
  \cite[Remark 5.1.1]{AGS08}
  There is a sequence of continuous functions \( (f_k)_{k \in \mathbb{N}} \) on \( \mathbb{R} \) with \( \sup_{x \in \mathbb{R}^d} \left| f_k(x) \right| \leq 1 \) such that the narrow convergence in \( \mathcal{P}(\mathbb{R}^d) \) can be metrized by 
  \begin{equation}
    \label{eq:318}
    \delta(\mu,\nu) := \sum_{k = 1}^{\infty} 2^{-k} \left| \int_{\mathbb{R}^d} f_k(x) \diff \mu(x) - \int_{\mathbb{R}^d} f_k(x) \diff \nu(x) \right|.
  \end{equation}
\end{lemma}

\begin{lemma}[Convergence of product measures]
  \label{lem:27}
  Let \( \Omega \subseteq \mathbb{R}^d \). Since \( \Omega \) is separable, from \cite[Theorem 2.8]{68-Billingsley-Conv-Proba} it follows that if \( \left( \mu_n \right)_n \), \( \left( \nu_n \right)_n \) are two sequences in \( \mathcal{P}(\Omega) \) and \( \mu, \nu \in \mathcal{P}(\Omega) \), then
  \begin{equation}
    \label{eq:35}
    \mu_n \times \nu_n \rightarrow \mu \times \nu \text{ narrowly} \Leftrightarrow \mu_n \rightarrow \mu \text{ and } \nu_n \rightarrow \nu \text{ narrowly}.
  \end{equation}
\end{lemma}

Finally, we include some results about the continuity of integral functionals with respect to Wasserstein-convergence.

\begin{definition}[Wasserstein distance]
  \label{def:wasserstein-distance}
  \cite[Chapter 7.1]{AGS08}
  Let \( \Omega \subseteq \mathbb{R}^d \), \( p \in [1,\infty) \) as well as \( \mu_1, \mu_2 \in \mathcal{P}_p(\Omega) \) be two probability measures with finite \( p \)th moment. Denoting by \( \Gamma(\mu_1, \mu_2) \) the probability measures on \( \Omega \times \Omega \) with marginals \( \mu_1 \) and \( \mu_2 \), then we define
  \begin{equation}
    \label{eq:322}
    W_p^p(\mu_1, \mu_2) := \min \left\{ \int_{\Omega^2} \left| x_1 - x_2 \right|^p \diff \bm{\mu}(x_1, x_2) : \bm{\mu} \in \Gamma(\mu_1, \mu_2) \right\},
  \end{equation}
  the \emph{Wasserstein}-\( p \) distance between \( \mu_1 \) and \( \mu_2 \).

  Additionally, by Hölder's inequality, \( W_p(\mu_1,\mu_2) \) is non-increasing in \( p \) and therefore we can define
  \begin{equation}
    \label{eq:371}
    W_\infty(\mu_1,\mu_2) := \lim_{p \to \infty} W_p(\mu_1,\mu_2).
  \end{equation}
\end{definition}

\begin{definition}[Uniform integrability]
  On a measurable space \( X \), a measurable function \( f : X \rightarrow [0,\infty] \) is \emph{uniformly integrable} \wrt a family of measures \( \left\{ \mu_i \right\}_{i \in I} \), if
  \begin{equation}
    \label{eq:46}
    \lim_{M \rightarrow \infty} \sup_{i \in I} \int_{\left\{ f(x) \geq M\right\}} f(x) \diff \mu_i(x) = 0.
  \end{equation}
\end{definition}

\begin{lemma}[Topology of Wasserstein spaces]
  \label{lem:26}
  \cite[Proposition 7.1.5]{AGS08}
  For \( p  \geq 1\) and a subset \( \Omega \subseteq \mathbb{R}^n \), \( \mathcal{P}_p (\Omega) \) endowed with the Wasserstein-\( p \) distance is a separable metric space which is complete if \( \Omega \) is closed. A set \(  \mathcal{K} \subseteq \mathcal{P}_p(\Omega) \) is relatively compact iff it is \( p \)-uniformly integrable (and hence by Lemma \ref{lem:24} tight). In particular, for a sequence \( (\mu_n)_{n \in \mathbb{N}} \subseteq \mathcal{P}_p(\Omega) \),
  \begin{equation}
    \label{eq:320}
    \lim_{n\rightarrow\infty} W_p(\mu_n, \mu) = 0 \Leftrightarrow
    \begin{cases}
      \mu_n \rightarrow \mu \text{ narrowly},\\
      (\mu_n)_n \text{ has uniformly integrable } p \text{-moments}. 
    \end{cases}
  \end{equation}
\end{lemma}

\begin{lemma}[Continuity of integral functionals]
  \label{lem:4}
  \cite[Lemma 5.1.7]{AGS08}
  Let \( \mu_n \in \mathcal{P}(\mathbb{R}^d) \) a sequence converging narrowly to \( \mu \in \mathcal{P}(\mathbb{R}^d) \), \( g : \mathbb{R}^d \rightarrow \mathbb{R} \) lower semi-continuous and \( f : \mathbb{R}^d \rightarrow \mathbb{R} \) continuous. If \( \left| f \right|, g^- := -\min \left\{ g, 0 \right\} \) are uniformly integrable \wrt \( \left\{ \mu_{n} \right\}_n \), then
  \begin{align}
    \label{eq:52}
    \liminf_{n\rightarrow\infty} \int_{\mathbb{R}^d} g(x) \diff \mu_n(x) \geq {} & \int_{\mathbb{R}^d} g(x) \diff \mu(x)\\
    \lim_{n\rightarrow\infty} \int_{\mathbb{R}^d} f(x) \diff \mu_n(x) = {} & \int_{\mathbb{R}^d} f(x) \diff \mu(x)
  \end{align}
\end{lemma}

\begin{lemma}[Uniform integrability of moments]
  \label{lem:24}
  \cite[Corollary to Theorem 25.12]{95-Billingsley-Proba_and_Measuer}
  Given \( r > 0 \) and a family \( \left\{ \mu_i \right\}_{i \in I} \) of probability measures with
  \begin{equation}
    \label{eq:316}
    \sup_{i \in I} \int_{\mathbb{R}^d} \left| x \right|^r \diff \mu_i(x) < \infty,
  \end{equation}
  then the family \( \left\{ \mu_{i} \right\}_i \) is tight and for all \( 0 < q < r \), \( x \mapsto \left| x \right|^q \) is uniformly integrable \wrt \( \left\{ \mu_i \right\}_{i \in I} \). 
\end{lemma}
\begin{proof}
  For the uniform integrability, let \( M > 0 \). By the monotonicity of the power functions \( t \mapsto t^p \) for \( t > 0 \) and \( p > 0 \), we have
  \begin{align}
    \label{eq:337}
    \int_{\left\{ \left| x \right|^q \geq M \right\}} \left| x \right|^q \diff \mu_i = {} & \int_{\left\{ \left| x \right|^q \geq M \right\}} \left| x \right|^q \frac{M^{(r-q)/q}}{M^{(r-q)/q}} \diff \mu_i \\
    \leq {} & M^{-(r-q)/q} \int_{\left\{ \left| x \right|^q \geq M \right\}} \left| x \right|^r \diff \mu_i \\
    \leq {} & M^{-(r-q)/q} \int_{\mathbb{R}^d} \left| x \right|^r \diff \mu_i \rightarrow 0,
  \end{align}
  for \( M \rightarrow \infty \), uniformly in \( i \in I \).

  Similarly, for the tightness, 
  \begin{align}
    \label{eq:326}
    \mu_i\left( \left\{ \left| x \right| \geq M \right\} \right) \leq M^{-r} \int_{\mathbb{R}^d} \left| x \right|^r \diff \mu_i(x) \rightarrow 0
  \end{align}
  for \( M \rightarrow \infty \).
\end{proof}

\subsection{Situation on a compact set}
\label{sec:situ-comp-set}

From now on, let \( q_a, q_r \in [1,2] \).

\begin{proposition}
  Let \(\Omega \subset \mathcal{\mathbb{R}}^{d}\) be a compact subset in \(\mathbb{R}^{d}\). Then, the functionals \(\mathcal{E}\) and \(\mathcal{E}^{\lambda}\) are well-defined on \(\mathcal{P}(\Omega)\) and  \(\mathcal{P}(\Omega) \cap BV(\Omega)\), respectively, and \(\mathcal{E}\) admits a minimizer.

  If additionally \(\Omega\) has a Lipschitz boundary, \(\mathcal{E}^{\lambda}\) admits a minimizer as well.
\end{proposition}

\begin{proof}
  Note that since the mapping
  \begin{equation}
    \label{eq:15}
    (x,y) \mapsto \left| y - x \right|^{q},\quad x,y \in \mathbb{R}^{d},
  \end{equation}
  is jointly continuous in \(x\) and \(y\), it attains its maximum on the compact set \(\Omega \times \Omega\). Hence, the kernel \eqref{eq:15} is a bounded continuous function, which on the one hand implies that the functional \(\mathcal{E}\) is bounded (and in particular well-defined) on \( L^1(\Omega) \) and on the other hand that it is continuous with respect to the narrow topology. Together with the compactness of \( \mathcal{P}(\Omega) \), this implies we can employ the direct method of the calculus of variations to find a minimiser for \(\mathcal{E}\).

  The situation for \(\mathcal{E}^{\lambda}\) is similar: Due to the boundedness of \(\Omega\) and the regularity of its boundary, sub-levels of \( \left| D\,.\, \right|(\Omega) \) are relatively compact in \( L^{1}(\Omega) \cap \mathcal{P}(\Omega) \) by \cite[Chapter 5.2, Theorem 4]{92-Evans-fine-properties}. As the total variation is lower semi-continuous with respect to \(L^{1}\)-convergence by \cite[Chapter 5.2, Theorem 1]{92-Evans-fine-properties} and \(L^{1}\)-convergence implies narrow convergence by Lemma \ref{lem:1}, we get lower semi-continuity of \( \mathcal{E}^\lambda \) and therefore again existence of a minimizer.
\end{proof}

\subsection{Existence of minimizers for stronger attraction on arbitrary domains}
\label{sec:exist-minim-strong}

Note that from here on, the constants \( C \) and \( c \) are generic and may change in each line of a calculation.

\begin{lemma}
  \label{lem:2}
  For \( q \geq 1 \) and \( x,y \in \mathbb{R}^d \), there exist \( C,c > 0 \) such that
  \begin{equation}
    \label{eq:16}
    \left| x + y \right|^q \leq C \left( \left| x \right|^q + \left| y \right|^q \right).
  \end{equation}
  and
  \begin{equation}
    \label{eq:17}
    \left| x - y \right|^q \geq c \left| x \right|^q - \left| y \right|^q 
  \end{equation}
\end{lemma}

\begin{proof}
  By the monotonicity of the power function \( x \mapsto x^q \) for \( x \in [0,\infty) \) and the triangle inequality, we can deduce
  \begin{equation}
    \label{eq:18}
    \left| x + y \right|^q \leq \left( \left| x \right| + \left| y \right| \right)^q
  \end{equation}
  for \( x, y \in \mathbb{R} \). By the convexity of \( x \mapsto x^q \) for \( q \in [1,\infty) \) and \( x \in [0,\infty) \), we see that
  \begin{align}
    \label{eq:19}
    \left( \left| x \right| + \left| y \right| \right)^q = {} & \left( \frac{1}{2} \cdot 2 \left| x \right| + \frac{1}{2} \cdot 2 \left| y \right| \right)^q\\
    \leq {} & \frac{1}{2} \left( 2 \left| x \right| \right)^q + \frac{1}{2} \left( 2 \left| y \right| \right)^q,
  \end{align}
  yielding estimate \eqref{eq:16} with \( C := 2^{q-1} \).

  Now, using estimate \eqref{eq:16} on \(  x = x - y + y \) yields
  \begin{align}
    \label{eq:20}
    \left| x \right|^q \leq {} & C \left( \left| x - y \right|^q + \left| y \right|^q \right);
  \end{align}
  implying \eqref{eq:17} with \( c := C^{-1} \).
\end{proof}

\begin{theorem}
  \label{thm:exist-min-strong}
  Let \( q_a, q_r \in [1,2] \), \( \Omega \subseteq \mathbb{R}^d \) closed and \( q_a > q_r \). If \( \omega \in \mathcal{P}_{q_a}(\Omega) \), then the sub-levels of \( \mathcal{E} \) have uniformly bounded \( q_a \)th moments and \( \mathcal{E} \) admits a minimizer on \( \mathcal{P}_{q_r}(\Omega) \).
\end{theorem}

\begin{proof}
  We can show that the sub-levels of \( \mathcal{E} \) have a uniformly bounded \( q_a \)th moment, so that they are Wasserstein-\( q \) compact for any \( q < q_a \) by Lemma \ref{lem:26} and Lemma \ref{lem:24}, which means that we can extract a narrowly converging subsequence \( \left( \mu_n \right)_n \) with uniformly integrable \( q_r \)th moments. With respect to that convergence (which by Lemma \ref{lem:27} also implies the narrow convergence of \( \left( \mu_n \otimes \mu_n \right)_n \) and \( \left( \mu_n \otimes \omega \right)_n \)), the functional \( \mathcal{W} \) is continuous and the functional \( \mathcal{V} \) is lower semi-continuous by Lemma \ref{lem:4}, so we shall be able to apply the direct method of the calculus of variations to show existence of a minimizer in \( \mathcal{P}_{q_r}(\Omega) \).

  \emph{Ad moment bound:} Let \( \mu \in \mathcal{P}_{q_r}(\Omega) \). By estimate \eqref{eq:17} of Lemma \ref{lem:2}, we have
  \begin{align}
    \label{eq:22}
    \mathcal{V}[\mu] = {} & \int_{\Omega \times \Omega} \left| x - y \right|^{q_a} \diff \mu(x) \diff \omega(y)\\
    \geq {} & \int_{\Omega \times \Omega} \left( c \left| x \right|^{q_a} - \left| y \right|^{q_a} \right) \diff \mu(x) \diff \omega(x)\\
    = {} & c \int_{\Omega} \left| x \right|^{q_a} \diff \mu(x) - \int_{\Omega} \left| y \right|^{q_a} \diff \omega(y).\label{eq:23}
  \end{align}
  On the other hand, by estimate \eqref{eq:16}
  \begin{align}
    \label{eq:24}
    \mathcal{W}[\mu] = {} & - \frac{1}{2} \int_{\Omega \times \Omega} \left| x - y \right|^{q_r} \diff \mu(x) \diff \mu(y)\\
    \geq {} & - C \int_{\Omega \times \Omega} \left( \left| x \right|^{q_r} + \left| y \right|^{q_r} \right) \diff \mu(x) \diff \mu(y)\\
    \geq {} & - C \int_{\Omega} \left| x \right|^{q_r} \diff \mu(x)\label{eq:25}
  \end{align}
  Combining \eqref{eq:23} and \eqref{eq:25}, we have
  \begin{align}
    \label{eq:26}
    \mathcal{E}[\mu] + \int_{\Omega} \left| x \right|^{q_a} \diff \omega(x) \geq {} &\int_{\Omega} \left( c\left| x \right|^{q_a} - C \left| x \right|^{q_r} \right) \diff \mu(x)\\
    \geq {} &\int_{\Omega} \left( c - C \left| x \right|^{q_r - q_a} \right) \left| x \right|^{q_a} \diff \mu(x)
  \end{align}
  Since \( q_a > q_r \), there is an \( M > 0 \) such that
  \begin{equation}
    \label{eq:27}
    c - C \left| x \right|^{q_r - q_a} \geq \frac{c}{2}, \quad \left| x \right| \geq M,
  \end{equation}
  and hence
  \begin{align}
    \label{eq:28}
    \int_{\Omega} \left| x \right|^{q_a} \diff \mu(x) = {} & \int_{B_M(0)} \left| x \right|^{q_a} \diff \mu(x) + \int_{\Omega \setminus B_M(0)} \left| x \right|^{q_a} \diff \mu(x) \\
    \leq {} & M^{q_a} + \frac{2}{c} \left[ \mathcal{E}[\mu] + \int_{\Omega} \left| x \right|^{q_a} \diff \omega(x) \right] \label{eq:358}
  \end{align}
\end{proof}

\subsection{Counterexample to the existence of minimizers for stronger repulsion}
\label{sec:absence-minim-strong}

Now, let \(q_{a},q_{r} \in [1,2]\) with \( q_r > q_a \). On \(\Omega=\mathbb{R}^{d}\), this problem need not have a minimizer.

\begin{example}[Absence of minimizers for stronger repulsion]
  \label{exa:count-exist-minim}
  Let \(\Omega=\mathbb{R}\), \(q_{r} > q_{a}\), \(\omega = 1_{[-1,0]}\mathcal{L}^{1}\) and
  consider the sequence \(\mu_{n} := n^{-1}1_{[0,n]}\mathcal{L}^{1}\). Computing
  the values of the functionals used to define \(\mathcal{E}\) and
  \(\mathcal{E}^{\lambda}\) yields
  \begin{align}
    \mathcal{V}[\mu_{n}] = {} &\frac{1}{n} \int_{-1}^{0} \int_{0}^{n} \left| y -
      x \right|^{q_{a}} \diff x \diff y\\
    \leq {} &\frac{1}{n} \int_{0}^{n} (y+1)^{q_{a}} \diff y\\
    = {} &\frac{1}{n (q_{a}+1)} \left(n+1\right)^{q_{a}+1} - \frac{1}{n (q_{a}+1)}\\
    \leq {} &\frac{\left(n+1\right)^{q_{a}}}{q_{a}+1};\\
    \label{eq:29}
    \mathcal{W}[\mu_{n}] = {} & - \frac{1}{2n^{2}} \int_{0}^{n}
    \int_{0}^{n}\left| y-x \right|^{q_{r}} \diff x \diff y \\
    = {} & - \frac{1}{2n^{2}(q_{r}+1)} \int_{0}^{n} \left[(n - y)^{q_{r}+1} +
      y^{q_{r}+1}\right] \diff y\\
    = {} & - \frac{1}{2n^{2}(q_{r}+1)(q_{r}+2)} 2n^{q_{r}+2} =
    \frac{n^{q_{r}}}{(q_{r}+1)(q_{r}+2)};\\
    \left\| D\mu_{n} \right\| = {} &- \frac{2}{n}.
  \end{align}
  Taken together, we see that
  \begin{equation}
    \label{eq:30}
    \mathcal{E}[\mu_{n}]\rightarrow -\infty, \quad \mathcal{E}^{\lambda}[\mu_{n}]\rightarrow -\infty \quad \text{for } n\rightarrow \infty,
  \end{equation}
  which means there are no minimizers in this case.
\end{example}

\section{Properties of the functional on $\mathbb{R}^d$}
\label{sec:prop-funct-mathbbrd}

Now, let us consider \(\Omega=\mathbb{R}^{d}\) and 
\begin{equation}
  \label{eq:317}
  q := q_{a} = q_{r}, \quad \psi(x) := \psi_a(x) = \psi_r(x) = \left| x \right|^q, \quad x \in \mathbb{R}^d,
\end{equation}
for \( 1 \leq q < 2 \).

Here, neither the well-definedness of \(\mathcal{E}[\mu]\) for all \(\mu\in\mathcal{P}(\mathbb{R}^{d})\) nor the narrow compactness of the sub-levels as in the case of a compact \( \Omega \) in Section \ref{sec:situ-comp-set} are clear, necessitating additional conditions on \( \mu \) and \( \omega \). For example, if we assume the existence of the second moments, i.e.\@ \(\mu,\omega \in \mathcal{P}_{2}(\mathbb{R}^{d})\), the space of probability measures with finite second moment, we can a priori see that both \(\mathcal{V}[\mu]\) and \(\mathcal{W}[\mu]\) are finite.

Under this restriction, we can show a formula for \(\mathcal{E}\) involving the Fourier-Stieltjes transform of the measures \(\mu\) and \(\omega\). Namely, there is a constant \(C = C(q,\omega) \in \mathbb{R}\) such that
\begin{equation}
  \label{eq:31}
  \mathcal{E}[\mu] + C = -2^{-1} (2\pi)^{-d}\int_{\mathbb{R}^d}\left| 
    \widehat{\mu}(\xi) - \widehat{\omega}(\xi) \right|^2
  \widehat{\psi}(\xi) \dd \xi =: \widehat{\mathcal{E}}[\mu],
\end{equation}
where for \(\mu \in \mathcal{P}(\mathbb{R}^{d})\), \(\widehat{\mu}\) denotes its Fourier-Stieltjes transform,
\begin{equation}
  \label{eq:32}
  \widehat{\mu}(\xi) = \int_{\mathbb{R}^{d}} \exp(-\i x^{T}\xi) \dd \mu(x),
\end{equation}
and \(\widehat{\psi}\) is the \emph{generalized Fourier-transform of \(\psi\)}, \ie a Fourier transform with respect to a certain duality. We have gathered most of the important facts about it in Appendix \ref{cha:cond-posit-semi}. In this case, it can be computed to be
\begin{equation}
  \label{eq:33}
      \widehat{\psi}(\xi) := -2\cdot(2\pi)^{d} D_{q} \, \left| \xi \right|^{-d-q}, \quad \text{with a } D_{q} > 0,
\end{equation}
with
\begin{equation}
  D_{q} := -(2\pi)^{-d/2}\frac{2^{q + d/2}\, \Gamma((d+q)/2)}{2\Gamma(-q/2)} > 0,
\end{equation}
so that
\begin{equation}
  \label{eq:34}
  \widehat{\mathcal{E}}[\mu] = D_{q} \int_{\mathbb{R}^d}\left| 
    \widehat{\mu}(\xi) - \widehat{\omega}(\xi) \right|^2
  \left| \xi \right|^{-d-q} \dd \xi,
\end{equation}
which will be proved in Section \ref{sec:fourier-formula-p2}.

Formula \eqref{eq:34} makes sense on the whole space \(\mathcal{P}(\mathbb{R}^{d})\) and the sub-levels of \( \widehat{\mathcal{E}} \) can be proved to be narrowly compact as well as lower semi-continuous \wrt to the narrow topology (see Proposition \ref{prp:compctness-sublvls}), motivating the proof in Section \ref{sec:extension-p} that up to a constant, this formula is exactly the lower semi-continuous envelope of \(\mathcal{E}\) on \(\mathcal{P}(\mathbb{R}^{d})\) endowed with the narrow topology.

\subsection{Fourier formula in $\mathcal{P}_{2}(\mathbb{R}^{d})$}
\label{sec:fourier-formula-p2}

Assume that \(\mu,\omega \in \mathcal{P}_{2}(\mathbb{R}^{d})\) and observe that by using the symmetry of $\psi$, $\mathcal{E}[\mu]$ can be written as
\begin{align}
  \mathcal{E}[\mu] = {} & -\frac{1}{2}\int_{\mathbb{R}^d \times \mathbb{R}^d} \psi(y - x) \dd \mu(x) \dd \mu(y)
  + 
  \frac{1}{2}\int_{\mathbb{R}^d \times \mathbb{R}^d} \psi(y - x) \dd \omega(x) \dd \mu(y) \\
  & +\frac{1}{2}\int_{\mathbb{R}^d \times \mathbb{R}^d} \psi(y - x) \dd \omega(y) \dd \mu(x) -
  \frac{1}{2}\int_{\mathbb{R}^d \times \mathbb{R}^d} \psi(y - x) \dd \omega(x) \dd \omega(y)\\
  & +\frac{1}{2}\int_{\mathbb{R}^d \times \mathbb{R}^d} \psi(y - x) \dd \omega(x) \dd \omega(y)\\
  = {} &- \frac{1}{2}\int_{\mathbb{R}^d \times \mathbb{R}^d} \psi(y - x) \dd [\mu - \omega](x) \dd [\mu - \omega](y) +
  C,
\end{align}
where
\begin{equation}
  \label{eq:36}
  C = \frac{1}{2}\int_{\mathbb{R}^d \times \mathbb{R}^d} \psi(y - x) \dd \omega(x) \dd \omega(y).
\end{equation}
In the following, we shall mostly work with the symmetrized variant and denote it
by
\begin{equation}
  \label{eq:37}
  \widetilde{\mathcal{E}}[\mu] := - \frac{1}{2}\int_{\mathbb{R}^d \times \mathbb{R}^d} \psi(y - x) \dd [\mu -
  \omega](x) \dd [\mu - \omega](y).
\end{equation}

\subsubsection{Representation for point-measures}
\label{sec:four-repr-point}

Our starting point is a representation of $\widetilde{\mathcal{E}}$ in the case
that $\mu$ and $\omega$ are point-measures, which has been derived in \cite{Wend05}.

\begin{lemma}
  Let $\mu$ and $\omega$ be finite sums of Dirac measures such that
  \begin{equation}
    \label{eq:38}
    \mu - \omega = \sum_{j = 1}^N \alpha_j \delta_{x_j}
  \end{equation}
  with suitable $N \in \mathbb{N}$, $\alpha_j \in \mathbb{R}$ and pairwise distinct $x_j \in \mathbb{R}^d$ for all $j = 1,\ldots,N$.
  Then
  \begin{equation}
    \label{eq:39}
    \widetilde{\mathcal{E}}[\mu] = -2^{-1}(2\pi)^{-d}\int_{\mathbb{R}^d}\left| \sum_{j = 1}^N \alpha_j
      \exp(\i x_j^{T}\xi) \right|^2 \widehat{\psi}(\xi) \dd \xi,
  \end{equation}
  where
  \begin{equation}
    \label{eq:40}
    \widehat{\psi}(\xi) := -2\cdot(2\pi)^{d} D_{q} \, \left| \xi \right|^{-d-q}, \quad \text{with a } D_{q} > 0.
  \end{equation}
\end{lemma}

\begin{proof}
  The claim is an application of a general representation theorem for conditionally positive semi-definite functions. An extensive introduction can be found in \cite{Wend05}, of which we have included a brief summary in Appendix \ref{cha:cond-posit-semi}. Here, we use Theorem \ref{thm:repr-thm-cond-semi} together with the explicit computation of the generalized Fourier transform of \( \psi \) in Theorem \ref{thm:cond-ft-power}.
                \end{proof}

\begin{remark}
  By
  \begin{equation}
    \label{eq:42}
    \overline{\exp(\i x)} = \exp(-\i x), \quad x \in \mathbb{R},
  \end{equation}
  we can also write the above formula \eqref{eq:39} as
  \begin{equation}
    \label{eq:43}
    \widetilde{\mathcal{E}}[\mu] = D_q\int_{\mathbb{R}^d}\left| \widehat{\mu}(\xi) - \widehat{\omega}(\xi) \right|^2 \left| \xi \right|^{-d-q} \dd \xi, \quad \xi \in \mathbb{R}^d.
  \end{equation}
\end{remark}

\subsubsection{Point approximation of probability measures by the empirical process}
\label{sec:point-appr-empir}

\begin{lemma}[Consistency of empirical process]
  \label{lem:3}
  Let \(\mu \in \mathcal{P}(\mathbb{R}^{d})\) and \((X_{i})_{i \in \mathbb{N}}\) be  a sequence of i.i.d. random variables with \(X_{i} \sim \mu\) for all \(i \in \mathbb{N}\). Then the empirical distribution
  \begin{equation}
    \label{eq:44}
    \mu_{N} := \frac{1}{N} \sum_{i=1}^{N} \delta_{X_{i}}
  \end{equation}
  converges with probability \(1\) narrowly to \(\mu\), i.e.
  \begin{equation}
    \label{eq:45}
    P(\{\mu_{N}\rightarrow \mu \text{ narrowly}\}) = 1.
  \end{equation}
  
  Additionally, if for a \(p \in [1,\infty)\), \(\int_{\mathbb{R}^{d}}\left| x \right|^{p}\dd \mu < \infty\), then \( x \mapsto \left| x \right|^p \) is almost surely uniformly integrable \wrt \( \left\{ \mu_N \right\}_N \), which by Lemma \ref{lem:26} implies almost sure convergence of \( \mu_N \rightarrow \mu \) in the \( p \)-Wasserstein topology.
\end{lemma}

\begin{proof}
  By Lemma \ref{lem:25}, it is sufficient to prove convergence of the integral functionals associated to a sequence of functions \( \left( f_k \right)_{k \in \mathbb{N}} \). But
  \begin{equation}
    \label{eq:319}
    \int_{\mathbb{R}^d} f_k(x) \diff \mu_N(x) = \frac{1}{N}\sum_{i = 1}^{N} f_k(X_i) \xrightarrow{N\rightarrow\infty} E[f_k(X)] = \int_{\mathbb{R}^d} f_k(x) \diff \mu(x),
  \end{equation}
  almost surely by the Strong Law of Large Numbers, \cite[Theorem 2.4.1]{Dur10}, leading to null sets \( A_k \) where the above convergence fails. Since a countable union of null sets is again a null set, the first claim follows.

  For the second claim, we apply the Strong Law of Large Numbers to the functions \( f_M(x) := \left| x \right|^p \cdot 1_{\left\{ \left| x \right|^p \geq M \right\}}\) for \( M > 0 \) to get the desired uniform integrability: For a given \( \varepsilon > 0 \), choose \( M > 0 \) large enough such that
  \begin{equation}
    \label{eq:321}
    \int_{\mathbb{R}^d} f_M(x) \diff \mu(x) < \frac{\varepsilon}{2},
  \end{equation}
  and then \( N_0 \in \mathbb{N} \) large enough such that
  \begin{equation}
    \label{eq:325}
    \left| \int_{\mathbb{R}^d} f_M(x) \diff \mu_N(x) - \int_{\mathbb{R}^d} f_M(x) \diff \mu(x) \right| < \frac{\varepsilon}{2}, \quad N \geq N_0, \text{ almost surely}.
  \end{equation}
  Now we possibly enlarge \( M \) by choosing \( M' \geq M \) sufficiently large to ensure that \( \left| X_{i} \right|^p < M' \) almost surely for all \( i < N_0 \). By the monotonicity of \( \int_{\mathbb{R}^d} f_M(x) \diff \mu(x) \) in \( M \), this ensures
  \begin{align}
    \label{eq:327}
    \sup_{N\in\mathbb{N}} \int_{\mathbb{R}^d} f_{M'}(x) \diff \mu_N = {} & \sup_{N \geq N_0} \int_{\mathbb{R}^d} f_{M'}(x) \diff \mu_N \leq {} \sup_{N\geq N_0} \int_{\mathbb{R}^d} f_M(x) \diff \mu_N(x) \\
    < {} & \frac{\varepsilon}{2} + \frac{\varepsilon}{2} = \varepsilon
  \end{align}

\end{proof}

\subsubsection{Representation for $\mathcal{P}_{2}(\mathbb{R}^{d})$}
\label{sec:four-repr-gener}

Now we establish continuity in both sides of \eqref{eq:39} with respect to the
\( 2 \)-Wasserstein-convergence to obtain the generalization we were aiming at.

\begin{lemma}[Continuity of \(\widetilde{\mathcal{E}}\)]
  \label{lem:5}
  Let
  \begin{equation}
    \label{eq:53}
    \mu_{k} \rightarrow  \mu, \quad \omega_{k} \rightarrow  \omega \quad     \text{for } k\rightarrow \infty \text{ in } \mathcal{P}_{2}(\mathbb{R}^{d}).
  \end{equation}
  Then,
  \begin{align}
    \label{eq:54}
    \leadeq\int_{\mathbb{R}^d \times \mathbb{R}^d} \psi(y - x) \dd [\mu_k - \omega_k](x) \dd [\mu_k - \omega_k](y) \\
    \rightarrow  {} &\int_{\mathbb{R}^d \times \mathbb{R}^d} \psi(y - x) \dd [\mu - \omega](x) \dd [\mu - \omega](y), \quad
    \text{for } k \rightarrow \infty.
  \end{align}
\end{lemma}

\begin{proof}
  By the particular choice of $\psi$, we have the estimate
  \begin{equation}
    \label{eq:55}
    \left| \psi(y - x) \right| \leq C(1 + \left| y - x \right|^2) \leq 2C(1 + \left| x \right|^2 +
    \left| y \right|^2).
  \end{equation}
  After expanding the expression to the left of \eqref{eq:54} so that we only have to deal with integrals with respect to probability measures, we can use this estimate to get the uniform integrability of the second moments of \( \mu \) and \( \omega \) by Lemma \ref{lem:26} and are then able to apply Lemma \ref{lem:4} to obtain convergence.
\end{proof}

\begin{lemma}[Continuity of \(\widehat{\mathcal{E}}\)]
  \label{lem:6}
  Let
  \begin{equation}
    \label{eq:56}
    \mu_{k}\rightarrow \mu,\quad \omega_{k}\rightarrow \omega \quad \text{for } k\rightarrow \infty\, \text{ in } \mathcal{P}_{2}(\mathbb{R}^{d}),
  \end{equation}
  such that
  \begin{equation}
    \label{eq:57}
    \mu_k - \omega_k = \sum_{j = 1}^{N_k} \alpha_j^k \delta_{x_j^k}
  \end{equation}
  for suitable $\alpha_j^k \in \mathbb{R}$ and pairwise distinct $x_j^k \in \mathbb{R}^{d}$. Then,
  \begin{align}
    \label{eq:58}
    \leadeq{\int_{\mathbb{R}^d}\left| \sum_{j = 1}^{N_k} \alpha_j^k \exp(\i \xi \cdot x_j^k) \right|^2
    \widehat{\psi}(\xi) \dd \xi} \\ \rightarrow  {}
    &\int_{\mathbb{R}^d}\left| \int_{\mathbb{R}^d} \exp(\i \xi \cdot x) \dd [\mu - \omega](x) \right|^2
    \widehat{\psi}(\xi) \dd \xi \quad \text{for } k\rightarrow \infty.
  \end{align}
\end{lemma}

\begin{proof}
  By the narrow convergence of $\mu_k$ and $\omega_k$, we get pointwise convergence of the Fourier transform, i.e.
  \begin{equation}
    \label{eq:59}
    \sum_{j = 1}^{N_k} \alpha_j^k \exp(\i \xi \cdot x_j^k) \rightarrow  \int_{\mathbb{R}^d} \exp(\i \xi \cdot x) \dd [\mu - \omega](x) \quad \text{for all } \xi \in \mathbb{R}^d \text{ and } k\rightarrow \infty.
  \end{equation}
  We want to use the Dominated Convergence Theorem: The Fourier transform  of \(\mu - \omega\) is bounded in \(\xi\), so that the case $\xi\rightarrow \infty$ poses no problem due to the integrability of $\widehat{\psi}(\xi) = C \left| \xi \right|^{-d-q}$ away from $0$. In order to justify the necessary decay at $0$, we use the control of the first moments (since we even control the second moments by the \(\mathcal{P}_{2}\) assumption): Inserting the Taylor expansion of the exponential function of order \(0\),
  \begin{equation}
    \label{eq:60}
    \exp(\i \xi \cdot x) = 1 + \i \xi \cdot x \int_0^1 \exp(\i \xi \cdot tx) \dd t,
  \end{equation}
  into the expression in question and using the fact that \(\mu_{k}\) and \(\omega_{k}\) are probability measures results in
  \begin{align}
    \leadeq{\left| \int_{\mathbb{R}^d} \exp(\i \xi \cdot x) \dd [\mu_k - \omega_k](x) \right| }\\
    = {} & \left| \int_{\mathbb{R}^d} \left( 1 + \i \xi \cdot x \int_0^1 \exp(\i \xi \cdot tx) \dd t \right) \dd
      [\mu_k - \omega_k](x) \right|\\
    = {} & \left| \int_{\mathbb{R}^d} \left( \i \xi \cdot x \int_0^1 \exp(\i \xi \cdot tx) \dd t \right) \dd [\mu_k - \omega_k](x) \right|\\
    \leq {} &|\xi| \underbrace{\left(\int_{\mathbb{R}^d} |x| \dd \mu_k(x) + \int_{\mathbb{R}^d} |x|\dd \omega_k(x) \right)}_{:= C}.
  \end{align}
  Therefore, we have a $k$-uniform bound $C$ such that
  \begin{equation}
    \label{eq:61}
    \left| \sum_{j = 1}^{N_k} \alpha_j^k \exp(\i \xi \cdot x_j^k) \right|^2 \leq C \, |\xi|^2,
  \end{equation}
  compensating the singularity of $\widehat{\psi}$ at the origin, hence together
  with the Dominated Convergence Theorem proving the claim.
\end{proof}

Combining the two lemmata above with the approximation provided by Lemma \ref{lem:3} yields

\begin{corollary}[Fourier-representation for $\widetilde{\mathcal{E}}$ on
  \(\mathcal{P}_{2}(\mathbb{R}^{d})\)]
  \label{cor:four-repr-widet}
  \begin{equation}
    \label{eq:62}
    \widetilde{\mathcal{E}}[\mu] = \widehat{\mathcal{E}}[\mu], \quad \mu \in \mathcal{P}_{2}(\mathbb{R}^{d}).
  \end{equation}
\end{corollary}

\subsection{Extension to $\mathcal{P}(\mathbb{R}^{d})$}
\label{sec:extension-p}

While the well-definedness of \(\mathcal{E}[\mu]\) is not clear for all \(\mu \in \mathcal{P}(\mathbb{R}^{d})\), since the sum of two integrals with values \( \pm \infty \) may occur, for each such \(\mu\) we can certainly assign a value in \(\mathbb{R}\cup \left\{ \infty \right\}\) to \(\widehat{\mathcal{E}}[\mu]\). In the following, we want to justify in what sense it is possible to consider \(\widehat{\mathcal{E}}\) instead of the original functional, namely that \( \widehat{\mathcal{E}} \) can be considered the \emph{lower semi-continuous envelope} of \( \mathcal{E} \).

Firstly, we prove that \(\widehat{\mathcal{E}}\) has compact sub-levels in \(\mathcal{P}(\mathbb{R}^{d})\) endowed with the narrow topology, using the following lemma as a main ingredient.

\emph{Please note that in the following, \(C\) will be used as a generic positive constant, which might change during the course of an equation.}

\begin{lemma}
  \label{lem:7}
  [See \cite[Theorem 3.3.6]{Dur10} for a proof in the case \( d = 1 \).] Given a probability measure \(\mu \in \mathcal{P}(\mathbb{R}^{d})\) with Fourier transform
  \(\widehat{\mu}\colon \mathbb{R}^{d} \rightarrow  \mathbb{C}\), there are \(C_{1} = C_{1}(d) > 0\) and \(C_{2} = C_{2}(d) > 0\) such that for all
  \(u > 0\),
  \begin{equation}
    \label{eq:64}
    \mu\left(\left\{x : \left| x \right| \geq u^{-1}\right\}\right) \leq
    \frac{C_{1}}{u^{d}}\int_{\left| \xi \right|
      \leq C_{2} u} (1 - \Re \widehat{\mu}(\xi)) \dd \xi.
  \end{equation}
\end{lemma}

\begin{proof}
  Let \(u > 0\). Firstly, note that
  \begin{equation}
    \label{eq:65}
    1 - \Re \widehat{\mu}(\xi) = \int_{\mathbb{R}^{d}}(1 - \cos(\xi \cdot x)) \dd \mu(x) 
    \geq 0 \quad \text{for all } \xi \in \mathbb{R}^{d}.
  \end{equation}
  By starting with the integral on the right-hand side of \eqref{eq:64} (up to a constant in the integration domain) and using Fubini-Tonelli as well as integration in spherical coordinates, we get
  \begin{align}
    \label{eq:66}
    \leadeq{\int_{\left| \xi \right|\leq u} (1 - \Re \widehat{\mu}(\xi)) \dd \xi}\\
    = {} & \int_{\mathbb{R}^{d}} \int_{\left| \xi \right| \leq u} (1 - \cos(\xi \cdot x)) \dd \xi \dd \mu(x)\\
    = {} & \int_{\mathbb{R}^{d}} \int_{\left| \widetilde{\xi} \right| = 1} \int_{0}^{u} (1 - \cos(r\widetilde{\xi} \cdot x)) r^{d-1}\dd r \dd \sigma(\widetilde{\xi}) \dd \mu(x) \label{eq:356}\\
    = {} & \int_{\mathbb{R}^{d}} \int_{\left| \widetilde{\xi} \right| = 1} \left[ \frac{u^d}{d} - \int_{0}^{u} \cos(r\widetilde{\xi} \cdot x) r^{d-1}\dd r \right] \dd \sigma(\widetilde{\xi}) \dd \mu(x)\label{eq:67}
  \end{align}
  If \(d \geq 2\), integrating the integral over \(\cos(r\widetilde{\xi} \cdot x)r^{d-1}\) in
  \eqref{eq:67} by parts yields
  \begin{align}
    \label{eq:68}
    \int_{0}^{u} \cos(r\widetilde{\xi} \cdot x) r^{d-1}\dd r = {} &
    \frac{\sin(u\widetilde{\xi} \cdot x) u^{d-1}}{\widetilde{\xi} \cdot x} - (d-1)
    \int_{0}^{u}\frac{\sin(r\widetilde{\xi} \cdot x)}{\widetilde{\xi} \cdot x}r^{d-2} \dd r,
  \end{align}
  which can also be considered true for \(d = 1\) if the second part is assumed to be zero because of the factor \((d-1)\).

  We now prove \eqref{eq:64} by estimating the integrand in \eqref{eq:67} suitably from below. Using \(\left| \sin(x) \right| \leq 1\) for all \(x \in \mathbb{R}\) and dividing by \(u^{d}\), we get
                  \begin{align}
    \label{eq:70}
    \leadeq{d^{-1} - u^{-d}\int_{0}^{u} \cos(r\widetilde{\xi} \cdot x) r^{d-1}\dd r} \\
    = {} & d^{-1} - \frac{\sin(u \widetilde{\xi} \cdot x)}{u\widetilde{\xi} \cdot x} + \frac{(d-1)}{u^{d}} \int_{0}^{u}\frac{\sin(r\widetilde \xi \cdot x)}{\widetilde{\xi} \cdot x} r^{d-2}\dd r\\
    \geq {} & d^{-1} - \frac{1}{u\left| \widetilde{\xi} \cdot x \right|} - \frac{(d-1)}{u^{d}} \int_{0}^{u}\frac{1}{\left| \widetilde{\xi} \cdot x \right|} r^{d-2}\dd r\\
    = {} & d^{-1} - \frac{2}{u\left| \widetilde{\xi} \cdot x \right|}.
  \end{align}
  As we want to achieve an estimate from below, by the non-negativity of the integrand \( 1 - \cos(\xi \cdot x) \), we can restrict the integration domain in \eqref{eq:356} to
    \begin{equation}
    \label{eq:71}
    \widetilde{S}(x) := \left\{ \widetilde \xi \in S^{d-1} : \left| \widetilde \xi \cdot x \right| \geq \frac{1}{2}\left| x \right| \right\} \quad
    \text{and} \quad D(u) := \left\{ x : \left| x \right| \geq \frac{8d}{u} \right\},
  \end{equation}
  yielding
  \begin{equation}
    \frac{1}{d} - \frac{1}{u^{d}}\int_{0}^{u} \cos(r\widetilde{\xi} \cdot x) r^{d-1}\dd r \geq \frac{1}{2d}, \quad x \in D(u),\, \widetilde{\xi} \in \widetilde{S}(x).\label{eq:72}
  \end{equation}
  Combining \eqref{eq:72} with \eqref{eq:67} gives us
  \begin{equation}
    \label{eq:73}
    \frac{1}{u^{d}} \int_{\left| \xi \right|\leq u} (1 - \Re \widehat{\mu}(\xi)) \dd \xi \geq \frac{1}{C_{3}} \mu\left(\left\{ \left| x \right| \geq 8D's^{-1} \right\}\right)
  \end{equation}
  with
  \begin{equation}
    \label{eq:74}
    C_{3} := \frac{1}{2d} \operatorname{vol} (\widetilde S(x)),
  \end{equation}
  where \(\operatorname{vol} (\widetilde S(x))\) is independent of \(x\). Finally, we substitute \(\widetilde u := (8d)^{-1}u\) to get
  \begin{equation}
    \label{eq:75}
    \mu\left(\left\{x : \left| x \right| \geq \widetilde{u}^{-1}\right\}\right) \leq \frac{C_{1}}{\widetilde{u}^{d}}\int_{\left| \xi \right| \leq C_{2}\widetilde{u}} (1 - \Re \widehat{\mu}(\xi)) \dd \xi
  \end{equation}
  with
  \begin{equation}
    \label{eq:76}
    C_{1} := \frac{C_{3}}{(8d)^{d}} \quad \text{and} \quad  C_{2} := 8d.\qedhere
  \end{equation}
\end{proof}

\raggedbottom
\begin{proposition}
  \label{prp:compctness-sublvls}
  \(\widehat{\mathcal{E}}\colon \mathcal{P}(\mathbb{R}^{d}) \rightarrow  \mathbb{R}_{\geq 0} \cup \{\infty\}\) is lower
  semi-continuous with respect to narrow convergence and its sub-levels are
  narrowly compact.
\end{proposition}

\begin{proof}
  Lower semi-continuity and thence closedness of the sub-levels follows from Fatou’s Lemma, because narrow convergence corresponds to pointwise convergence of the Fourier transform and the integrand in the definition of \(\widehat{\mathcal{E}}\) is non-negative.
  
  Now, assume we have a \(K > 0\) and
  \begin{equation}
    \label{eq:77}
    \mu \in N_{K}(\widehat{\mathcal{E}}) := \{\mu \in \mathcal{P}(\mathbb{R}^{d}) : \widehat{\mathcal{E}}[\mu] \leq K\}.
  \end{equation}
  We show the tightness of the family of probability measures
  \(N_{K}(\widehat{\mathcal{E}})\) using Lemma \ref{lem:7}. Let \(0 < u \leq 1\). Then,
  \begin{align}
    \label{eq:78}
    \leadeq[2]{\frac{1}{u^{d}} \int_{\left| \xi \right|\leq C_{2}u} \left(1 - \Re \widehat{\mu}(\xi)\right) \dd \xi}\\
    \leq {} & C_2^d \int_{\left| \xi \right|\leq C_{2}u} \left| \xi \right|^{-d} \left(1 - \Re \widehat{\mu}(\xi)\right) \dd \xi \\
    \leq {} & C_2^d \int_{\left| \xi \right|\leq C_{2}u} \left| \xi \right|^{-d} \left(\left| 1 - \Re \widehat{\omega}(\xi) \right| + \left| \Re \widehat{\omega}(\xi) - \Re \widehat{\mu}(\xi) \right| \right)\dd \xi \\
    \leq {} & C_2^d \int_{\left| \xi \right|\leq C_{2}u} \left| \xi \right|^{-d} \left(\left| 1 - \widehat{\omega}(\xi) \right| + \left| \widehat{\omega}(\xi) - \widehat{\mu}(\xi) \right| \right)\dd \xi\\
    = {} & C_2^d \int_{\left| \xi \right|\leq C_{2}u} \left| \xi \right|^{(-d-q)/2} \cdot \left| \xi \right|^{(-d+q)/2} \left(\left| 1 - \widehat{\omega}(\xi) \right| + \left| \widehat{\omega}(\xi) - \widehat{\mu}(\xi) \right| \right)\dd \xi\\
     \leq {} & C_2^d \, {\underbrace{\left(\int_{\left| \xi \right|\leq C_{2}u} \left| \xi \right|^{-d+q} \dd \xi\right)}_{\text{\(=:f(u)\)}}}^{1/2} \cdot \Bigg[{\underbrace{\left(\int_{\left| \xi \right|\leq C_{2}u} \left| \xi \right|^{-d-q} \left| 1 - \widehat{\omega}(\xi) \right|^{2}\dd \xi\right)}_{\text{\(= C \cdot \widehat{\mathcal{E}}[\delta_0] < \infty\)}}}^{1/2}\label{eq:79}\\
     & + {\underbrace{\left(\int_{\left| \xi \right|\leq C_{2}u}\left| \xi \right|^{-d-q} \left| \widehat{\omega}(\xi) - \widehat{\mu}(\xi) \right|^{2} \dd \xi\right)}_{\text{\(\leq D_{q}^{-1}K\)}}}^{1/2}\Bigg] \quad \text{(Hölder’s inequality)} \label{eq:80}\\
    \leq {} & C_2^d \, (f(u))^{1/2} \left( C^{1/2} + \left(D_{q}^{-1}K\right)^{1/2} \right),
  \end{align}
  where in equations \eqref{eq:79} and \eqref{eq:80} we used the boundedness of the first summand in \eqref{eq:79} by a constant \(C > 0\), which is justified because \(\omega\) has an existing second moment. But
  \begin{equation}
    \label{eq:81}
    f(u) = \int_{\left| \xi \right|\leq C_{2} u} \left| \xi \right|^{-d+q} \dd \xi = O(u^{q})
    \quad \text{for } u \rightarrow  0,
  \end{equation}
  giving a uniform control of the convergence to zero of the left-hand side of \eqref{eq:78}. Together with Lemma \ref{lem:7}, this yields tightness of \(N_{K}(\widehat{\mathcal{E}})\), hence relative compactness with respect to narrow convergence. Compactness then follows from the aforementioned lower semi-continuity of \(\widehat{\mathcal{E}}\).
                                                        \end{proof}

From this proof, we cannot deduce a stronger compactness, so that the limit of a minimizing sequence for the original functional \(\widetilde{\mathcal{E}}\) (which coincides with \(\widehat{\mathcal{E}}\) on \(\mathcal{P}_{2}(\mathbb{R}^{d})\) by Corollary \ref{cor:four-repr-widet}) need not lie in the set \(\mathcal{P}_{2}(\mathbb{R}^{d})\) (actually, in Section \ref{sec:moment-bound-symm}, we shall see that we can prove a slightly stronger compactness). To apply compactness arguments, we hence need an extension of \(\widetilde{\mathcal{E}}\) to the whole of \(\mathcal{P}(\mathbb{R}^{d})\). For the direct method in the calculus of variations to work, this extension should also be lower semi-continuous; therefore the natural candidate is the \emph{lower semi-continuous envelope} of \(\widetilde{\mathcal{E}}\), now defined on the whole of \( \mathcal{P}(\mathbb{R}^d) \) by
\begin{equation}
  \label{eq:357}
  \widetilde{\mathcal{E}}[\mu] =
  \begin{cases}
    \widetilde{\mathcal{E}}[\mu], \quad & \mu \in \mathcal{P}_2(\mathbb{R}^d), \\
    \infty, & \mu \in \mathcal{P}(\mathbb{R}^d) \setminus \mathcal{P}_2(\mathbb{R}^d),
  \end{cases}
\end{equation}
which in our case can be defined as
\begin{equation}
  \label{eq:85}
  \widetilde{\mathcal{E}}^{-}[\mu] := \inf_{\substack{\mu_{n}\rightarrow \mu \text{        narrowly}\\\mu_{n} \in \mathcal{P}_{2}(\mathbb{R}^{d})}} \liminf_{n\rightarrow \infty}  \widetilde{\mathcal{E}}[\mu_{n}],
\end{equation}
or equivalently as the largest lower semi-continuous function \( \widetilde{\mathcal{E}^-} \leq \widetilde{\mathcal{E}} \).
This corresponds to \cite[Definition 3.1]{93-Dal_Maso-intro-g-conv} if we consider our functional initially to be \( +\infty \) for \( \mu \in \mathcal{P}(\mathbb{R}^d) \setminus \mathcal{P}_2(\mathbb{R}^d) \).

In order to show that \(\widetilde{\mathcal{E}}^{-} = \widehat{\mathcal{E}}\), which is the content of Corollary \ref{cor:lower-semi-cont} below, we need a sequence along which there is continuity in the values of \( \widetilde{\mathcal{E}} \), which we find by dampening an arbitrary \( \mu \) with a Gaussian:

\flushbottom
\begin{proposition}
  \label{prp:rec-seq}
  For \( \omega \in \mathcal{P}_2(\mathbb{R}^d) \) and \(\mu \in \mathcal{P}(\mathbb{R}^{d})\), there exists a sequence \((\mu_{n})_{n \in \mathbb{N}} \subseteq \mathcal{P}_{2}(\mathbb{R}^{d})\) such that
  \begin{alignat}{2}
    \label{eq:86}
    \mu_{n}&\rightarrow \mu \text{ narrowly} \quad&&\text{ for } n\rightarrow \infty,\\
    \widehat{\mathcal{E}}[\mu_{n}]&\rightarrow \widehat{\mathcal{E}}[\mu] &&\text{ for } n\rightarrow \infty.
  \end{alignat}
\end{proposition}

\begin{proof}
  \thmenumhspace{-1em}
  \begin{enumpara}
  \item \emph{Definition of \(\mu_{n}\).} Define
    \begin{equation}
      \label{eq:87}
      \eta(x) := (2\pi)^{-d/2}\exp\left(-\frac{1}{2}\left| x \right|^{2}\right), \quad \eta_{\varepsilon}(x) := \varepsilon^{-d}\eta(\varepsilon^{-1}x), \quad x \in \mathbb{R}^{d}.
          \end{equation}
    Then \((2\pi)^{-d} \widehat{\widehat{\eta_{\varepsilon}}} = \eta_\varepsilon \) is a non-negative approximate identity with respect to the convolution and \( \widehat{\eta}_\varepsilon = \exp(-\varepsilon^2 \left| x \right|^2/2) \). To approximate \(\mu\), we use a smooth dampening of the form
    \begin{equation}
      \label{eq:88}
      \mu_{n} := \widehat{\eta}_{n^{-1}}\cdot \mu + \left(1 - (\widehat{\eta}_{n^{-1}}\cdot \mu)(\mathbb{R}^{d})\right)\delta_{0},
    \end{equation}
    such that the resulting \( \mu_n \) are in \( \mathcal{P}_2 \), with Fourier transforms
    \begin{equation}
      \label{eq:89}
      \widehat{\mu}_{n}(\xi) =  (\widehat{\mu} \ast
      \eta_{n^{-1}})(\xi) - (\widehat{\mu} \ast
      \eta_{n^{-1}})(0) + 1, \quad \xi \in \mathbb{R}^{d}.
    \end{equation}
    Note that because \(\widehat{\mu}\) is continuous, \(\widehat{\mu}_{n}(\xi) \rightarrow \widehat{\mu}(\xi)\) for all \(\xi \in \mathbb{R}^{d}\). We want to use the Dominated Convergence Theorem to deduce that
    \begin{equation}
      \label{eq:90}
      \widehat{\mathcal{E}}[\mu_{n}] = D_{q} \int_{\mathbb{R}^{d}} \left| \xi \right|^{-d-q}
      \left| \widehat{\mu}_{n}(\xi)-\widehat{\omega}(\xi) \right|^{2} \dd \xi \rightarrow  \widehat{\mathcal{E}}[\mu] \quad \text{ for }
      n\rightarrow \infty.
    \end{equation}
  \item \emph{Trivial case and dominating function.} Firstly, note that if \(\widehat{\mathcal{E}}[\mu] = \infty\), then Fatou’s Lemma ensures that \(\widehat{\mathcal{E}}[\mu_{n}]\rightarrow \infty\) as well.

    Secondly, by the assumptions on \( \omega \), it is sufficient to find a dominating function for \( \xi \mapsto \left| \xi \right|^{-d-q} \left| \widehat{\mu}_n(\xi) - 1 \right|^2 \), which will only be problematic for \( \xi \) close to \( 0 \). We can estimate the behavior of \(\widehat{\mu}_{n}\) there by the behavior of \(\widehat{\mu}\) there by computing
    \begin{align}
      \label{eq:91}
      \left| \widehat{\mu}_{n}(\xi) - 1 \right| \leq {} & \int_{\mathbb{R}^{d}} \int_{\mathbb{R}^{d}} \eta_{n^{-1}}(\zeta) \left| \exp(\i(\zeta-\xi)\cdot x) - \exp(\i\zeta\cdot x) \right| \dd \mu(x) \dd \zeta \\
      = {} & \int_{\mathbb{R}^{d}} \underbrace{\int_{\mathbb{R}^{d}} \eta_{n^{-1}}(\zeta) \dd \zeta}_{\text{\(=1\)}}\,\left| \exp(-i\xi\cdot x) - 1 \right| \dd \mu(x) \\
      \leq {} &C\bigg[(1 - \Re \widehat{\mu}(\xi)) + \underbrace{\int_{\mathbb{R}^{d}} \left| \sin (\xi\cdot x) \right| \dd \mu(x)}_{\text{\(:= f(\xi)\)}}\bigg],\label{eq:92}
    \end{align}
    where the right-hand side \eqref{eq:92} is to serve as the dominating function. Note that we can estimate each summand in \eqref{eq:92} separately to justify integrability due to the elementary inequality
    \begin{equation}
      \label{eq:93}
      \left| a+b \right|^{2} \leq 2\left(\left| a \right|^{2} + \left| b \right|^{2}\right) \quad
      \text{for all } a,b \in \mathbb{C}.
    \end{equation}
    Taking the square of \eqref{eq:92} yields
    \begin{equation}
      \label{eq:94}
      \left| \widehat \mu_{n}(\xi) - 1 \right|^{2} \leq C\left[(1 - \Re \widehat{\mu}(\xi))^{2} + \left(\int_{\mathbb{R}^{d}} \left| \sin (\xi\cdot x) \right| \dd \mu(x)\right)^{2}\right].
    \end{equation}
    Now, by the existence of the second moment of \(\omega\), we know that
                                    \begin{align}
      \label{eq:96}
      \leadeq{\int_{\mathbb{R}^{d}} \left| \xi \right|^{-d-q} (1 - \Re \widehat{\mu}(\xi))^{2} \dd \xi}\\
      \leq {} & \int_{\mathbb{R}^{d}} \left| \xi \right|^{-d-q} \left| \widehat{\mu}(\xi) - 1 \right|^{2} \dd \xi\\
      \leq {} & 2\int_{\mathbb{R}^{d}} \left| \xi \right|^{-d-q}\left| \widehat{\mu}(\xi) - \widehat{\omega}(\xi) \right|^{2} \dd \xi + 2 \int_{\mathbb{R}^{d}} \left| \xi \right|^{-d-q} \left| \widehat{\omega}(\xi) - 1 \right|^{2} \dd \xi < \infty\label{eq:97}
    \end{align}
    This yields the integrability condition for the first term in equation \eqref{eq:94}. What remains is to show the integrability for the term \(f\), which will occupy the rest of the proof.
  \item \emph{Splitting \(f\).} We apply the estimate
    \begin{equation}
      \label{eq:98}
      \left| \sin(y) \right| \leq \min\{\left| y \right|,1\} \quad \text{for } y \in \mathbb{R},
    \end{equation}
    resulting in
    \begin{equation}
      \label{eq:99}
      f(\xi) = \int_{\mathbb{R}^{d}} \left| \sin(\xi\cdot x) \right| \dd \mu(x) \leq \underbrace{\left| \xi \right|
        \int_{\left| x \right| \leq \left| \xi \right|^{-1}} \left| x \right| \dd \mu(x)}_{\text{\(:= f_{1}(\xi)\)}}
      + \underbrace{\int_{\left| x \right| \geq
          \left| \xi \right|^{-1}} \dd \mu(x)}_{\text{\(:= f_{2}(\xi)\)}}.
    \end{equation}
  \item \emph{Integrability of \(f_{2}\):}\label{item:1} By Lemma \ref{lem:7}
    and Hölder's inequality, we can estimate \(f_{2}\) as follows:
    \begin{align}
      \label{eq:100}
      f_{2}(\xi) \leq {} &\frac{C_{1}}{\left| \xi \right|^{d}} \int_{\left| y \right| \leq C_{2}\left| \xi \right|} (1
      - \Re \widehat{\mu}(y)) \dd y\\
      \leq {} & \frac{C_{1}}{\left| \xi \right|^{d}} {\underbrace{\left(\int_{\left| y \right| \leq
              C_{2}\left| \xi \right|} 1 \dd y\right)^{1/2}}_{\text{\(= C\left| \xi \right|^{d/2}\)}}}
      \left(\int_{\left| y \right| \leq C_{2}\left| \xi \right|}(1 - \Re \widehat{\mu}(y))^{2} \dd
        y\right)^{1/2}\label{eq:101}
    \end{align}
    Hence, inserting \eqref{eq:101} into the integral which we want to show to be finite and applying Fubini-Tonelli yields
    \begin{align}
      \label{eq:102}
      \int_{\mathbb{R}^{d}} \left| \xi \right|^{-d-q} f_{2}(\xi)^{2} \dd \xi \leq {} &C \int_{\mathbb{R}^{d}} \left| \xi \right|^{-2d-q} \int_{\left| y \right|\leq C_{2}\left| \xi \right|} (1 - \Re \widehat{\mu}(y))^{2}\dd y \dd \xi\\
      \leq {} & C\int_{\mathbb{R}^{d}} (1 - \Re \widehat{\mu}(y))^{2} \underbrace{\int_{C_{2}\left| \xi \right| \geq \left| y \right|} \left| \xi \right|^{-2d-q} \dd \xi}_{\text{\(=C\left| y \right|^{-d-q}\)}} \dd y\\
      \leq {} & C\int_{\mathbb{R}^{d}} \left| y \right|^{-d-q}(1 - \Re \widehat{\mu}(y))^{2} \dd y < \infty
    \end{align}
    by \eqref{eq:97}.
  \item \emph{Integrability of \(f_{1}\):} We use Fubini-Tonelli to get a well-known estimate for the first moment, namely
    \begin{align}
      \label{eq:103}
      f_{1}(\xi) = {} & \left| \xi \right| \int_{\left| x \right|\leq \left| \xi \right|^{-1}} \left| x \right| \dd
      \mu(x)\\
      = {} & \left| \xi \right| \int_{\left| x \right|\leq \left| \xi \right|^{-1}} \int_{0}^{\left| x \right|} 1
      \dd z \dd \mu(x)\\
      = {} & \left| \xi \right| \int_{0}^{\infty} \int_{\mathbb{R}^{d}} 1_{\text{\(\{z \leq \left| x \right| \leq \left| \xi \right|^{-1}\}\)}} \dd \mu(x) \dd z\\
      \leq {} & \left| \xi \right| \int_{0}^{\left| \xi \right|^{-1}} \mu(\{z \leq \left| x \right|\}) \dd z.
    \end{align}
    Next, we use Lemma \ref{lem:7} and Hölder's inequality (twice) to obtain
    (remember that \(1 \leq q < 2\) which ensures integrability)
    \begin{align}
      \label{eq:104}
      f_{1}(\xi) \leq {} & C_{1} \left| \xi \right| \int_{0}^{\left| \xi \right|^{-1}} z^{d} \int_{\left| \zeta \right| \leq C_{2}z^{-1}} (1 - \Re \widehat{\mu}(\zeta)) \dd \zeta \dd z\\
      \leq {} & C_{1} \left| \xi \right| \int_{0}^{\left| \xi \right|^{-1}} {z^{d}}\underbrace{\left(\int_{\left| \zeta \right| \leq C_{2} z^{-1}} 1 \dd \zeta\right)^{1/2}}_{\text{\(=C\,z^{-d/2} = C\,z^{q/4 + (-d/2 - q/4)}\)}} \left(\int_{\left| \zeta \right| \leq C_{2} z^{-1}} (1 - \Re \widehat{\mu}(\zeta))^{2} \dd \zeta\right)^{1/2} \dd z\\
      \leq {} & C \left| \xi \right| \underbrace{\left(\int_{0}^{\left| \xi \right|^{-1}} z^{-q/2}\dd z\right)^{1/2}}_{\text{\(=C\left| \xi \right|^{q/4-1/2}\)}} \left(\int_{0}^{\left| \xi \right|^{-1}} \int_{\left| \zeta \right| \leq C_{2}z^{-1}} z^{d+q/2}(1 - \Re \widehat{\mu}(\zeta))^{2} \dd \zeta \dd z\right)^{1/2}.
    \end{align}
    Squaring the expression and using Fubini-Tonelli on the second term, we get
    \begin{align}
      \label{eq:105}
      f_{1}(\xi)^{2} \leq {} & C \left| \xi \right|^{1+q/2} \int_{\mathbb{R}^{d}} (1 - \Re \widehat{\mu}(\zeta))^{2} \int_{0}^{\left| \xi \right|^{-1}} 1_{\{z \leq C_{2}\left| \zeta \right|^{-1}\}} z^{d+q/2} \dd z \dd \zeta\\
      \leq {} & C \left| \xi \right|^{1+q/2} \int_{\mathbb{R}^{d}} (1 - \Re \widehat{\mu}(\zeta))^{2} \min\left\{\left| \xi \right|^{-d-q/2-1},\left| \zeta \right|^{-d-q/2-1}\right\} \dd \zeta\\
      = {} & C \left| \xi \right|^{-d} \int_{\left| \zeta \right| \leq \left| \xi \right|} (1 - \Re \widehat{\mu}(\zeta))^{2} \dd \zeta\label{eq:106}\\
      &+ \underbrace{C \left| \xi \right|^{1+q/2} \int_{\left| \zeta \right|\geq\left| \xi \right|}\left| \zeta \right|^{-d-q/2-1} (1 - \Re \widehat{\mu}(\zeta))^{2} \dd \zeta}_{:= f_{3}(\xi)}\label{eq:107}
    \end{align}
    The integrability against \(\xi \mapsto \left| \xi \right|^{-d-q}\) of the term \eqref{eq:106} can now be shown analogously to \eqref{eq:101} in Step \ref{item:1}. Inserting the term \eqref{eq:107} into the integral and again applying Fubini-Tonelli yields
    \begin{align}
      \label{eq:108}
      \leadeq{\int_{\mathbb{R}^{d}} \left| \xi \right|^{-d-q} f_{3}(\xi)^{2} \dd \xi} \\
      \leq {} & C \int_{\mathbb{R}^{d}} \left| \xi \right|^{-d-q/2+1} \int_{\left| \zeta \right|\geq\left| \xi \right|} \left| \zeta \right|^{-d-q/2-1} (1 - \Re \widehat{\mu}(\zeta))^{2} \dd \zeta \dd \xi\\
      = {} & C \int_{\mathbb{R}^{d}} \left| \zeta \right|^{-d-q/2-1} (1 - \Re \widehat{\mu}(\zeta))^{2} \underbrace{\int_{\left| \xi \right|\leq\left| \zeta \right|} \left| \xi \right|^{-d-q/2+1}\dd \xi}_{\text{\(=C\left| \zeta \right|^{-q/2+1}\)}} \dd \zeta\\
      = {} & C \int_{\mathbb{R}^{d}} \left| \zeta \right|^{-d-q} (1 - \Re \widehat{\mu}(\zeta))^{2} \dd \zeta < \infty,
    \end{align}
    because of \eqref{eq:97}, which ends the proof. \qedhere
  \end{enumpara}
\end{proof}

\begin{corollary}
  \label{cor:lower-semi-cont}
  We have that
  \begin{equation}
    \label{eq:109}
    \widetilde{\mathcal{E}}^{-}[\mu] = \widehat{\mathcal{E}}[\mu], \quad \mu \in \mathcal{P}(\mathbb{R}^{d})
  \end{equation}
  and that \( \omega \) is the unique minimizer of \( \widetilde{\mathcal{E}}^- \).
\end{corollary}

\begin{proof}
  For \(\mu \in \mathcal{P}(\mathbb{R}^{d})\) and any sequence \((\mu_{n})_{n\in\mathbb{N}} \subseteq
  \mathcal{P}_{2}(\mathbb{R}^{d})\) with \(\mu_{n}\rightarrow \mu\) narrowly, we have
  \begin{equation}
    \label{eq:110}
     \liminf_{n\rightarrow \infty} \widetilde{\mathcal{E}}[\mu_{n}] = \liminf_{n\rightarrow \infty}
     \widehat{\mathcal{E}}[\mu_{n}] \geq \widehat{\mathcal{E}}[\mu]
  \end{equation}
  by the lower semi-continuity of \(\widehat{\mathcal{E}}\). By taking the infimum, we conclude
  \begin{equation}
    \label{eq:111}
    \widetilde{\mathcal{E}}^{-}[\mu] \geq \widehat{\mathcal{E}}[\mu] \quad \text{for all } \mu \in
    \mathcal{P}(\mathbb{R}^{d}).
  \end{equation}
  
  Conversely, for \(\mu \in \mathcal{P}(\mathbb{R}^{d})\), employing the sequence \((\mu_{n})_{n\in\mathbb{N}}\subseteq \mathcal{P}_{2}(\mathbb{R}^{d})\) of Proposition \ref{prp:rec-seq} allows us to see that
  \begin{equation}
    \label{eq:112}
    \widehat{\mathcal{E}}[\mu] = \lim_{n\rightarrow \infty} \widehat{\mathcal{E}}[\mu_{n}] = \lim_{n\rightarrow \infty}
    \widetilde{\mathcal{E}}[\mu_{n}] \geq \widetilde{\mathcal{E}}^{-}[\mu].
  \end{equation}
  Combining \eqref{eq:112} with \eqref{eq:111} yields the first claim, while the characterization of the minimizer follows from the form of \( \widehat{\mathcal{E}}\) in \eqref{eq:34}.
\end{proof}

Having verified this, in the following we shall work with the functional
\(\widehat{\mathcal{E}}\) instead of \(\mathcal{E}\) or \(\widetilde{\mathcal{E}}\).

\begin{remark}
  The lower semi-continuous envelope and therefore \(\widehat{\mathcal{E}}\) is also the \(\Gamma\)-limit, see Definition \ref{def:gamma-conv} below, of a regularization of \(\widetilde{\mathcal{E}}\) using the second moment, i.e.\@ with
  \begin{equation}
    \mathcal{I}_{\varepsilon}[\mu] := \widetilde{\mathcal{E}}[\mu] + \varepsilon \int_{\mathbb{R}^{d}}\left| x \right|^{2} \dd \mu,
  \end{equation}
  we have
  \begin{equation}
    \mathcal{I}_{\varepsilon} \xrightarrow{\Gamma} \widetilde{\mathcal{E}}^{-} \quad \text{for
    } \varepsilon \rightarrow 0.
  \end{equation}
\end{remark}

\subsection{Consistency of the particle approximations}
\label{sec:part-appr}

We are interested in particle approximations to the minimization problem in accordance with the derivation of the functional in \cite{FHS12}. For this, let \(N \in \mathbb{N}\) and define
\begin{equation}
  \label{eq:114}
  \mathcal{P}^{N}(\mathbb{R}^{d}) := \left\{ \mu \in \mathcal{P}(\mathbb{R}^{d}) : \mu =
    \frac{1}{N}\sum_{i = 1}^{N}\delta_{x_{i}} \text{ for some } \{x_{i}\}_{i=1}^{N}
    \subseteq \mathbb{R}^{d} \right\}
\end{equation}
and consider the restricted minimization problem
\begin{equation}
  \label{eq:115}
  \widehat{\mathcal{E}}_{N}[\mu] :=
  \begin{cases}
    \widehat{\mathcal{E}}[\mu], &\mu \in \mathcal{P}^{N}(\mathbb{R}^{d}),\\
    \infty, &\text{otherwise}
  \end{cases}\rightarrow \min_{\mu \in \mathcal{P}(\mathbb{R}^{d})}.
\end{equation}

We want prove consistency of the restriction in terms of \(\Gamma\)-convergence of \(\widehat{\mathcal{E}}_{N}\) to \(\widehat{\mathcal{E}}\).

\begin{definition} [\( \Gamma \)-convergence]
  \label{def:gamma-conv}
  \cite[Definition 4.1, Proposition 8.1]{93-Dal_Maso-intro-g-conv}
  Let \( X \) be a metrizable space and \( F_N \colon X \rightarrow (-\infty,\infty] \), \( N \in \mathbb{N} \) be a sequence of functionals. Then we say that \( F_N \) \emph{\( \Gamma \)-converges} to \( F \), written as \( F_N \xrightarrow{\Gamma} F \), for an \( F \colon X \rightarrow (-\infty,\infty] \), if
  \begin{enumerate}
  \item \emph{\( \liminf \)-condition:} For every \( x \in X \) and every sequence \( x_N \rightarrow x \),
    \begin{equation}
      \label{eq:116}
      F(x) \leq \liminf_{N\rightarrow\infty} F_N(x_N);
    \end{equation}
  \item \emph{\( \limsup \)-condition:} For every \( x \in X \), there exists a sequence \( x_N \rightarrow x \), called \emph{recovery sequence}, such that
    \begin{equation}
      \label{eq:117}
      F(x) \geq \limsup_{N\rightarrow\infty} F_N(x_N).
    \end{equation}
  \end{enumerate}

  Furthermore, we call the sequence \( (F_N)_N \) \emph{equi-coercive} if for every \( c \in \mathbb{R} \) there is a compact set \( K \subseteq X \) such that \( \left\{ x : F_N(x) \leq c \right\} \subseteq K \) for all \( N \in \mathbb{N} \).
\end{definition}

\begin{lemma}
  [Convergence of minimizers]
  \label{lem:8}
  Let \( (F_N)_N \) be a family of equi-coercive functionals on a metrizable space \( X \), \( F_N \xrightarrow{\Gamma} F \) and \( x_N \in \argmin F_N \). Then, there is a subsequence \( (x_{N_k})_k \) and \( x^\ast \in X \) with
  \begin{equation}
    \label{eq:118}
    x_{N_k} \rightarrow x^\ast \in \argmin F.
  \end{equation}
\end{lemma}

\begin{proof}
  Let \( \left( x_N \right)_N \) be such a sequence. By equi-coercivity, it has a convergent subsequence \( \left( x_{N_k} \right)_k \), \( x_{N_k} \rightarrow x^\ast \).

  Now, let \( \widetilde{x} \in X \). By the \( \limsup \)-condition, there exists another sequence \( \widetilde{x}_{N_k} \) with \( \widetilde{x}_{N_k} \rightarrow \widetilde{x} \) and
  \begin{equation}
    \label{eq:119}
    \limsup_k F_{N_k}(\widetilde{x}_{N_k}) \leq F(\widetilde{x}).
  \end{equation}
  On the one hand, the \( \liminf \)-condition yields
  \begin{equation}
    \label{eq:120}
    F(x^\ast) \leq \liminf_k F_{N_k}(x_{N_k}),
  \end{equation}
  while on the other hand, by the fact that the \( x_{N_k} \) are minimizers,
  \begin{equation}
    \label{eq:121}
    F_{N_k}(x_{N_k}) \leq F_{N_k}(\widetilde{x}_{N_k}), \quad k \in \mathbb{N},
  \end{equation}
  which combined gives
  \begin{equation}
    \label{eq:122}
    F(x^\ast) \leq \liminf_k F_{N_k}(x_{N_k}) \leq \limsup_k F_{N_k}(\widetilde{x}_{N_k}) \leq F(\widetilde{x}),
  \end{equation}
  showing that the limit \( x^\ast \) is indeed a minimizer of \( F \).
\end{proof}

We shall need a further simple lemma justifying the existence of minimizers for the problem \eqref{eq:115}.

\begin{lemma}
  \label{lem:9}
  For all \(N \in \mathbb{N}\), \(\mathcal{P}^{N}(\mathbb{R}^{d})\) is closed in the narrow topology.
\end{lemma}

\begin{proof}
  Note that \(\mathcal{P}(\mathbb{R}^{d})\) endowed with the narrow topology is a metrizable space, hence it is Hausdorff and we can characterize its topology by sequences. Let \(N \in \mathbb{N}\) and \((\mu_{k})_{k \in \mathbb{N}} \subseteq \mathcal{P}^{N}(\mathbb{R}^{d})\) with
  \begin{equation}
    \label{eq:123}
    \mu_{k} \rightarrow  \mu \in \mathcal{P}(\mathbb{R}^{d}) \quad \text{narrowly for } k\rightarrow \infty.
  \end{equation}
  By ordering the points composing each measure, for example using a lexicographical ordering, we can identify the measures \(\mu_{k}\) with a collection of points \(x^{k} \in \mathbb{R}^{d\times N}\). As the sequence \((\mu_{k})_{k}\) is convergent, it is tight, whence the columns of \((x^{k})_{k}\) must all lie in a compact set \(K \subseteq \mathbb{R}^{d}\). So we can extract a subsequence \((x^{k_{l}})_{l \in \mathbb{N}}\) with
  \begin{equation}
    \label{eq:124}
    x^{k_{l}} \rightarrow  x^{\ast} = (x^{\ast}_{i})_{i = 1}^{N} \in \mathbb{R}^{d\times N} \quad \text{for } l\rightarrow \infty.
  \end{equation}
  This implies that
  \begin{equation}
    \label{eq:125}
    \mu_{k_{l}} \rightarrow  \mu^{\ast} = \frac{1}{N} \sum_{i}^{N}\delta_{x^{\ast}_{i}} \quad \text{narrowly for } l\rightarrow \infty.
  \end{equation}
  Since \(\mathcal{P}(\mathbb{R}^{d})\) is Hausdorff, \(\mu = \mu^{\ast} \in \mathcal{P}^{N}(\mathbb{R}^{d})\), concluding the proof.
\end{proof}

\begin{theorem}[Consistency of particle approximations]
  \label{thm:cons-part-appr}
  The functionals \((\widehat{\mathcal{E}}_{N})_{N \in \mathbb{N}}\) are equi-coercive and
  \begin{equation}
    \label{eq:131}
    \widehat{\mathcal{E}}_{N} \xrightarrow{\Gamma} \widehat{\mathcal{E}} \quad \text{for } N\rightarrow \infty
  \end{equation}
  with respect to the narrow topology. In particular,
  \begin{equation}
    \label{eq:132}
    \argmin_{\mu \in \mathcal{P}(\mathbb{R}^{d})} \widehat{\mathcal{E}}_{N}[\mu] \ni \widetilde{\mu}_{N} \rightarrow \widetilde{\mu} = \argmin_{\mu \in \mathcal{P}(\mathbb{R}^{d})} \widehat{\mathcal{E}}[\mu] = \omega,
  \end{equation}
  for each choice of minimizers \(\mu_{N}\).
\end{theorem}

\begin{proof}
  \thmenumhspace{-1em}
  \begin{enumpara}
  \item \emph{Equi-coercivity:} This follows from the fact that \(\widehat{\mathcal{E}}\) has compact sub-levels by Proposition \ref{prp:compctness-sublvls}, together with \(\widehat{\mathcal{E}}_{N} \geq \widehat{\mathcal{E}}\).

  \item \emph{\( \liminf \)-condition:} Let \(\mu_{N} \in \mathcal{P}(\mathbb{R}^{d})\) with \(\mu_{N} \rightarrow \mu\) narrowly for \(N\rightarrow \infty\). Then
    \begin{equation}
      \label{eq:133}
      \liminf_{N\rightarrow \infty} \widehat{\mathcal{E}}_{N}[\mu_{N}] \geq \liminf_{N\rightarrow \infty} \widehat{\mathcal{E}}[\mu_{N}] \geq \widehat{\mathcal{E}}[\mu]
    \end{equation}
    by the lower semi-continuity of \(\widehat{\mathcal{E}}\).

  \item \emph{\( \limsup \)-condition:} Let \(\mu \in \mathcal{P}(\mathbb{R}^{d})\). By Proposition \ref{prp:rec-seq}, we can find a sequence \((\mu^{k})_{k \in \mathbb{N}} \subseteq \mathcal{P}_{2}(\mathbb{R}^{d})\) for which \(\widehat{\mathcal{E}}[\mu^{k}] \rightarrow \widehat{\mathcal{E}}[\mu]\). Furthermore, by Lemma \ref{lem:3}, we can approximate each \(\mu^{k}\) by \((\mu^{k}_{N})_{N\in\mathbb{N}} \subseteq \mathcal{P}_{2}(\mathbb{R}^{d})\cap\mathcal{P}^{N}(\mathbb{R}^{d})\), a realization of the empirical process of \(\mu^{k}\). This has a further subsequence which converges in the \(2\)-Wasserstein distance by Lemma \ref{lem:26} for which we have continuity of \(\widehat{\mathcal{E}}\) by Lemma \ref{lem:6}. A diagonal argument then yields a sequence \(\mu_{N} \in \mathcal{P}^{N}(\mathbb{R}^{d})\) for which
    \begin{equation}
      \label{eq:134}
      \widehat{\mathcal{E}}_{N}[\mu_{N}] = \widehat{\mathcal{E}}[\mu_{N}] \rightarrow  \widehat{\mathcal{E}}[\mu] \quad \text{for } N\rightarrow \infty.
    \end{equation}

  \item \emph{Convergence of minimizers:} We find minimizers for \(\widehat{\mathcal{E}}_{N}\) by applying the direct method in the calculus of variations, which is justified because the \((\widehat{\mathcal{E}}_{N})_{N}\) are equi-coercive and each \(\widehat{\mathcal{E}}_{N}\) is lower semi-continuous by Fatou’s Lemma and Lemma \ref{lem:9}. The convergence of the minimizers \(\widetilde{\mu}^{N}\) to a minimizer \(\widetilde{\mu}\) of \(\widehat{\mathcal{E}}\) then follows by Lemma \ref{lem:8}. But \(\widetilde{\mu} = \omega\) because \(\omega\) is the unique minimizer of \(\widehat{\mathcal{E}}\).\qedhere
  \end{enumpara}
\end{proof}

\section{Moment bound in the symmetric case}
\label{sec:moment-bound-symm}

Let \( q_a = q_r \in (1,2) \) be strictly larger than \( 1 \) now. We want to prove that in this case, we have a stronger compactness than the one showed in Proposition \ref{prp:compctness-sublvls}, namely that the sub-levels of \( \widehat{\mathcal{E}} \) have a uniformly bounded \( r \)th moment for \( r < q/2 \).

In the proof, we shall be using the theory developed in Appendix \ref{cha:cond-posit-semi} in a more explicit form than before, in particular the notion of the generalized Fourier transform (Definition \ref{def:gen-fourier-transform}) and its computation in the case of the power function (Theorem \ref{thm:cond-ft-power}).

\begin{theorem}
  \label{thm:moment-bound-symmetric}
  Let \( \omega \in \mathcal{P}_2(\mathbb{R}^d) \). For \( r < q/2 \) and a given \( M \in \mathbb{R} \), there exists an \( M' \in \mathbb{R} \) such that
  \begin{equation}
    \label{eq:206}
    \int_{\mathbb{R}^d} \left| x \right|^{r} \diff \mu(x) \leq M', \quad \text{for all } \mu \text{ such that } \widehat{\mathcal{E}}[\mu] \leq M.
  \end{equation}
\end{theorem}

\begin{proof}
  Let \( \mu \in \mathcal{P}(\mathbb{R}^d) \). If \( \widehat{\mathcal{E}}[\mu] \leq M \), then we also have
  \begin{align}
    \label{eq:207}
    M \geq {} & \widehat{\mathcal{E}}[\mu] = D_q \int_{\mathbb{R}^d} \left| \widehat{\mu}(\xi) - \widehat{\omega}(\xi) \right|^2 \left| \xi \right|^{-d-q} \diff \xi\\
    \geq {} & c \int_{\mathbb{R}^d} \left| \widehat{\mu}(\xi) - 1 \right|^2 \left| \xi \right|^{-d-q} \diff \xi - \int_{\mathbb{R}^d} \left| \widehat{\omega}(\xi) - 1 \right|^2 \left| \xi \right|^{-d-q} \diff \xi,
  \end{align}
  so that there is an \( M'' > 0\) with
  \begin{equation}
    \label{eq:208}
    \int_{\mathbb{R}^d} \left| \widehat{\mu} - 1 \right|^2 \left| \xi \right|^{-d-q} \diff \xi \leq M''.
  \end{equation}
  Now approximate \( \mu \) by the sequence of Proposition \ref{prp:rec-seq}, denoting it by \( \mu_n \),
  \begin{equation}
    \label{eq:353}
    \mu_{n} := \widehat{\eta}_{n^{-1}}\cdot \mu + \left(1 - (\widehat{\eta}_{n^{-1}}\cdot \mu)(\mathbb{R}^{d})\right)\delta_{0},
  \end{equation}
 and then \( \mu_n \) by a Gaussian mollification with \( \eta_{k_n^{-1}} \) to obtain the diagonal sequence \( \mu_{n}' := \mu_n \ast \eta_{k_n^{-1}}\), so that we have convergence \( \widehat{\mathcal{E}}[\mu_n'] \rightarrow \widehat{\mathcal{E}}[\mu] \). We set \( \widehat{\nu}_n := (\mu_n' - \eta_{k_n^{-1}}) \).

Then, \( \widehat{\nu}_n \in \mathcal{S}(\mathbb{R}^d) \), the space of Schwartz functions: By the dampening of Proposition \ref{prp:rec-seq}, the underlying measures have finite moment of any order, yielding decay of \( \widehat{\nu_n}(x) \) of arbitrary polynomial order for \( \left| x \right| \to \infty \), and the mollification takes care of \( \widehat{\nu}_n \in C^\infty(\mathbb{R}^d) \). Furthermore, \( \nu_n = \widehat{\nu}^\vee_n \) and recall that the inverse Fourier transform can also be expressed as the integral of an exponential function. By expanding this exponential function in its power series, we see that for each \( n \),
  \begin{equation}
    \label{eq:209}
    \widehat{\nu}_n(\xi) = O(\left| \xi \right|) \quad \text{for } \xi \rightarrow 0,
  \end{equation}
  by the fact that \( \mu_n' \) and \( \delta_0 \) have the same mass, namely \( 1 \). Therefore, \( \widehat{\nu}_n \in \mathcal{S}_1(\mathbb{R}^d) \), see Definition \ref{def:restr-schwartz}, and we can apply Theorem \ref{thm:cond-ft-power}\ref{itm:cond-ft-power} to get
  \begin{align}
    \label{eq:210}
    \leadeq{\int_{\mathbb{R}^d} \left| x \right|^{r} \widehat{\nu}_n(x) \diff x}\\
    = {} & C \int_{\mathbb{R}^d} \left| \xi \right|^{-d-r} \nu_n(\xi) \diff \xi\\
    \leq {} & C \Bigg[ \int_{\left| \xi \right| \leq 1} \underbrace{\left| \xi \right|^{-d-r}}_{\mathclap{= \left| \xi \right|^{-\frac{d-q+2r}{2}} \left| \xi \right|^{-\frac{d + q}{2}}}} \left| \nu_n(\xi) \right| \diff \xi + \underbrace{\int_{\left| \xi \right| > 1} \left| \xi \right|^{-d-r} \left| \nu_n(\xi) \right| \diff \xi}_{\mathclap{\leq C < \infty}} \Bigg]\\
    \leq {} & C \Bigg[ \underbrace{\left( \int_{\left| \xi \right| \leq 1} \left| \xi \right|^{-d+(q - 2r)} \diff \xi \right)^{1/2}}_{\smash{< \infty}} \left( \int_{\mathbb{R}^d} \left| \xi \right|^{-d-q} \left| \nu_n \right|^2 \diff \xi \right)^{1/2} + 1 \Bigg]\\
    \leq {} & C \left[ \left( \int_{\mathbb{R}^d} \left| \xi \right|^{-d-q} \left| \nu_n \right|^2 \diff \xi \right)^{1/2} + 1 \right].
  \end{align}
  
  Now, we recall again the continuity of \( \widehat{\mathcal{E}} \) for \( \omega = \delta_0 \) along \( \mu_n \) (Proposition \ref{prp:rec-seq}) and its continuity \wrt the Gaussian mollification. The latter can be seen either by the \( 2 \)-Wasserstein-convergence of the mollification for \( n \) fixed or by using the Dominated Convergence Theorem together with the power series expansion of \( \exp \), similar to Lemma \ref{lem:11} below. In total, we see that
  \begin{equation}
    \label{eq:211}
    \lim_{n\rightarrow\infty} \int_{\mathbb{R}^d} \left| \xi \right|^{-d-q} \left| \nu_n \right|^2 \diff \xi = (2 \pi)^{-d} \int_{\mathbb{R}^d} \left| \xi \right|^{-d-q} \left| \widehat{\mu} - 1 \right|^2 \diff \xi \leq (2 \pi)^{-d} M'',
  \end{equation}
  while on the other hand we have
  \begin{align}
    \label{eq:212}
    \liminf_{n\rightarrow\infty} \int_{\mathbb{R}^d} \left| x \right|^{r} \widehat{\nu}_n(x) \diff x = {} & \liminf_{n\rightarrow\infty} \int_{\mathbb{R}^d} \left| x \right|^{r} \diff \mu_n(x) - \underbrace{\lim_{n\rightarrow\infty} \int_{\mathbb{R}^d} \left| x \right|^r \eta_{k_n^{-1}} (x) \diff x}_{=0}\\
    \geq {} & \int_{\mathbb{R}^d} \left| x \right|^{r} \diff \mu(x)
  \end{align}
  by Lemma \ref{lem:4}, concluding the proof.
\end{proof}

\section{Regularization by using the total variation}
\label{sec:tv-reg}

We would like to regularize the functional \( \widehat{\mathcal{E}} \) by an additional total variation term, for example to reduce the possible effect of  noise in the given datum \( \omega \). In particular, we expect the minimizer of the corresponding functional to be piecewise smoothed or even constant while any sharp edges in \( \omega \) should be preserved, as it is the case for the regularization of a quadratic fitting term, see for example \cite[Chapter 4]{10_Chambolle_ea_tv_intro}.

In the following, we begin by introducing this regularization and prove that for a vanishing regularization parameter, the minimizers of the regularizations converge to the minimizer of the original functional. One effect of the regularization will be to allow us to consider approximating or regularized minimizers of \( \widehat{\mathcal{E}}[\mu] \) in \( \mathcal{P}(\mathbb{R}^d) \cap BV(\mathbb{R}^d) \), where \( BV(\mathbb{R}^d) \) is the space of bounded variation functions. In the classical literature, one finds plenty of discrete approximations to \( BV \)-minimizers of functionals including total variation terms, by means of finite element type approximations of the functions, see for example \cite{12-Bartels-TotalVariation}. Here however, we propose an approximation which depends on the position of (freely moving) particles in \( \mathbb{R}^d \), which can be combined with the particle approximation of Section \ref{sec:part-appr}. To this end, in Section \ref{sec:discrete-version-tv-1}, we shall present two ways of embedding the Dirac masses which are associated to particles into \( L^1 \).

\subsection{Consistency of the regularization for the continuous functional}
\label{sec:cons-regul-cont}

For \(\mu \in \mathcal{P}(\mathbb{R}^{d})\), define
\begin{equation}
  \label{eq:135}
  \widehat{\mathcal{E}}^{\lambda}[\mu] :=
  \begin{cases}
    \widehat{\mathcal{E}}[\mu] + \lambda \left| D\mu \right|(\mathbb{R}^d), &\mu \in \mathcal{P}(\mathbb{R}^{d}) \cap BV(\mathbb{R}^{d}),\\
    \infty, &\text{otherwise},
  \end{cases}
\end{equation}
where \(D\mu\) denotes the distributional derivative of \(\mu\) (being a finite Radon-measure) and \(\left| D \mu \right|(\mathbb{R}^d)\) its total variation. We present two easy lemmata before proceeding to prove the \( \Gamma \)-convergence \( \widehat{\mathcal{E}}^\lambda \xrightarrow{\Gamma} \widehat{\mathcal{E}} \).

\begin{lemma}
  [Continuity of \(\widehat{\mathcal{E}}\) \wrt Gaussian mollification]
  \label{lem:11}
  Let \( \omega \in \mathcal{P}_2(\mathbb{R}^d) \), \(\mu \in \mathcal{P}(\mathbb{R}^{d})\) and set
  \begin{equation}
    \label{eq:136}
    \eta(x) := (2\pi)^{-d/2}\exp\left(-\frac{1}{2}\left| x \right|^{2}\right), \quad \eta_{\varepsilon}(x) := \varepsilon^{-d}\eta(\varepsilon^{-1}x), \quad x \in \mathbb{R}^{d}.
  \end{equation}
  Then,
  \begin{equation}
    \label{eq:137}
    \widehat{\mathcal{E}}[\eta_{\varepsilon}\ast \mu] \rightarrow \widehat{\mathcal{E}}[\mu], \quad \text{for } \varepsilon \rightarrow 0.
  \end{equation}
\end{lemma}

\begin{proof}
  If \(\widehat{\mathcal{E}}[\mu] = \infty\), then the claim is true by the lower semi-continuity of \(\widehat{\mathcal{E}}\) together with the fact that \(\eta_{\varepsilon} \ast \mu \rightarrow \mu\) narrowly.

  If \(\widehat{\mathcal{E}}[\mu] < \infty\), we can estimate the difference \(\left|\widehat{\mathcal{E}}[\eta_{\varepsilon}\ast \mu]-\widehat{\mathcal{E}}[\mu]\right|\) (which is well defined, but for now may be \(\infty\)) by using
  \begin{equation}
    \label{eq:138}
    \left| a^{2} - b^{2} \right| \leq \left| a - b \right| \cdot \big( \left| a \right| + \left| b \right|\big), \quad a,b \in \mathbb{\mathbb{C}}
  \end{equation}
  and
  \begin{equation}
    \label{eq:139}
    \widehat{\eta_{\varepsilon}\ast\mu}(\xi) = \exp\left( -\frac{\varepsilon^2}{2} \left| \xi \right|^{2} \right) \widehat\mu(\xi)
  \end{equation}
  as
  \begin{align}
    \label{eq:140}
    \leadeq{\left|\widehat{\mathcal{E}}[\eta_{\varepsilon}\ast \mu]-\widehat{\mathcal{E}}[\mu]\right|} \\
    \leq {} & D_{q}\int_{\mathbb{R}^{d}} \left| \left| \widehat\eta_{\varepsilon}(\xi) \widehat\mu(\xi) - \widehat\omega(\xi) \right|^{2} - \left| \widehat\mu(\xi) - \widehat\omega(\xi) \right|^{2}\right| \left| \xi \right|^{-d-q} \diff \xi\\
    \leq {} & D_{q}\int_{\mathbb{R}^{d}} \underbrace{\left(\left| \widehat\eta_{\varepsilon}(\xi) \widehat\mu(\xi) - \widehat\omega(\xi) \right| + \left| \widehat\mu(\xi) - \widehat\omega(\xi) \right| \right)}_{\leq 4} \underbrace{\left| \widehat\eta_{\varepsilon}(\xi) \widehat\mu(\xi) - \widehat\mu(\xi) \right|}_{\mathclap{=\left( 1 - \exp\left( -(\varepsilon^2/2)\left| \xi \right|^{2} \right) \right) \widehat\mu(\xi)}} \left| \xi \right|^{-d-q} \diff \xi\\
    \leq {} & C\int_{\mathbb{R}^{d}}\left( 1 - \exp\left( - \frac{\varepsilon^2}{2}\left| \xi \right|^{2} \right) \right) \left| \xi \right|^{-d-q}\diff \xi,\label{eq:141}
  \end{align}
  which converges to \(0\) by the Dominated Convergence Theorem: On the one hand we can estimate
  \begin{equation}
    \label{eq:142}
    \exp\left( -\frac{\varepsilon^2}{2}\left| \xi \right|^{2} \right) \geq 0, \quad \xi \in \mathbb{R}^{d},
  \end{equation}
  yielding a dominating function for the integrand in \eqref{eq:141} for \(\xi\) bounded away from \(0\) because of the integrability of \( \xi \mapsto \left| \xi \right|^{-d-q} \) there. On the other hand
  \begin{align}
    \label{eq:143}
    1 - \exp\left( -\frac{\varepsilon^2}{2}\left| \xi \right|^{2} \right) {} = & -\sum_{n = 1}^{\infty} \frac{1}{n!} \left(-\frac{\varepsilon^2}{2}\left| \xi \right|^{2}\right)^{n}\\
    = {} &\frac{\varepsilon^2}{2}\left| \xi \right|^{2} \sum_{n = 0}^{\infty} \frac{1}{(n+1)!} \left(-\frac{\varepsilon^2}{2}\left| \xi \right|^{2}\right)^{n},
  \end{align}
  where the sum on the right is bounded for \(\varepsilon\rightarrow 0\) as a convergent power-series, which combined with \(q < 2\) renders the integrand in \eqref{eq:141} dominated for \(\xi\) near \(0\) as well.
\end{proof}

  \begin{lemma}[Product formula for \( BV(\mathbb{R}^d) \)]
  \label{lem:14}
  Let \( f \in BV(\mathbb{R}^d) \) and \( \varphi \in C^\infty_c(\mathbb{R}^d) \). Then the distributional derivative of \( f\varphi \) is
  \begin{equation}
    \label{eq:183}
    D(f\varphi) = \varphi Df + f \nabla \varphi.
  \end{equation}
\end{lemma}

\begin{proof}
  Let \( g \in C^\infty_c(\mathbb{R}^d) \). Then,
  \begin{align}
    \label{eq:184}
    \int_{\mathbb{R}^d} f(x) \varphi(x) \nabla g(x) \diff x = {} & \int_{\mathbb{R}^d} f(x) \nabla (\underbrace{\varphi g}_{\mathclap{\in C^\infty_c(\mathbb{R}^d)}})(x) \diff x - \int_{\mathbb{R}^d} f(x) \nabla \varphi(x) g(x) \diff x\\
    = {} & -\int_{\mathbb{R}^d} g(x) \varphi(x) \diff Df(x) - \int_{\mathbb{R}^d} g(x) f(x) \nabla \varphi(x) \diff x,
  \end{align}
  proving that in a distributional sense, \( D(f\varphi) = \varphi Df + f \nabla \varphi \).
\end{proof}

\begin{proposition}
  [Consistency]
  \label{prp:consist-cont-bv}
  The functionals \((\widehat{\mathcal{E}}^{\lambda})_{N \in \mathbb{N}}\) are equi-coercive and
  \begin{equation}
    \label{eq:144}
    \widehat{\mathcal{E}}^{\lambda} \xrightarrow{\Gamma} \widehat{\mathcal{E}} \quad \text{for } \lambda\rightarrow 0
  \end{equation}
  with respect to the narrow topology. In particular,
  \begin{equation}
    \label{eq:145}
    \argmin_{\mu \in \mathcal{P}(\mathbb{R}^{d})} \widehat{\mathcal{E}}^{\lambda}[\mu] \ni \mu^{\lambda} \rightarrow \omega = \argmin_{\mu \in \mathcal{P}(\mathbb{R}^{d})} \widehat{\mathcal{E}}[\mu], \quad \lambda \to 0,
  \end{equation}
  for each choice of minimizers \(\mu^{\lambda}\).
\end{proposition}

\begin{proof}
  Firstly, observe that equi-coercivity follows from the narrow compactness of the sub-levels of \(\widehat{\mathcal{E}}\) (Proposition \ref{prp:compctness-sublvls}) and that the \( \liminf \)-condition is a consequence of the lower semi-continuity of \(\widehat{\mathcal{E}}\) as in the proof of Theorem \ref{thm:cons-part-appr}.

  \emph{Ad existence of minimizers:} We again want to apply the direct method of the calculus of variations.

  Let \( (\mu_k)_k \) be a minimizing sequence for \( \widehat{\mathcal{E}}^\lambda \), so that the \( \mu_k \) are all contained in a common sub-level of the functional. Now, for a given \( \lambda \), the sub-levels of \( \widehat{\mathcal{E}}^\lambda \) are relatively compact in \( L^1(\mathbb{R}^d) \), which can be seen by combining the compactness of the sub-levels of the total variation in \( L^1_{\text{loc}}(\mathbb{R}^d) \) with the tightness gained by \( \widehat{\mathcal{E}} \): If \( \widehat{\mathcal{E}}^\lambda[\mu_k] \leq M < \infty \), we can consider \( (\theta_l \mu_k)_k \) for a smooth cut-off function \( \theta_l \) having its support in \( [-l^{-1}, l^{-1}] \). By Lemma \ref{lem:14}, we have
  \begin{equation}
    \label{eq:186}
    D\left(\theta_l \mu_k\right) = D \theta_l \mu_k + \theta_l D \mu_k
  \end{equation}
  and therefore
  \begin{align}
    \label{eq:187}
    \left| D \left( \theta_l \mu_k \right) \right| (\mathbb{R}^d) \leq {} &  \int_{\mathbb{R}^d} \mu_k(x) \left| D\theta_l (x) \right| \diff x + \int_{\mathbb{R}^d}  \theta_l (x) \diff \left| D \mu_k \right| (x)\\
    \leq {} & C_l + \left| D\mu_k \right|(\mathbb{R}^d),
  \end{align}
  so that for each \( l \), by the compactness of the sub-levels of the total variation in \( L^1_{\text{loc}} \), see \cite[Chapter 5.2, Theorem 4]{92-Evans-fine-properties}, we can select an \( L^{1} \)-convergent subsequence \( (\theta_l \mu_{k(l,i)})_i \). Then, we can choose a diagonal sequence \( (\mu_{k(l,i(l))})_l \) (for which we just write \( (\mu_{k(l)}) \)) such that it is a Cauchy sequence and therefore convergent by the completeness of \( L^1 \):
  \begin{align}
    \label{eq:188}
    \left\| \mu_{k(m)} - \mu_{k(l)} \right\|_{L^1} \leq {} & \left\| (1-\theta_{l}) (\mu_{k(m)} - \mu_{k(l)}) \right\|_{L^1} + \left\| \theta_{l} \mu_{k(m)} - \theta_{l} \mu_{k(l)} \right\|_{L^1}\\
    =: {} &d_1(m,l) + d_2(m,l),
  \end{align}
  where the \( i(l) \) can be chosen such that \(d_2 \rightarrow 0\) for \( \min \left\{ m,l \right\} \to \infty \) by the selection of the \(\mu_{k(l,i)} \) as convergent sequences and the fact that \( \theta_l(x) \) is increasing for all \( x \in \mathbb{R}^d \), and \( d_1 \rightarrow 0\) for \( \min \left\{ m,l \right\} \to \infty \) because of the tightness of the sub-levels of \( \widehat{\mathcal{E}} \).

  The lower semi-continuity of \( \widehat{\mathcal{E}}^\lambda \) follows from the lower semi-continuity of the total variation with respect to \( L^{1} \)-convergence and the lower semi-continuity of \( \widehat{\mathcal{E}} \) with respect to narrow convergence (which by Lemma \ref{lem:1} is weaker than \( L^1 \)-convergence). Summarizing, we have compactness and lower semi-continuity, giving us that \( (\mu_k)_k \) has a limit point which is a minimizer.

  \emph{Ad \( \limsup \)-condition:} Let \(\mu \in \mathcal{P}(\mathbb{R}^{d})\) and write \(\mu_{\varepsilon} := \widehat{\eta}_{\varepsilon}\ast\mu\) for the mollification of Lemma \ref{lem:11}. Now, by Fubini’s Theorem,
  \begin{align}
    \label{eq:146}
    \left| D (\widehat{\eta}_{\varepsilon}\ast \mu) \right|(\mathbb{R}^d) = {} &\int_{\mathbb{R}^{d}} \left| \left(\nabla \widehat{\eta}_{\varepsilon}\ast \mu\right)(x) \right| \diff x \\
    \leq {} & \left\| \nabla \widehat{\eta}_{\varepsilon} \right\|_{L^{1}(\mathbb{R}^{d})} \mu(\mathbb{R}^{d})\\
    = {} & \varepsilon^{-d} \left\| \nabla\widehat{\eta} \right\|_{L^{1}(\mathbb{R}^{d})},
  \end{align}
  so if we choose \(\varepsilon(\lambda)\) such that \(\lambda = o(\varepsilon^{d})\), for example \(\varepsilon(\lambda) := \lambda^{1/(d+1)}\), then
  \begin{equation}
    \label{eq:147}
    \lambda \left| D\mu_{\varepsilon(\lambda)} \right|(\mathbb{R}^d) \rightarrow 0, \quad \text{for } \lambda \rightarrow 0.
  \end{equation}
  On the other hand, \(\widehat{\mathcal{E}}[\mu_{\varepsilon(\lambda)}] \rightarrow \widehat{\mathcal{E}}[\mu]\) by Lemma \ref{lem:11}, yielding the required convergence \(\widehat{\mathcal{E}}^{\lambda}[\mu_{\varepsilon(\lambda)}] \rightarrow \widehat{\mathcal{E}}[\mu]\).

  The convergence of the minimizers then follows by Lemma \ref{lem:8}
\end{proof}

\subsection{Discrete versions of the TV regularization}
\label{sec:discrete-version-tv-1}

As one motivation for the functional \( \mathcal{E} \) was to compute its particle minimizers, we would also like to consider a discretized version of the total variation regularization, for example to be able to compute the minimizers of the functional directly on the level of the point approximations. We propose two techniques for this discretization: 

The first technique is well known in the non-parametric estimation of \( L^1 \) densities and consists of replacing each point with a small ``bump'' instead of interpreting it as a point measure. In order to get the desired convergence properties, we have to be careful when choosing the corresponding scaling of the bump. For an introduction to this topic, see \cite[Chapter 3.1]{85-Devroye-Gyoerfi-non-param-den-est}.

The second technique replaces the Dirac deltas by indicator functions which extend from the position of one point to the next one. Unfortunately, this poses certain difficulties in generalizing it to higher dimensions, as the set on which we extend would have to be replaced by something like a Voronoi cell, an object well-known in the theory of optimal quantization of measures, see for example \cite{00_Graf_Luschgy_quantization}.

Note that approximating the total variation regularization in this way in general unfortunately will not be computationally efficient due to the lack of the convexity of the regularization functional (see also Section \ref{sec:numer-exper} for some numerical examples). However, in the context of attraction-repulsion functionals, it is worth noting that the effect of the additional particle total variation term can again be interpreted as an attractive-repulsive-term. See Figure \ref{fig:discrete-tv} for an example in the case of kernel density estimation with a piecewise linear estimation kernel, where it can be seen that each point is repulsive at a short range, attractive at a medium range, and at a long range does not factor into the total variation any more.

\begin{figure}[t]
  \centering
  \subfigure[Linear \( K_1(x) \) and corresponding \( K_1'(x) \) ]{\includegraphics[]{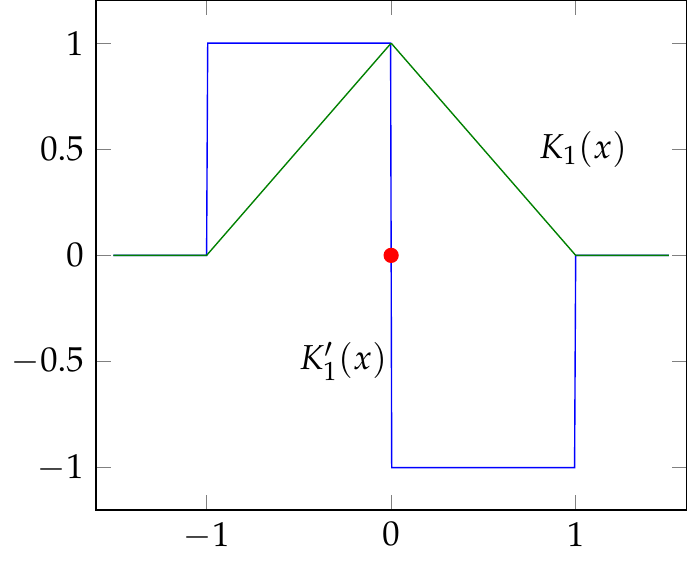}}
  \hspace{0.5cm}
  \subfigure[Discrete total variation for \( x_1 = 0 \) fixed (red), \( x_2 \) free, \( h = 0.75 \)]{\includegraphics[]{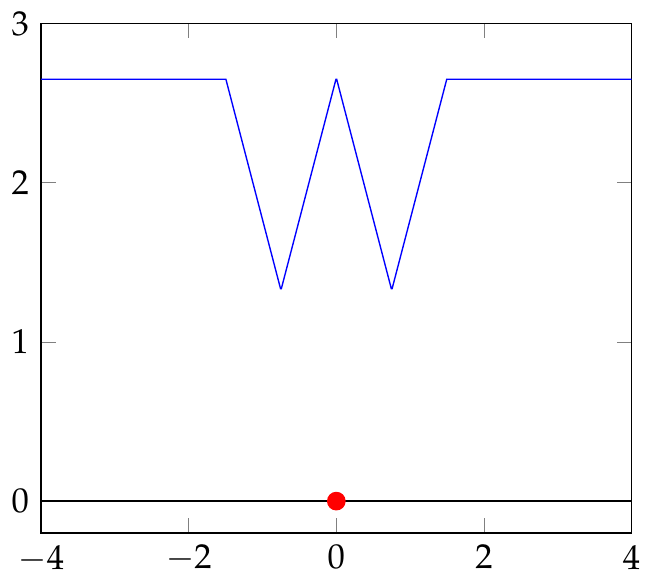}}   \caption{Example for the discrete total variation functional}
  \label{fig:discrete-tv}
\end{figure}

\subsubsection{Discretization by kernel estimators}
\label{sec:discr-kern-estim}

\begin{definition}
  [Discrete total variation via kernel estimate]
  \label{def:dens-est}
  For a \(\mu_{N} = \frac{1}{N}\sum_{i=1}^{N} \delta_{x_{i}} \in \mathcal{P}^{N}(\mathbb{R}^{d})\), a \emph{scale parameter} \(h = h(N)\) and a \emph{density estimation kernel} \( K \in W^{1,1}(\mathbb{R}^{d}) \) such that \( \nabla K \in BV(\mathbb{R}^d, \mathbb{R}^d) \), as well as
  \begin{equation}
    \label{eq:148}
    K \geq 0, \quad \int_{\mathbb{R}^{d}}K(x) \diff x = 1,
  \end{equation}
  we set
  \begin{equation}
    \label{eq:149}
    K_{h}(x) := \frac{1}{h^{d}}K\left(\frac{x}{h}\right)
  \end{equation}
  and define the corresponding \emph{\(L^{1}\)-density estimator} by
  \begin{equation}
    \label{eq:150}
    Q_h[\mu_N](x) := K_{h}\ast\mu_{N}(x) = \frac{1}{N h^{d}}\sum_{i=1}^{N} K\left(\frac{x-x_{i}}{h}\right),
  \end{equation}
  where the definition has to be understood for almost every \( x \). Then, we can introduce a discrete version of the regularization in \eqref{eq:135} as
  \begin{equation}
    \label{eq:151}
                        \widehat{\mathcal{E}}_{N}^{\lambda}[\mu_N] := \widehat{\mathcal{E}}[\mu_N] + \lambda \left| DQ_{h(N)}[\mu_N]\right|(\mathbb{R}^d), \quad \mu_N \in \mathcal{P}^N(\mathbb{R}^d).
  \end{equation}
            \end{definition}

We want to prove consistency of this approximation in terms of \( \Gamma \)-convergence of the functionals \( \widehat{\mathcal{E}}_N^\lambda \) to \( \widehat{\mathcal{E}}^\lambda \). For a survey of the consistency of kernel estimators in the probabilistic case under various sets of assumptions, see \cite{12-Wied-Weissbach-Kernel-Estimators}. Here however, we want to give a proof using deterministic and explicitly constructed point approximations.

\begin{figure}[t]
  \centering
  \subfigure[Notation of Definition \ref{def:good-tiling} for \( N = 5 \) ]{\includegraphics{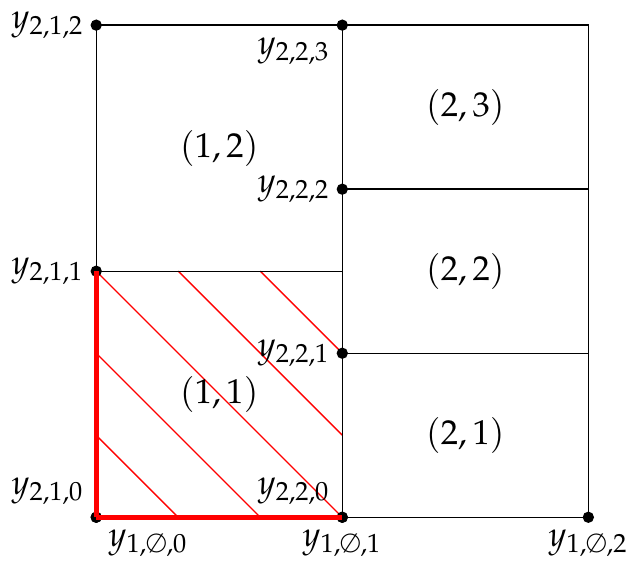}}
  \hspace{0\textwidth}
  \subfigure[{Tiling as in Example \ref{exp:constr-2d} for a uniform measure on two squares in \( [0,1]^2 \)}]{\includegraphics{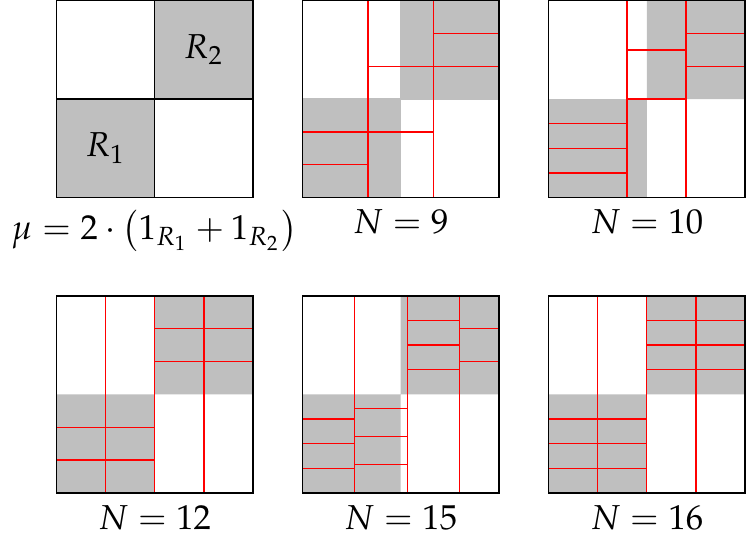}}
  \caption{Illustration of the tiling}
  \label{fig:tiling}
\end{figure}

In order to find a recovery sequence for the family of functionals \eqref{eq:151}, we have to find point approximations to a given measure with sufficiently good spatial approximation properties. For this, we suggest using a generalization of the quantile construction to higher dimensions. Let us state the properties we expect from such an approximation:

\begin{definition}
  [Tiling associated to a measure]
  \label{def:good-tiling}
  Let \( \mu \in \mathcal{P}_c(\mathbb{R}^d)\cap L^1(\mathbb{R}^d) \), where \( \mathcal{P}_c(\mathbb{R}^d) \) denotes the space of compactly supported probability measures, such that \( \operatorname{supp}(\mu) \subseteq [-R_\mu,R_\mu]^d \) and let \( N \in \mathbb{N} \). Set \( \widetilde{n} := \lfloor N^{1/d} \rfloor \). A good tiling (for our purposes) will be composed of an index set \( I \) and an associated tiling \( \left( T_i \right)_{i \in I} \) such that (see Figure \ref{fig:tiling} for an example of the notation):

  \begin{enumerate}
  \item \( I \) has \( N \) elements, \( \# I = N \), and in each direction, we have at least \( \widetilde{n} \) different indices, \ie,
    \begin{equation}
      \label{eq:155}
      \left\{ 1,\ldots,\widetilde{n} \right\}^d \subseteq I \subseteq \left\{ 1,\ldots,\widetilde{n}+1 \right\}^d.\\
    \end{equation}
    Additionally, for all \( k \in {1,\ldots,d} \) and \( (i_1,\ldots,i_{k-1},i_k,\ldots,i_d) \in I \),
    \begin{equation}
      \label{eq:156}
      n_{k,i_1,\ldots,i_{k-1}} := \# \left\{ j_k : j \in I, \, (j_1,\ldots,j_{k-1}) = (i_1,\ldots,i_{k-1}) \right\} \in \left\{ \widetilde{n},\widetilde{n}+1 \right\}.
    \end{equation}
  \item There is a family of ordered real numbers only depending on the first \( k \) coordinates,
    \begin{gather}
      \label{eq:157}
      y_{k,i_1,\ldots,i_k} \in [-R_\mu,R_\mu], \quad y_{k,i_1,\ldots,i_k-1} < y_{k,i_1,\ldots,i_k},\\
      \text{for all } k \in \left\{1,\ldots,d\right\} \text{ and } (i_1,\ldots,i_k,i_{k+1},\ldots,i_d) \in I, \nonumber
    \end{gather}
    with fixed end points,
    \begin{equation}
      \label{eq:372}
      y_{k,i_1,\ldots,i_{k-1},0}=-R_{\mu}, \quad y_{k,i_1,\ldots,i_{k-1},n_{k,i_1,\ldots,i_{k-1}}} = R_{\mu},
    \end{equation}
    associated tiles
    \begin{equation}
      \label{eq:373}
      T_i := \bigtimes_{k=1}^{d} \left[y_{k,i_1,\ldots,(i_k-1)},y_{k,i_1,\ldots,i_k}\right],
    \end{equation}
    and such that the mass of \( \mu \) is equal in each of them,
    \begin{equation}
      \label{eq:374}
      \mu \left( T_i \right) = \frac{1}{N}, \quad \text{for all }i \in I.
    \end{equation}
  \end{enumerate}
\end{definition}

Such a construction can always be found by generalizing the quantile construction. Let us show the construction explicitly for \( d = 2 \) as an example.

\begin{example}
  [Construction in 2D]
  \label{exp:constr-2d}
  Given \(N \in \mathbb{N}\), let \(\widetilde{n} := \lfloor \sqrt{N} \rfloor\). We can write \( N \) as
  \begin{equation}
    \label{eq:159}
    N = \widetilde{n}^{2-m}\left( \widetilde{n} +1 \right)^m + l,
  \end{equation}
with unique \( m \in \left\{ 0, 1 \right\} \) and \( l \in \left\{ 0,\ldots,\widetilde{n}^{1-m} \left( \widetilde{n} + 1 \right)^m - 1 \right\}. \)
  Then we get the desired tiling by setting
  \begin{alignat}{2}
    \label{eq:160}
    n_{1,\emptyset} := {} &
    \begin{cases}
      \widetilde{n} + 1 \quad &\text{if } m = 1,\\
      \widetilde{n} &\text{if } m = 0,
    \end{cases}\\
    n_{2,i_1} := {} &
    \begin{cases}
      \widetilde{n}+1 \quad &\text{if } i_1 \leq l,\\
      \widetilde{n} &\text{if } i_1 \geq l + 1,
    \end{cases} \quad && i_1 = 1,\ldots,n_{1,\emptyset},\label{eq:375}\\
    w_{2,i_1,i_2} := {} &\frac{1}{n_{2,i_1}}, \quad && i_1 = 1,\ldots,n_{1,\emptyset},\; i_2 = 1,\ldots,n_{2,i_1},\\
    w_{1,i_1} := {} &\frac{n_{2,i_1}}{\sum_{j_1} n_{2,j_1}}, \quad &&i_1 = 1,\ldots,n_{1,\emptyset},
  \end{alignat}
  and choosing the end points of the tiles such that
  \begin{align}
    \label{eq:161}
    \sum_{j_1 = 1}^{i_1} w_{1,j_1} = {} &\int_{-R_\mu}^{y_{1,i_1}} \int_{-R_\mu}^{R_\mu} \diff \mu(x_1,x_2),\\
    \sum_{j_1 = 1}^{i_1} \sum_{j_2 = 1}^{i_2} w_{1,j_1} w_{2,j_1,j_2} = {} &\int_{-R_\mu}^{y_{1,i_1}} \int_{-R_\mu}^{y_{2,i_1,i_2}} \diff \mu(x_1,x_2).
  \end{align}
  Now, check that indeed \(\sum_{j_1} n_{2,j_1} = N\) by \eqref{eq:159} and \eqref{eq:375} and that we have
  \begin{equation}
    \label{eq:162}
    \mu(T_{i_1,i_2}) = w_{1,i_1} w_{2,i_1,i_2} = \frac{1}{N} \quad \text{for all } i_1, i_2,
  \end{equation}
  by the choice of the weights \( w_{1,j_1} \), \( w_{2,j_1,j_2} \) as desired.
\end{example}

The general construction now consists of choosing a subdivision in \( \widetilde{n}+1 \) slices uniformly in as many dimensions as possible, while keeping in mind that in each dimension, we have to subdivide in at least \( \widetilde{n} \) slices. There will again be a rest \( l \), which is filled up in the last dimension.

\begin{proposition}
  [Construction for arbitrary \( d \)]
  \label{prp:constr-gen-dim}
  A tiling as defined in Definition \ref{def:good-tiling} exists for all \( d \in \mathbb{N} \).
\end{proposition}

\begin{proof}
  Analogously to Example \ref{exp:constr-2d}, let \( \widetilde{n}:= \lfloor N^{1/d} \rfloor \) and set
  \begin{equation}
    \label{eq:163}
    N = \widetilde{n}^{d-m}\left( \widetilde{n} +1 \right)^m + l,
  \end{equation}
  with unique \( m \in \left\{ 0, \ldots, d-1 \right\} \) and \( l \in \left\{ 0,\ldots,\widetilde{n}^{d-1-m} \left( \widetilde{n} + 1 \right)^m - 1 \right\} \). Then, we get the desired ranges by
  \begin{alignat}{2}
    \label{eq:164}
    n_{k,i_1,\ldots,i_{k-1}} &{}:= \widetilde{n} + 1, &&\text{for } k \in \left\{ 1,\ldots,m \right\} \text{ and all relevant indices};\\     n_{k,i_1,\ldots,i_{k-1}}, &{}:= \widetilde{n}, &&\text{for } k \in \left\{ m+1,\ldots,d-1 \right\} \text{ and all relevant indices};\\
    n_{d,i_1,\ldots,i_{d-1}} &{}\in \left\{ \widetilde{n}, \widetilde{n}+1 \right\}, \quad&&\text{such that exactly \( l \) multi-indices are } \widetilde{n} + 1.
  \end{alignat}
  The weights can then be selected such that we get equal mass after multiplying them, and the tiling is found by iteratively using a quantile construction similar to \eqref{eq:161} in Example \ref{exp:constr-2d}.
\end{proof}

\begin{lemma}
  [Consistency of the approximation]
  \label{lem:13}
  For \(\mu\in \mathcal{P}_c(\mathbb{R}^d) \cap BV(\mathbb{R}^d)\), let \( (T_i)_{i \in I} \) be a tiling as in Definition \ref{def:good-tiling}, and \( x_i \in T_i \) for all \( i \in I \) an arbitrary point in each tile. Then, \(\mu_N = \frac{1}{N} \sum_{i=1}^{N} \delta_{x_i}\) converges narrowly to \(\mu\) for \(N \rightarrow\infty\).
  Furthermore, if
  \begin{equation}
    \label{eq:172}
    h = h(N)\rightarrow 0 \quad \text{and} \quad h^{2d}N \rightarrow \infty \quad \text{for } N\rightarrow \infty,
  \end{equation}
  then \(Q_{h(N)}[\mu_N] \rightarrow\mu\) strictly in \(BV(\mathbb{R}^d)\) (as defined in \cite[Definition 3.14]{00-Amb-Fusco-Pallara-BV}).
\end{lemma}

\begin{proof}
  Suppose again that
  \begin{equation}
    \label{eq:173}
    \operatorname{supp} \mu \subseteq [-R_{\mu},R_{\mu}]^d.
  \end{equation}
  
  \emph{Ad narrow convergence:}
  By \cite[Theorem 3.9.1]{Dur10}, it is sufficient to test convergence for bounded, Lipschitz-continuous functions. So let \(\varphi \in C_b(\mathbb{R}^d) \) be Lipschitz with constant \(L\). Then,
  \begin{align}
    \label{eq:174}
    \leadeq{\left| \int_{\mathbb{R}^d} \varphi(x) \diff \mu_N (x) - \int_{\mathbb{R}^d} \varphi(x) \diff \mu(x) \right|}\\
    = {} & \left| \frac{1}{N} \sum_{i=1}^{N} \varphi(x_i) - \int_{\mathbb{R}^d} \varphi(x) \diff \mu(x) \right|\\
    \leq {} & \sum_{i \in I} \int_{T_i} \left| \varphi(x) - \varphi(x_i) \right| \diff \mu(x)\\
    \leq {} & L \sum_{i \in I} \int_{T_i} \left| x - x_i \right| \diff \mu(x).
  \end{align}
  Denote by
  \begin{equation}
    \label{eq:175}
    \widehat{\pi}_k (i_1,\ldots,i_d) := (i_1,\ldots,i_{k-1},i_{k+1},i_d)
  \end{equation}
  the projection onto all coordinates except the \( k \)th one. Now, we exploit the uniformity of the tiling in all dimensions, \eqref{eq:155}: By using the triangular inequality and grouping the summands,
  \begin{align}
    \leadeq{\sum_{i \in I} \int_{T_i} \left| x - x_i \right| \diff \mu(x)} \\
    \leq {} & \sum_{i \in I} \sum_{k=1}^{d} \int_{T_i} \left| x^k - x_i^k \right| \diff \mu(x)\label{eq:176}\\
    = {} & \sum_{k=1}^{d} \sum_{i \in \widehat{\pi}_k(I)} \sum_{j=1}^{n_{k,i_1,\ldots,i_{k-1}}} \int_{T_i} \left| x^k - x_{i_1,\ldots,i_{k-1},j,i_k,\ldots,i_{d-1}}^k \right| \diff \mu(x)\\
    \leq {} & \sum_{k=1}^{d} \sum_{i \in \widehat{\pi}_k(I)} \underbrace{\sum_{j=1}^{n_{k,i_1,\ldots,i_{k-1}}} \left( y_{k,i_1,\ldots,i_{k-1},(j-1)}-y_{k,i_1,\ldots,i_{k-1},j} \right)}_{=2R_\mu} \underbrace{\int_{T_i} \diff \mu(x)}_{=1/N}\\
    \leq {} & 2R_{\mu}\,d \frac{\left(\widetilde{n}+1\right)^{d-1}}{N} \leq 2R_{\mu}\,d \frac{\left(\widetilde{n}+1\right)^{d-1}}{\widetilde{n}^{d}} \leq \frac{C}{\widetilde{n}} \rightarrow 0 \quad \text{for } N\rightarrow\infty.\label{eq:376}
  \end{align}
  
  \emph{Ad \(L^1\)-convergence:} As \( K \in W^{1,1}(\mathbb{R}^d) \subseteq BV(\mathbb{R}^d)\), we can approximate it by \( C^1 \) functions which converge \( BV \)-strictly, so let us additionally assume \( K \in C^1 \) for now. Then,
  \begin{align}
    \label{eq:177}
    \leadeq{\int_{\mathbb{R}^d} \left| K_h \ast \mu_N(x) - \mu(x) \right| \diff x}\\
    \leq {} & \int_{\mathbb{R}^d} \left| K_h \ast \mu_N(x) - K_h \ast \mu (x) \right| \diff x + \int_{\mathbb{R}^d} \left| K_h \ast \mu(x) - \mu(x) \right| \diff x.
  \end{align}
  By \(h \rightarrow 0\), the second term goes to \( 0 \) (see \cite[Chapter 5.2, Theorem 2]{92-Evans-fine-properties}), so it is sufficient to consider
  \begin{align}
    \leadeqnum{\int_{\mathbb{R}^d} \left| K_h \ast \mu_N(x) - K_h \ast \mu(x)  \right| \diff x}\label{eq:178}\\
    \leq {} & \sum_{i \in I} \int_{T_i} \int_{\mathbb{R}^d} \left| K_h(x - x_i) - K_h(x-y) \right| \diff x \diff \mu(y)\\
    = {} & \sum_{i \in I} \int_{T_i} \int_{\mathbb{R}^d} \left| \int_0^1 \nabla K_h(x - y + t (y - x_i)) \cdot (y - x_i) \diff t \right| \diff x \diff \mu(y)\\
    \leq {} & \sum_{i \in I} \int_{T_i} \int_0^1 \left| y - x_i \right| \int_{\mathbb{R}^d} \left| \nabla K_h(x - y + t (y - x_i)) \right| \diff x \diff t \diff \mu(y)\\
    = {} & \frac{1}{h} \left\| \nabla K \right\|_{L^1} \sum_{i \in I} \int_{T_i}  \left| y - x_i \right| \diff \mu(y).\label{eq:179}
  \end{align}
  Since the left-hand side \eqref{eq:178} and the right-hand side \eqref{eq:179} of the above estimate are continuous with respect to strict BV convergence (by Fubini-Tonelli and convergence of the total variation, respectively), this estimate extends to a general \( K \in BV(\mathbb{R}^d) \) and
  \begin{equation}
    \label{eq:180}
    \frac{1}{h} \sum_{i \in I} \int_{T_i}  \left| y - x_i \right| \diff \mu(y) \leq \frac{C}{\widetilde{n}h} \rightarrow 0, \quad \text{for } \mathbb{N} \rightarrow \infty,
  \end{equation}
                  by the calculation in \eqref{eq:176} and condition \eqref{eq:172}.

  \emph{Ad convergence of the total variation:} Similarly to the estimate in \eqref{eq:177}, by \( h \rightarrow 0 \) it is sufficient to consider the \(L^1\) distance between \( \nabla K_h \ast \mu_N \) and \( \nabla K_h \ast \mu \) if we approximate a general \( K \) with a \( K \in C^{2}(\mathbb{R}^d) \). By a calculation similar to \eqref{eq:178} -- \eqref{eq:179} as well as \eqref{eq:376} and using \( \nabla K_h(x) = h^{-d-1} K(x/h) \), we get
  \begin{align}
    \label{eq:182}
    \leadeq{\int_{\mathbb{R}^d} \left| \nabla K_h \ast \mu_N(x) - \nabla K_h \ast \mu(x)  \right| \diff x}\\
    \leq {} & C \frac{1}{h} \sum_{i \in I} \int_{T_i} \int_{\mathbb{R}^d} \left| \nabla K_h(x-x_i) - \nabla K_h(x-y) \right| \diff x \diff \mu(y)\\
    \leq {} & C \left\| D^2 K \right\|_{L^1} \frac{1}{\widetilde{n}h^2} \rightarrow 0 \quad \text{for } N\rightarrow \infty,
  \end{align}
  by the condition \eqref{eq:172} we imposed on \( h \).
\end{proof}

Since we associate to each \( \mu_N \in \mathcal{P}_N \) an \( L^1 \)-density \( Q_{h(N)}[\mu_N] \) and want to analyze both the behavior of \( \mathcal{E}[\mu_N] \) and \( \left| DQ_{h(N)}[\mu_N] \right|(\mathbb{R}^d) \), we need to incorporate the two different topologies involved, namely narrow convergence of \( \mu \) and \( L^1 \)-convergence of \( Q_{h(N)}[\mu] \), into the concept of \( \Gamma \)-convergence. This can be done by using a slight generalization introduced in \cite{94-Anzellotti-GeneralizedGamma}, named \( \Gamma(q,\tau^-) \)-convergence there:

\begin{definition}[\( \Gamma(q,\tau^-) \)-convergence]
  \label{def:-gammaq-tau}
  \cite[Definition 2.1]{94-Anzellotti-GeneralizedGamma}
  For \( N \in \mathbb{N} \), let \( X_N \) be a set and \( F_N \colon X_N \to \mathbb{R} \) a function. Furthermore, let \( Y \) be a topological space with topology \( \tau \) and \( q = \left\{ q_N \right\}_{N \in \mathbb{N}} \) a family of embedding maps \( q_N \colon X_N \to Y \). Then, \( F_N \) is said to \( \Gamma(q,\tau^-) \)-converge to a function \( F \colon Y \to \overline{\mathbb{R}} \) at \( y \in Y \), if
  \begin{enumrom}
  \item \emph{\( \liminf \)-condition}: For each sequence \( x_N \in X_N \) such that \( q_N(x_N) \xrightarrow{\tau} y \),
    \begin{equation}
      \label{eq:158}
      F(y) \leq \liminf_{N \to \infty} F_N(x_N).
    \end{equation}
  \item \emph{\( \limsup \)-condition}: There is a sequence \( x_N \in X_N \) such that \( q_N(x_N) \xrightarrow{\tau} y \) and
    \begin{equation}
      \label{eq:378}
      F(y) \geq \limsup_{N \to \infty} F_N(x_N).
    \end{equation}
  \end{enumrom}

  Furthermore, we say that the \( F_N \) \( \Gamma(q,\tau^-) \)-converge on a set \( \mathcal{D} \subseteq Y \) if the above is true for all \( y \in \mathcal{D} \) and we call the sequence \( F_N \) equi-coercive, if for every \( c \in \mathbb{R} \), there is a compact set \( K \subseteq Y \) such that \( q_N\left( \left\{ x : F_N(x) \leq c \right\} \right) \subseteq K \).
\end{definition}

\begin{remark}
  The main result with respect to \( \Gamma \)-convergence which we are interested in is the convergence of minimizers, Lemma \ref{lem:8}. This remains true in the case of \( \Gamma(q,\tau^-) \)-convergence, see \cite[Proposition 2.4]{94-Anzellotti-GeneralizedGamma}.
\end{remark}

Here, we are going to consider
\begin{equation}
  \label{eq:379}
  Y :=  \mathcal{P}(\mathbb{R}^d) \times BV(\mathbb{R}^d)
\end{equation}
with the corresponding product topology of narrow convergence and \( BV \)-convergence,
\begin{equation}
  \label{eq:380}
  X_N := \mathcal{P}^N(\mathbb{R}^d), \quad q_N(\mu) := (\mu,Q_{h(N)}[\mu]).
\end{equation}
and consider the limit \( \widehat{\mathcal{E}}^\lambda \) to be defined on the diagonal
\begin{equation}
  \label{eq:381}
  \mathcal{D} := \left\{ (\mu,\mu) : \mu \in \mathcal{P}(\mathbb{R}^d) \cap BV(\mathbb{R}^d) \right\}.
\end{equation}
Since for the existence of minimizers, we will be extracting convergent subsequences of pairs \( (\mu_N,Q_{h(N)}[\mu]) \), we need the following lemma to ensure that the limit is in \( \mathcal{Y} \).

\begin{lemma}[Consistency of the embedding \( Q_{h(N)} \)]
  \label{lem:29}
  If \( (\mu_N)_N \) is a sequence such that \( \mu_N \in \mathcal{P}^N(\mathbb{R}^d) \), \( \mu_N \to \mu \in \mathcal{P}^d \) narrowly and \( Q_{h(N)}[\mu_N] \to \widetilde{\mu} \in BV(\mathbb{R}^d) \) in \( L^1(\mathbb{R}^d) \), as well as \( h \to 0 \), then \( \mu = \widetilde{\mu} \).
\end{lemma}

\begin{proof}
  To show \( \mu = \widetilde{\mu} \), by the metrizability of \( \mathcal{P} \) it suffices to show that \( Q_{h(N)}[\mu_N] \to \mu \) narrowly. For this, as in the proof of Lemma \ref{lem:13}, we can restrict ourselves to test convergence of the integral against bounded and Lipschitz-continuous functions. Hence, let \( f \in C_b(\mathbb{R}^d) \cap \mathrm{Lip}(\mathbb{R}^d) \) with Lipschitz constant \( L \). Then,
  \begin{align}
    \label{eq:400}
    \leadeq{\left| \int_{\mathbb{R}^d} f(x) Q_{h(N)}[\mu_N](x) \diff x - \int_{\mathbb{R}^d} f(x) \diff \mu(x) \right|} \\
    \leq {} & \left| \int_{\mathbb{R}^d} f(x) K_{h(N)} \ast \mu_N(x) \diff x - \int_{\mathbb{R}^d} f(x) \diff \mu_N(x) \right|\\
    {} & + \left| \int_{\mathbb{R}^d} f(x) \diff \mu_N(x) - \int_{\mathbb{R}^d} f(x) \diff \mu(x) \right|,
  \end{align}
  where the second term goes to zero by \( \mu_N \to \mu \) narrowly. For the first term, by Fubini we get that
  \begin{equation}
    \label{eq:403}
    \int_{\mathbb{R}^d} f(x) K_{h(N)} \ast \mu_N(x) \diff x = \int_{\mathbb{R}^d} (f \ast K_{h(N)}(-.)) (x) \diff \mu_N(x)
  \end{equation}
  and therefore
  \begin{align}
    \label{eq:404}
    \leadeq{\left| \int_{\mathbb{R}^d} f(x) K_{h(N)} \ast \mu_N(x) \diff x - \int_{\mathbb{R}^d} f(x) \mu_N(x) \diff x \right|} \\
    = {} & \left| \int_{\mathbb{R}^d} \int_{\mathbb{R}^d} \left( f(x + y) - f(x) \right) K_{h(N)}(y) \diff y \diff \mu_N(x) \right| \\
    = {} & \left| \int_{\mathbb{R}^d} \int_{\mathbb{R}^d} \left( f(x + h(N)y) - f(x) \right) K(y) \diff y \diff \mu_N(x) \right| \\
    \leq {} & Lh \left\| K \right\|_{L^1} \mu_N(\mathbb{R}^d) \to 0, \quad N \to 0
  \end{align}
  by \( h(N) \to 0 \), proving \( Q_{h(N)}[\mu_N] \to \mu \) and therefore the claim.
\end{proof}

\begin{theorem}
  [Consistency of the kernel estimate]
  \label{thm:con-kern-est}
  The functionals \((\widehat{\mathcal{E}}_{N}^{\lambda})_{N \in \mathbb{N}}\) are equi-coercive and
  \begin{equation}
    \label{eq:185}
    \widehat{\mathcal{E}}_{N}^{\lambda} \xrightarrow{\Gamma(q,\tau^-)} \widehat{\mathcal{E}}^{\lambda} \quad \text{for } N \rightarrow \infty
  \end{equation}
  with respect to the topology of \( Y \) defined above, \ie weak convergence of \(\mu_{N}\) together with \(L^{1}\)-convergence of \(Q_{h(N)}[\mu_N]\). In particular, every sequence of minimizers of \( \widehat{\mathcal{E}}^{\lambda}_{N} \)  admits a subsequence converging to a minimizer of \(\widehat{\mathcal{E}}^{\lambda}\).
\end{theorem}

\begin{proof}
  \emph{Ad \( \liminf \)-condition:} This follows from the lower semi-continuity of \( \widehat{\mathcal{E}} \) and \( \mu \mapsto \left| D\mu \right|(\mathbb{R}^d) \) \wrt narrow convergence and \(L^1\)-convergence, respectively.

  \emph{Ad \( \limsup \)-condition:}  We use a diagonal argument to find the recovery sequence: A general \(\mu \in BV(\mathbb{R}^d) \cap \mathcal{P}(\mathbb{R}^d)\) can by Proposition \ref{prp:rec-seq} be approximated by probability measures \( \mu_n \) with existing second moment such that \( \widehat{\mathcal{E}}[\mu_n] \rightarrow \widehat{\mathcal{E}}[\mu] \), namely
  \begin{equation}
    \label{eq:401}
    \mu_n = \widehat{\eta}_{n^{-1}} \cdot \mu + \left( 1 - \widehat{\eta}_{n^{-1}} \cdot \mu(\mathbb{R}^d) \right)\delta_0.
  \end{equation}
  By Lemma \ref{lem:11}, we can also smooth the approximating measures by convolution with a Gaussian \( \eta_{\varepsilon(n)} \) to get a narrowly convergent sequence \( \mu_n' \to \mu \),
    \begin{equation}
      \label{eq:386}
      \mu_n' = \eta_{\varepsilon(n)} \ast \mu_n = \eta_{\varepsilon(n)} \ast (\widehat{\eta}_{n^{-1}}\cdot \mu) + \left(1 - (\widehat{\eta}_{n^{-1}}\cdot \mu)(\mathbb{R}^{d})\right)\eta_{\varepsilon(n)},
    \end{equation}
    while still maintaining continuity in \( \widehat{\mathcal{E}} \). Since \( \left(1 - (\widehat{\eta}_{n^{-1}}\cdot \mu)(\mathbb{R}^{d})\right) \rightarrow 0 \), we can replace its factor \( \eta_{\varepsilon(n)} \) by \( \eta_1 \) to get
    \begin{equation}
      \label{eq:354}
      \mu_n'' = \eta_{\varepsilon(n)} \ast (\widehat{\eta}_{n^{-1}}\cdot \mu) + \left(1 - (\widehat{\eta}_{n^{-1}}\cdot \mu)(\mathbb{R}^{d})\right)\eta_{1},
    \end{equation}
    and still have convergence and continuity in \( \widehat{\mathcal{E}} \). These \( \mu_n'' \) can then be (strictly) cut-off by a smooth cut-off function \( \chi_M \) such that
    \begin{alignat}{2}
      \label{eq:389}
      \chi_M(x) & = 1 &&\text{ for } \left| x \right| \leq M, \\
      \chi_M(x) & \in [0,1] &&\text{ for } M < \left| x \right| < M+1, \\
      \chi_M(x) &= 0 &&\text{ for } \left| x \right| \geq M+1.
    \end{alignat}
    Superfluous mass can then be thrown onto a normalized version of \( \chi_1 \), summarized yielding
    \begin{equation}
      \label{eq:387}
      \mu_n''' = \chi_{M(n)} \cdot \mu_n'' + (1-\chi_{M(n)} \cdot \mu_n'')(\mathbb{R}^d) \frac{\chi_1}{\left\| \chi_1 \right\|_1},
    \end{equation}
    which for fixed \( n \) and \( M(n) \to \infty  \) is convergent in the \( 2 \)-Wasserstein topology, hence we can maintain continuity in \( \widehat{\mathcal{E}} \) by choosing \( M(n) \) large enough.

    Moreover, the sequence \( \mu_n''' \) is also strictly convergent in \( BV \): For the \( L^1 \)-convergence, we apply the Dominated Convergence Theorem for \( M(n) \to \infty \) when considering \( \mu_n''' \) and the Dominated Convergence Theorem and the approximation property of the Gaussian mollification of \( L^1 \)-functions for \( \mu_n'' \). Similarly, for the convergence of the total variation, consider
                        \begin{align}
      \leadeq{\left| \left| D \mu_n''' \right|(\mathbb{R}^d) - \left| D \mu \right|(\mathbb{R}^d) \right|}\\
      \leq {} & \left| \int_{\mathbb{R}^d} \chi_{M(n)}(x) \left| D \mu_n''(x) \right| \diff x - \int_{\mathbb{R}^d} \left| D \mu_n''(x) \right| \diff x \right| \label{eq:391} \\
      {} & + \int_{\mathbb{R}^d} \left| \nabla \chi_{M(n)}(x) \right| \mu_n''(x) \diff x \label{eq:392} \\
      {} & + \left| \left| D \mu_n''(x) \right| - \left| D \mu \right|(\mathbb{R}^d) \right| \label{eq:383} \\
      {} & + (1-\chi_{M(n)} \cdot \mu_n'')(\mathbb{R}^d) \frac{\left\| \nabla \chi_1 \right\|_1}{\left\| \chi_1 \right\|_1}, \label{eq:388}
    \end{align}
    where the terms \eqref{eq:391}, \eqref{eq:392} and \eqref{eq:388} tend to \( 0 \) for \( M(n) \) large enough by Dominated Convergence. For the remaining term \eqref{eq:383}, we have
    \begin{align}
      \leadeq{\left| \left| D \mu_n'' \right| - \left| D \mu \right|(\mathbb{R}^d) \right|}\\
      \leq {} & \left| \left| \eta_{\varepsilon(n)} \ast D (\widehat{\eta}_n \cdot \mu) \right|(\mathbb{R}^d) - \left| D (\widehat{\eta}_n \cdot \mu) \right|(\mathbb{R}^d) \right| \label{eq:393} \\
      {} & + \int_{\mathbb{R}^d} \left| \nabla \widehat{\eta}_n(x) \right| \diff \mu(x) \label{eq:394} \\
      {} & + \int_{\mathbb{R}^d} (1 - \widehat{\eta}_n(x)) \diff \left| D \mu \right|(x) \label{eq:395} \\
      {} & + \left(1 - (\widehat{\eta}_{n^{-1}}\cdot \mu)(\mathbb{R}^{d})\right)\left| D\eta_{1} \right|(\mathbb{R}^d). \label{eq:396}
    \end{align}
    Here, all terms vanish as well: \eqref{eq:393} for \( \varepsilon(n) \) large enough by the approximation property of the Gaussian mollification for \( BV \)-functions and \eqref{eq:394}, \eqref{eq:395} and \eqref{eq:396} by the Dominated Convergence Theorem for \( n \to \infty \). Finally, Lemma \ref{lem:13} applied to the \( \mu_n''' \) yields the desired sequence of point approximations.

  \emph{Ad equi-coercivity and existence of minimizers:} Equi-coercivity and compactness strong enough to ensure the existence of minimizers follow from the coercivity and compactness of level sets of \( \widehat{\mathcal{E}} \) and by \( \| Q_{h(N)}(\mu_N) \|_{L^1} = 1 \) together with compactness arguments in \( BV \), similar to Proposition \ref{prp:consist-cont-bv}. Since Lemma \ref{lem:29} ensures that the limit is in \( \mathcal{Y} \), standard \( \Gamma \)-convergence arguments then yield the convergence of minimizers.
\end{proof}

\subsubsection{Discretization by point-differences}
\label{sec:discrete-version-tv}

In one dimension, the geometry is sufficiently simple to avoid the use of kernel density estimators and in consequence the introduction of an additional scaling parameter as in the previous section and to allow us to explicitly see the intuitive effect the total variation regularization has on point masses (similar to the depiction in Figure \ref{fig:discrete-tv} in the previous section). In particular, formula \eqref{eq:192} below shows that the total variation acts as an additional attractive-repulsive force which enforces equi-spacing between the points masses.

In the following, let \(d = 1\) and \(\lambda > 0\) fixed.

Let \( N\in\mathbb{N} \), \( N \geq 2 \) and \(\mu_{N} \in \mathcal{P}^{N}(\mathbb{R})\) with
\begin{equation}
  \label{eq:189}
  \mu_{N} = \frac{1}{N} \sum_{i=1}^{N}\delta_{x_{i}} \quad \text{for some } x_{i} \in \mathbb{R}.
\end{equation}
Using the ordering on \(\mathbb{R}\), we can assume the \( (x_i)_i \) to be ordered, which allows us to associate to \(\mu_{N}\) a unique vector
\begin{equation}
  \label{eq:190}
  x := x(\mu_{N}) := (x_{1},\dotsc,x_{N}), \quad x_{1} \leq \dotsc \leq x_{N}.
\end{equation}
If \(x_{i} \neq x_{j}\) for all \(i \neq j \in \{1,\dotsc,N\}\), we can further
define an \(L^{1}\)-function which is piecewise-constant by
\begin{equation}
  \label{eq:191}
  \widetilde{Q}_{N}[\mu_N] := \frac{1}{N}\sum_{i = 2}^{N} \frac{1}{x_{i} - x_{i-1}}1_{[x_{i-1},x_{i}]}
\end{equation}
and compute the total variation of its derivative to be
\begin{align}
  \label{eq:192}
  \leadeq{\left| D\widetilde{Q}_{N}[\mu_N] \right|(\mathbb{R})} \\
  = {} & \frac{1}{N}\left[\sum_{i = 2}^{N-1}
    \left| \frac{1}{x_{i+1}-x_{i}} - \frac{1}{x_{i}-x_{i-1}} \right| +
    \frac{1}{x_{2}-x_{1}} + \frac{1}{x_{N}-x_{N-1}}\right],
\end{align}
if no two points are equal, and \(\infty\) otherwise. This leads us to the following definition of the regularized functional using piecewise constant functions:
\begin{align}
  \label{eq:193}
  \mathcal{P}^{N}_{\times}(\mathbb{R}) := {} &\left\{\mu \in \mathcal{P}^{N}(\mathbb{R}) : \mu = \frac{1}{N}
    \sum_{i=1}^{N}\delta_{x_{i}} \text{ with } x_{i} \neq x_{j} \text{ for } i \neq j
  \right\},\\
  \widehat{\mathcal{E}}^{\lambda}_{N,\mathrm{pwc}}[\mu] := {} &
  \begin{cases}
    \widehat{\mathcal{E}}[\mu] + \lambda \left| D\widetilde{Q}_{N}[\mu] \right|(\mathbb{R}), &\mu\in\mathcal{P}_{\times}^{N}(\mathbb{R});\\
    \infty,&\mu \in \mathcal{P}^N(\mathbb{R}) \setminus \mathcal{P}^N_\times(\mathbb{R}).
  \end{cases}
\end{align}

\begin{remark}
  The functions \(\widetilde{Q}_{N}[\mu_N]\) as defined above are not probability densities, but instead have mass \((N-1)/N\).
\end{remark}

We shall again prove \( \Gamma(q,\tau^-) \)-convergence as in Section \ref{sec:discr-kern-estim}, this time with the embeddings \( q_N \) given by \( \widetilde{Q}_{N} \). The following lemma yields the necessary recovery sequence:

\begin{lemma}
  \label{lem:15}
  If \(\mu \in \mathcal{P}_c(\mathbb{R}) \cap C_{c}^{\infty}(\mathbb{R}) \) is the density of a compactly supported probability measure, then there is a sequence \(\mu_{N} \in \mathcal{P}^{N}(\mathbb{R})\), \(N \in \mathbb{N}_{\geq 2}\) such that
  \begin{equation}
    \label{eq:194}
    \mu_{N} \rightarrow  \mu \quad \text{narrowly for } N\rightarrow \infty
  \end{equation}
  and
  \begin{equation}
    \label{eq:195}
    \widetilde{Q}_{N}[\mu_N]\rightarrow \mu \quad \text{in } L^{1}(\mathbb{R}), \quad \left| D\widetilde{Q}_{N}[\mu_N] \right|(\mathbb{R}) \rightarrow  \int_{\mathbb{R}} \left| \mu'(x) \right|
    \dd x \quad \text{for } N\rightarrow \infty.
  \end{equation}
\end{lemma}

\begin{proof}
  \thmenumhspace{-1em}
  \begin{enumpara}
  \item \emph{Definition and narrow convergence:} Let \( \supp \mu \subseteq [-R_\mu, R_\mu] \) and define the vector \(x^{N} \in \mathbb{R}^{N}\) as an \(N\)th quantile of \( \mu \), i.e.\@
    \begin{equation}
      \label{eq:196}
      \int_{x^{N}_{i-1}}^{x^{N}_{i}} \mu(x) \dd x = \frac{1}{N} \quad \text{with
      } x^{N}_{i-1} < x^{N}_{i} \text{ for all } i = 1,\dotsc,N-1,
    \end{equation}
    where we set \( x^{N}_{0} = -R_\mu \) and \( x^{N}_{N} = R_\mu \). Narrow convergence of the corresponding measure then follows by the same arguments used in the proof of Lemma \ref{lem:15}.
  \item \emph{\(L^1\)-convergence:} We want to use the Dominated Convergence Theorem: Let \(x \in \mathbb{R}\) with \(\mu(x) > 0\). Then, by the continuity of \( \mu \), there are \(x_{i-1}^{N}(x)\), \( x_{i}^{N}(x) \) such that \( x \in [x_{i-1}^N(x), x_{i}^N(x)] \) and
    \begin{align}
      \label{eq:197}
      \mu(x) - \widetilde{Q}_{N}[\mu_N](x) = {} & \mu(x) - \frac{1}{N(x^{N}_{i}(x)-x^{N}_{i-1}(x))}\\
      = {} & \mu(x) - \frac{1}{x^{N}_{i}(x)-x^{N}_{i-1}(x)}
      \int_{x^{N}_{i-1}(x)}^{x^{N}_{i}(x)}\mu(y) \dd y. \label{eq:390}
    \end{align}
    Again by \(\mu(x) > 0\) and the continuity of \(\mu\),
    \begin{equation}
      \label{eq:198}
      x^{N}_{i}(x) - x^{N}_{i-1}(x) \rightarrow  0 \quad \text{for } N\rightarrow \infty,
    \end{equation}
    and therefore
    \begin{equation}
      \label{eq:199}
      \widetilde{Q}_{N}[\mu_N](x) \to \mu(x) \quad \text{for all } x \text{ with } \mu(x) > 0.
    \end{equation}
    On the other hand, if we consider an \( x \in [-R_\mu,R_\mu] \) such that \( x \notin \supp \mu \), say \( x \in [a,b] \) such that \( \mu(x) = 0 \) for all \( x \in [a,b] \) and again denote by \( x^N_{i-1}(x), x^N_i(x) \) the two quantiles for which \( x \in [x^N_{i-1}(x),x^N_i(x)] \), then \( x_i^N(x) - x_{i-1}^N(x) \) stays bounded from below because \( x_{i-1}^N(x) \leq a \) and \( x_i^N(x) \geq b \), together with \( N \to \infty \) implying that for such an \( x \),
    \begin{equation}
     \widetilde{Q}_{N}[\mu_N](x) = \frac{1}{N(x^N_i - x^N_{i-1})} \leq \frac{1}{N(b-a)} \to 0. \label{eq:402}
   \end{equation}
   Taking into account that \( \mu(x) = 0 \) for an \( x \in \supp \mu \) can only occur at countably many points, we thus have
    \begin{equation}
      \label{eq:397}
      \widetilde{Q}_{N}[\mu_N](x) \to \mu(x) \quad \text{for almost every } x \in \mathbb{R}.
    \end{equation}
    Furthermore, by \eqref{eq:390} and the choice of the \( (x^N_i)_i \), we can estimate the difference by 
    \begin{equation}
      \label{eq:398}
      \left| \mu(x) - \widetilde{Q}_{N}[\mu_N](x) \right| \leq 2 \left\| \mu \right\|_\infty \cdot 1_{[-R_\mu,R_\mu]}(x),
    \end{equation}
    yielding an integrable dominating function for \( \left| \mu(x) - \widetilde{Q}_{N}[\mu_N](x) \right| \) and therefore justifying the \( L^1 \)-convergence
    \begin{equation}
      \label{eq:385}
      \int_{\mathbb{R}}\left| \mu(x) - \widetilde{Q}_{N}[\mu_N](x) \right| \dd x \to 0, \quad N \to \infty.
    \end{equation}
                                                  \item \emph{Strict BV-convergence:} For strict convergence of \(\widetilde{Q}_{N}[\mu_N]\) to \(\mu\), we additionally have to check that \(\limsup_{N\rightarrow \infty} \left| D\widetilde{Q}_{N}[\mu_N] \right|(\mathbb{R}) \leq \left| D\mu \right|(\mathbb{R}) \) (since the inequality in the other direction is already fulfilled by the lower semi-continuity of the total variation). To this end, consider
    \begin{align}
      \label{eq:202}
      \leadeq{\left| D\widetilde{Q}_{N}[\mu_N] \right|(\mathbb{R})}\\
      = {} &\sum_{i = 2}^{N-1} \left| \frac{1}{N}\frac{1}{x^{N}_{i+1}-x^{N}_{i}} - \frac{1}{N}\frac{1}{x^{N}_{i}-x^{N}_{i-1}} \right| + \frac{1}{N(x^{N}_{2}-x^{N}_{1})} + \frac{1}{N(x^{N}_{N}-x^{N}_{N-1})}\\
      = {} & \sum_{i = 2}^{N-1} \left| \frac{1}{x^{N}_{i+1}-x^{N}_{i}}\int_{x^{N}_{i}}^{x^{N}_{i+1}}\mu(x) \dd x - \frac{1}{x^{N}_{i}-x^{N}_{i-1}}\int_{x^{N}_{i-1}}^{x^{N}_{i}}\mu(x) \dd x \right|\\
      &+ \frac{1}{x^{N}_{2}-x^{N}_{1}}\int_{x^{N}_{1}}^{x^{N}_{2}}\mu(x) \dd x + \frac{1}{x^{N}_{N}-x^{N}_{N-1}}\int_{x^{N}_{N-1}}^{x^{N}_{N}}\mu(x) \dd x\\
      = {} &\sum_{i = 1}^{N}\left| \mu(t_{i+1})-\mu(t_{i}) \right|
    \end{align}
    for \(t_{i} \in [x^{N}_{i},x^{N}_{i-1}]\), \(i = 2,\dotsc,N\) chosen by the mean value theorem (for integration) and \(t_{1},t_{N+1}\) denoting \( -R_\mu \) and \( R_\mu \), respectively. Hence,
    \begin{equation}
      \label{eq:203}
      \left| D\widetilde{Q}_{N}[\mu_{N}] \right|(\mathbb{R}) \leq \sup \left\{\sum_{i=1}^{n - 1}\left| \mu(t_{i+1})-\mu(t_{i}) \right| : n \geq 2,\, t_{1} < \dotsb < t_{n}\right\} = V(\mu),
    \end{equation}
    the pointwise variation of \(\mu\), and the claim now follows from \(V(\mu) = \left| D\mu \right|(\mathbb{R})\) by \cite[Theorem 3.28]{00-Amb-Fusco-Pallara-BV}, because by the smoothness of \( \mu \), it is a good representative of its equivalence class in \( BV(\mathbb{R}) \), \ie, one for which the pointwise variation coincides with the measure theoretical one. \qedhere
  \end{enumpara}
\end{proof}

As in the previous section, we have to verify that a limit point of a sequence \( (\mu_N, \widetilde{Q}_N[\mu_N]) \) is in the diagonal \( \mathcal{Y} \):

\begin{lemma}[Consistency of the embedding \( \widetilde{Q}_N \)]
  \label{lem:30}
  Let \( (\mu_N)_N \) be a sequence where \( \mu_N \in \mathcal{P}^N(\mathbb{R}) \), \( \mu_N \to \mu \) narrowly and \( \widetilde{Q}_N[\mu_N] \to \widetilde{\mu} \) in \( L^1(\mathbb{R}) \). Then \( \mu = \widetilde{\mu} \).
\end{lemma}

\begin{proof}
  Denote the distribution functions of \( \widetilde{Q}_N[\mu_N] \), \( \mu_N \) and \( \mu \) by \( \widetilde{F}_N \), \( F_N \) and \( F \), respectively. We can deduce \( \mu = \widetilde{\mu} \) if \( \widetilde{F}_N(x) \to F(x) \) for every \( x \in \mathbb{R} \)  (even if the measures \( \widetilde{Q}_N[\mu_N] \) have only mass \( (N-1)/N \), this is enough to show that the limit measures have to coincide, for example by rescaling the measures to have mass \( 1 \)). Note that the construction of \( \widetilde{Q}_N[\mu_N] \) exactly consists of replacing the piecewise constant functions \( F_n \) by piecewise linear functions interpolating between the points \( (x^N_i)_i \). Now, taking into account that the jump size \( F_N(x^N_i)-F_N(x^N_{i-1}) \) is always \( 1/N \) we see that
  \begin{align}
    \label{eq:405}
    | \widetilde{F}_N(x) - F(x)  | \leq {} & | \widetilde{F}_N(x) - F_N(x) | + \left| F_N(x) - F(x) \right|\\
    \leq {} & \frac{1}{N} + \left| F_N(x) - F(x) \right|\to 0, \quad N \to 0,
  \end{align}
  which is the claimed convergence.
\end{proof}

\begin{theorem}
  [Consistency of \( \widehat{\mathcal{E}}^{\lambda}_{N,\mathrm{pwc}} \)]
  \label{thm:cons-discrete-tv}
  For \(N\rightarrow \infty\), \(\widehat{\mathcal{E}}^{\lambda}_{N} \xrightarrow{\Gamma(q,\tau^-)} \widehat{\mathcal{E}}^{\lambda}\) with respect to the topology of \( Y \) in \eqref{eq:379} in the case \( d = 1 \), \ie, the topology induced by narrow convergence together with \(L^{1}\)-convergence of the associated densities, and the family \((\widehat{\mathcal{E}}^{\lambda}_{N})_{N}\) is equi-coercive. In particular, every sequence of minimizers of \( \widehat{\mathcal{E}}^{\lambda}_{N} \)  admits a subsequence converging to a minimizer of \(\widehat{\mathcal{E}}^{\lambda}\).
\end{theorem}

\begin{proof}
  \thmenumhspace{-1em}
  \begin{enumpara}
  \item \emph{\( \liminf \)-condition:} Let \(\mu_{N} \in \mathcal{P}_{N}(\mathbb{R})\) and
    \(\mu \in BV(\mathbb{R}) \cap \mathcal{P}(\mathbb{R}) \) with \(\mu_{N}\rightarrow \mu\) narrowly and \(\widetilde{Q}_{N}[\mu_N]\rightarrow \mu\) in \(L^{1}\).
    Then,
    \begin{equation}
      \label{eq:205}
      \liminf_{N\rightarrow \infty} \widehat{\mathcal{E}}^{\lambda}_{N,\mathrm{pwc}}[\mu_{N}] = \liminf_{N\rightarrow \infty} \left[\widehat{\mathcal{E}}[\mu_{N}] + \left| D\widetilde{Q}_{N}[\mu_N] \right|(\mathbb{R}) \right] \geq \widehat{\mathcal{E}}[\mu] + \left| D\mu \right|(\mathbb{R})
    \end{equation}
    by the lower semi-continuity of the summands with respect to the involved
    topologies.
  \item \emph{\( \limsup \)-condition:} We use the same diagonal argument used in the proof of Theorem \ref{thm:con-kern-est}, replacing the final application of Lemma \ref{lem:13} there by Lemma \ref{lem:15}, which serves the same purpose, but uses the point differences instead of the kernel estimators.
  \item \emph{Equi-coercivity and existence of minimizers:} The coercivity follows analogously to the proof of Theorem \ref{thm:con-kern-est}, which also justifies the existence of minimizers for each \( N \). The convergence of minimizers to an element of \( \mathcal{Y} \) then follows by standard arguments together with Lemma \ref{lem:30}.
  \end{enumpara}
\end{proof}

\begin{remark}
  \label{rem:non-diagonal-topology}
  In both cases, instead of working with two different topologies, we could also consider
  \begin{equation}
    \label{eq:399}
    \widehat{\mathcal{E}}^\lambda_{N,\mathrm{alt}} := \widehat{\mathcal{E}}[Q[\mu]] + \lambda \left| DQ[\mu] \right| (\mathbb{R}^d),
  \end{equation}
  for a given embedding \( Q \) (which in the case of point differences would have to be re-scaled to keep mass \( 1 \)). Then, we would obtain the same results (with identical arguments), but without the need to worry separately about narrow convergence, since it is then implied by the \( L^1 \)-convergence of \( Q[\mu_N] \) by Lemma \ref{lem:1}.
\end{remark}

\section{Numerical experiments}
\label{sec:numer-exper}

\begin{figure}[t]
  \centering
  \subfigure[\( \omega_1 \)  \label{fig:data1}
]{\includegraphics[width=6cm]{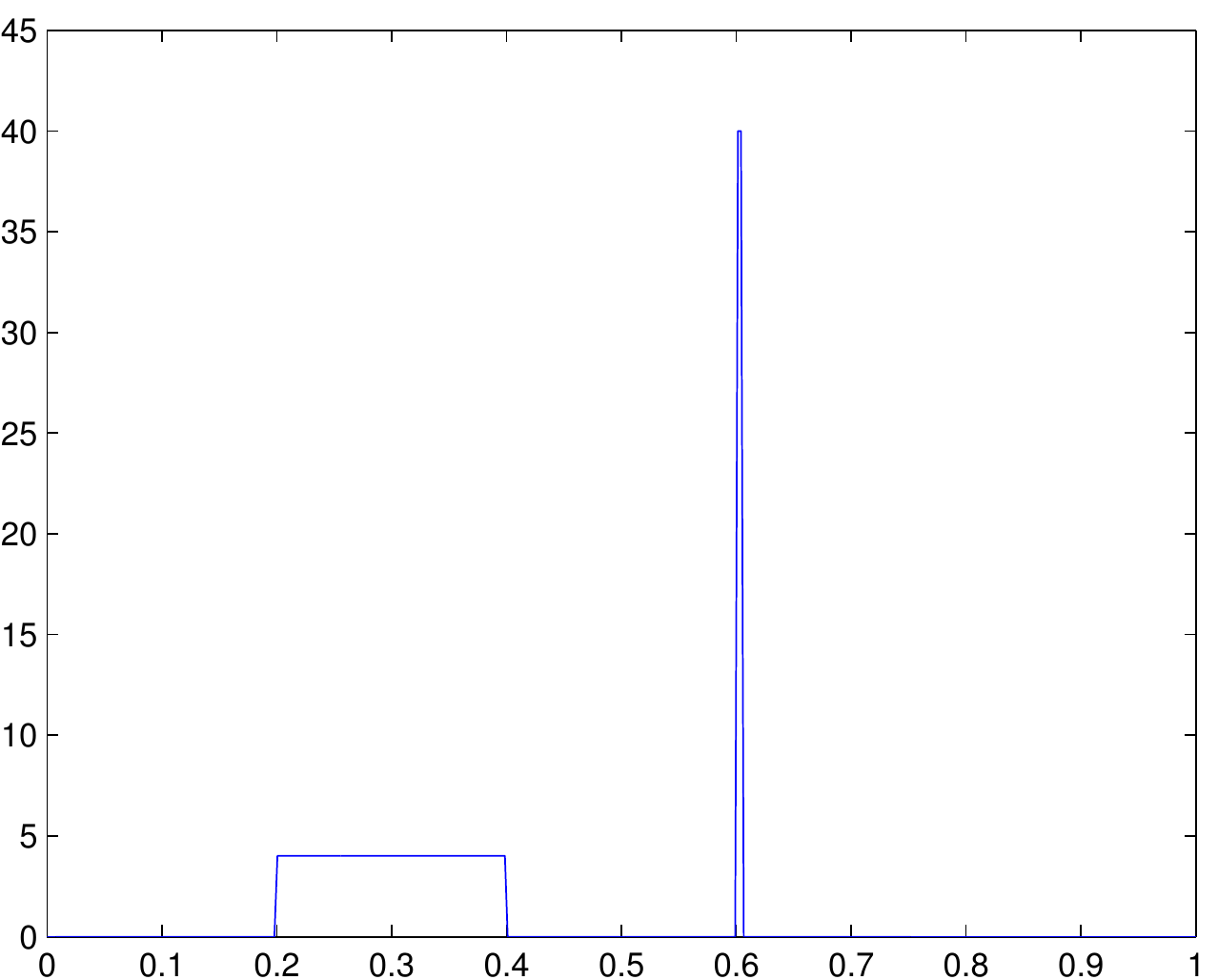}}
  \hspace{0.5cm}
    \subfigure[\( \omega_2 \)  \label{fig:data2} ]{\includegraphics[width=6cm]{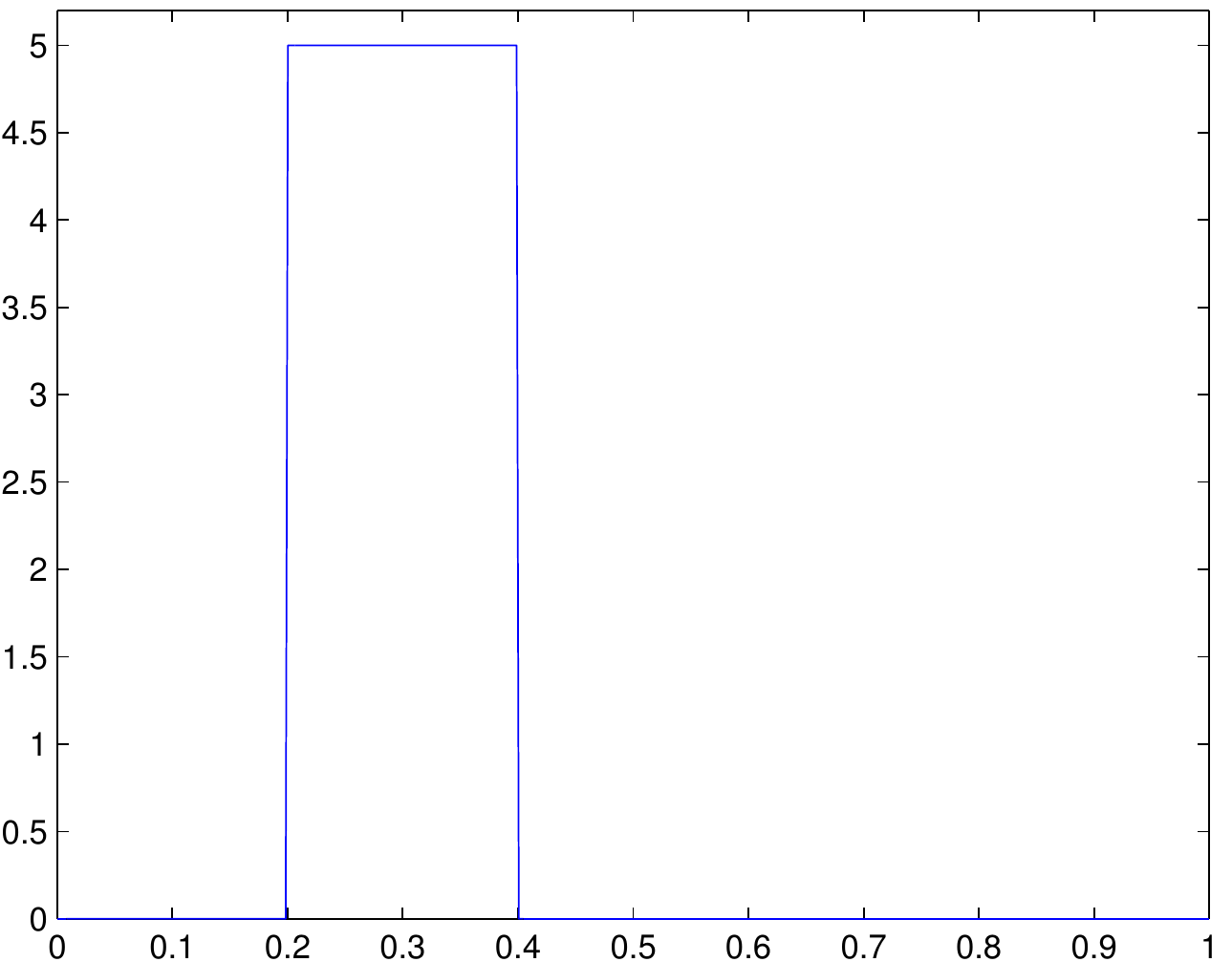}}
  \caption{The data \( \omega_1 \) and \( \omega_2 \)}
\end{figure} 

In this section, we shall show a few results of the numerical computation of minimizers to \( \widehat{\mathcal{E}}^\lambda \) and \( \widehat{\mathcal{E}}^\lambda_N \) in 1D in order to numerically verify the \( \Gamma \)-convergence result in Theorem \ref{thm:con-kern-est}.

\subsection{Grid approximation}
\label{sec:continuous-case}

\begin{figure}[t]
  \centering
  \subfigure[\( \mu \in L^1, \, \lambda = 10^{-4} \)]{\includegraphics[width=6cm]{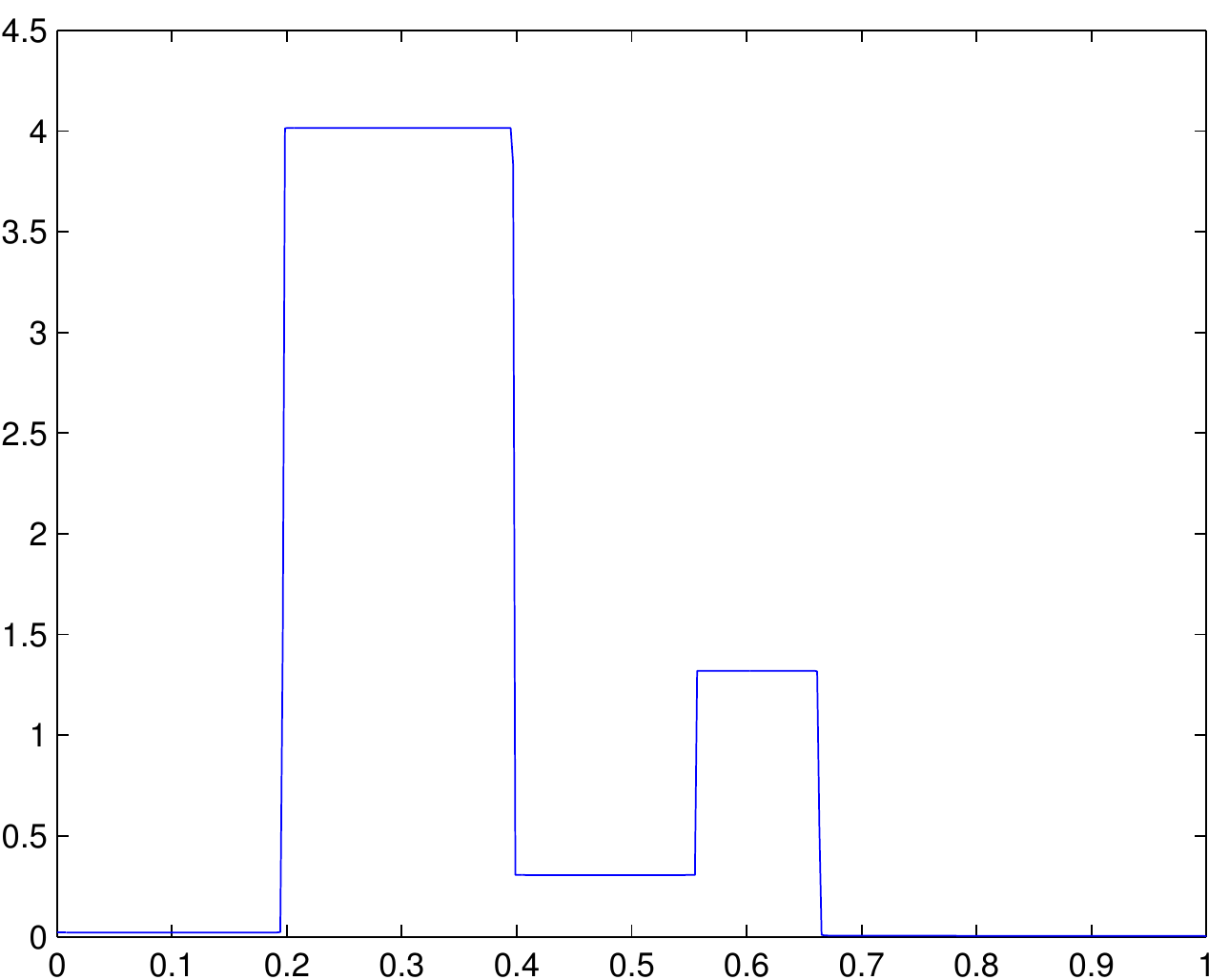}}
  \hspace{0.5cm}
  \subfigure[Particles, \( \lambda = 10^{-4} \)]{\includegraphics[width=6cm]{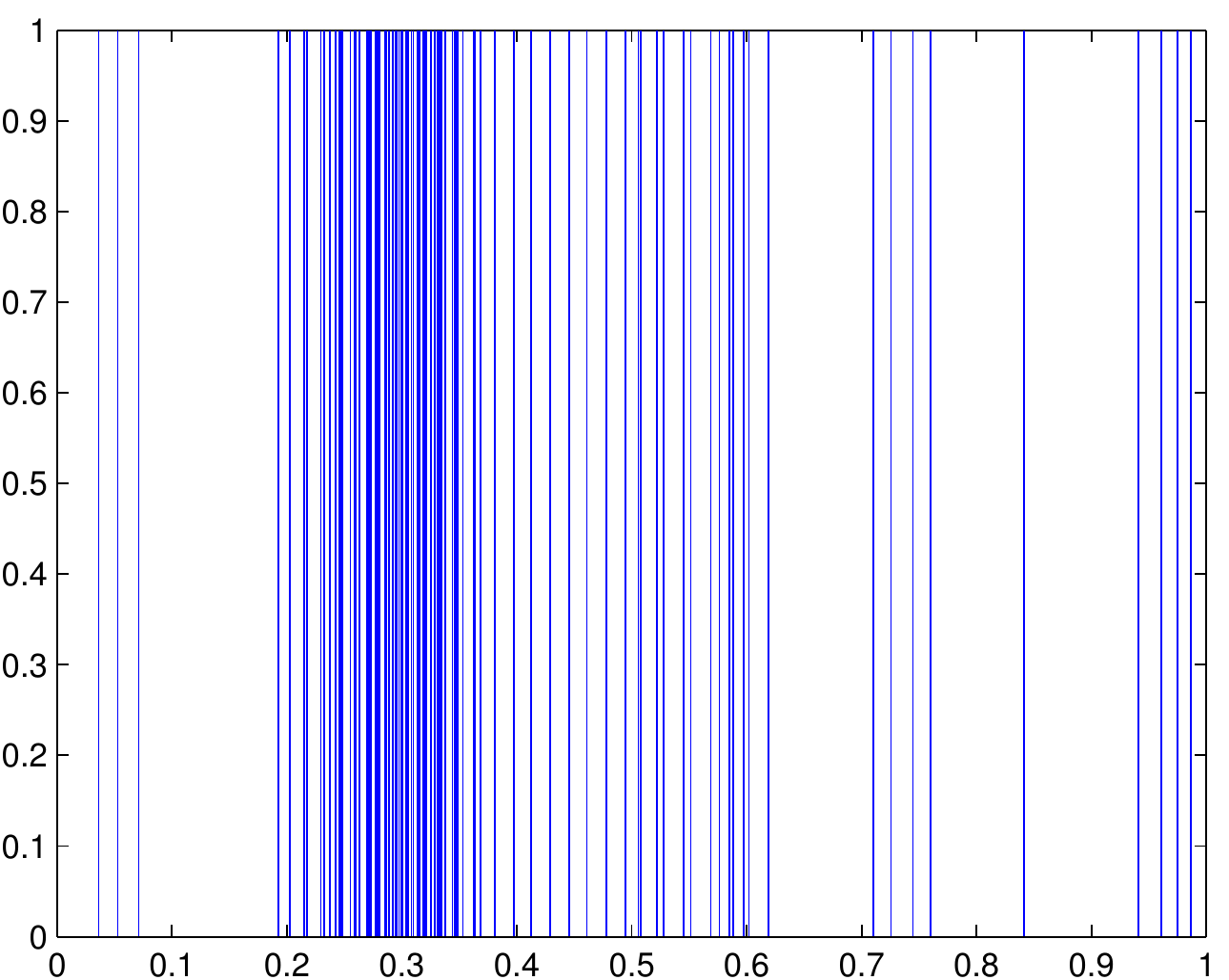}}

  \subfigure[\( \mu \in L^1, \, \lambda = 10^{-6} \)]{\includegraphics[width=6cm]{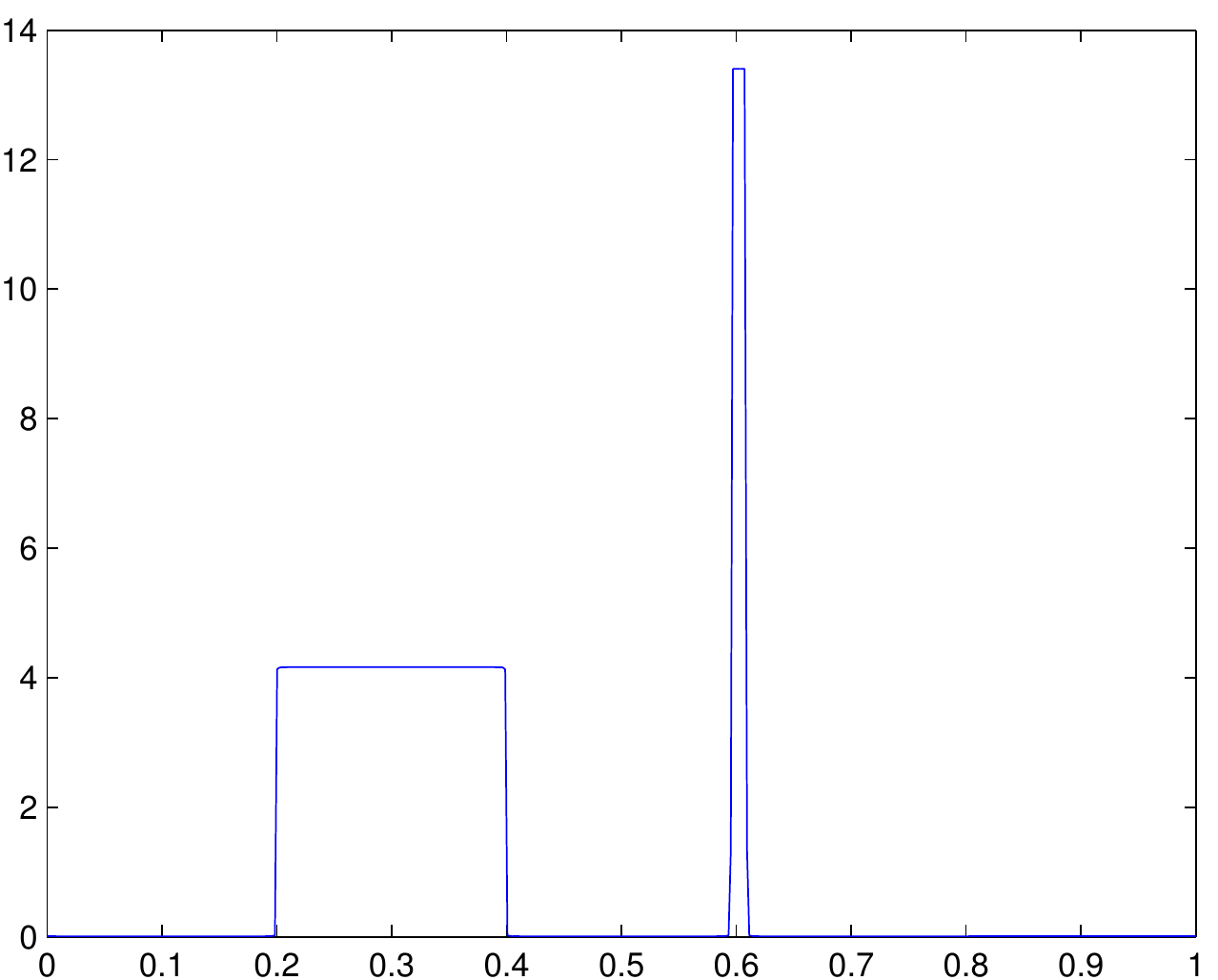}}
  \hspace{0.5cm}
  \subfigure[Particles, \( \lambda = 10^{-6} \)]{\includegraphics[width=6cm]{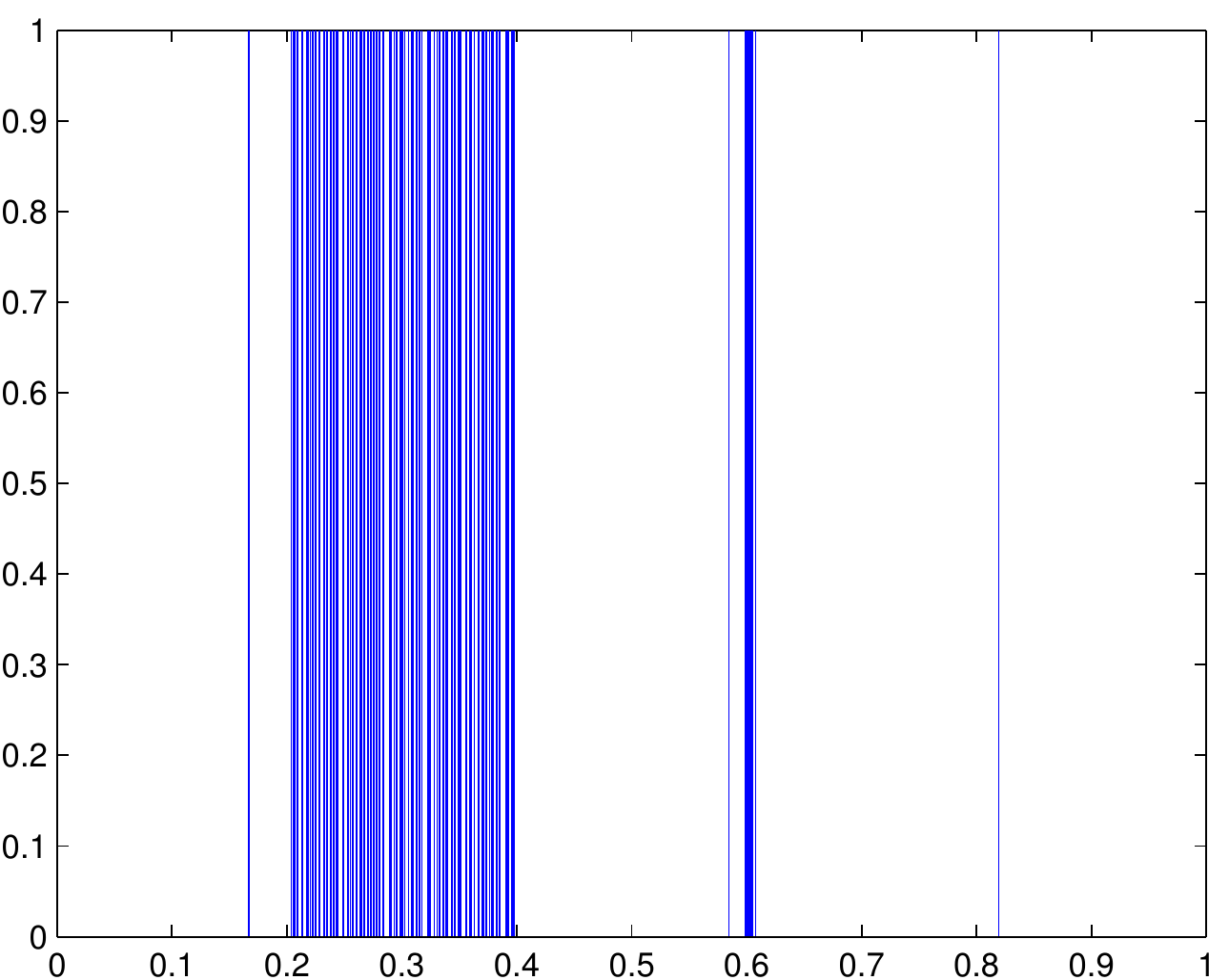}}
  \caption{Minimizers \( \mu \) of \eqref{eq:323} and minimizers \( \mu_N \) of \( \widehat{\mathcal{E}}^\lambda_N \) for \( \omega_1 \) as in \ref{fig:data1} and parameters \( q = 1.0, \, N = 100 \)}
  \label{fig:spikes}
\end{figure}

By Theorem \ref{thm:con-kern-est}, we know that \( \widehat{\mathcal{E}}^\lambda_N \xrightarrow{\Gamma} \widehat{\mathcal{E}}^\lambda \), telling us that the particle minimizers of \( \widehat{\mathcal{E}}^\lambda \) will be close to a minimizer of the functional \( \widehat{\mathcal{E}}^\lambda \), which will be a \( BV \) function. Therefore, we would like to compare the particle minimizers to minimizers which were computed by using a more classical approximation method which in contrast maintains the underlying \( BV \) structure. One such approach is to approximate a function in \( BV \) by interpolation by piecewise constant functions on an equispaced discretization of the interval \( \Omega = [0, 1] \). Denoting the restriction of \( \widehat{\mathcal{E}}^\lambda \) to the space of these functions on a grid with \( N \) points by \( \widehat{\mathcal{E}}^\lambda_{N,\mathrm{grid}} \), it can be seen that we have \( \widehat{\mathcal{E}}^\lambda_{N,\mathrm{grid}} \xrightarrow{\Gamma} \widehat{\mathcal{E}}^\lambda \), hence it makes sense to compare minimizers of \( \widehat{\mathcal{E}}^\lambda_{N,\mathrm{grid}} \) and \( \widehat{\mathcal{E}}^\lambda_N \) for large \( N \).

If we denote by \( u \in \mathbb{R}^m \) the approximation to \( \mu \) and by \( w \in \mathbb{R}^m \) the one to \( \omega \), then the problem to minimize \( \widehat{\mathcal{E}}^\lambda_{N,\mathrm{grid}} \) takes the form
\begin{equation}
  \label{eq:323}
  \begin{aligned}
    &\text{minimize }& &(u-w)^T A_{q,\Omega} (u-w) + \lambda \sum_{i = 1}^{m-1} \left| u_{i + 1}-u_{i} \right| \\
    &\text{subject to } & &u \geq 0, \quad \sum_{i = 1}^{m} u_i = m,
  \end{aligned}
\end{equation}
where \( A_{q,\Omega} \) is the corresponding discretization matrix of the quadratic integral functional \( \widehat{\mathcal{E}} \), which is positive definite on the set \( \left\{ v : \sum v = 0 \right\} \) by the theory of Appendix \ref{cha:cond-posit-semi}. Solving the last condition \( \sum_{i=1}^{m} u_i = m \) for one coordinate of \( u \), we get a reduced matrix \( \widetilde{A}_{q, \Omega} \) which is positive definite. Together with the convex approximation term to the total variation, problem \eqref{eq:323} is a convex optimization problem which can be solved with the \textsc{cvx} package \cite{cvx}, \cite{08-Grant-Boyd-Convex-Programs}.

As a model case to study the influence of the total variation, the following cases were considered
\begin{enumerate}
\item \( \omega_1 = 4 \cdot 1_{[0.2,0.4]} + 40 \cdot 1_{[0.6,0.605] } \), the effect of the regularization being that the second bump gets smaller and more spread out with increasing parameter \( \lambda \), see Figure \ref{fig:spikes};
\item a version of \( 5 \cdot 1_{[0.2,0.4]} + \eta \), where \( \eta \) is some Gaussian noise disturbing the reference measure \( \omega_2 = 5 \cdot 1_{[0.2,0.4]} \) and where we cut off the negative part and re-normalized the datum to get a probability measure. The effect of the regularization here is a filtering of the noise, see Figure \ref{fig:noise}.
\end{enumerate} 

\subsection{Particle approximation}
\label{sec:part-appr-1}

\begin{figure}[t]
  \centering
  \subfigure[\( \mu \in L^1, \, \lambda = 0 \)]{\includegraphics[width=6cm]{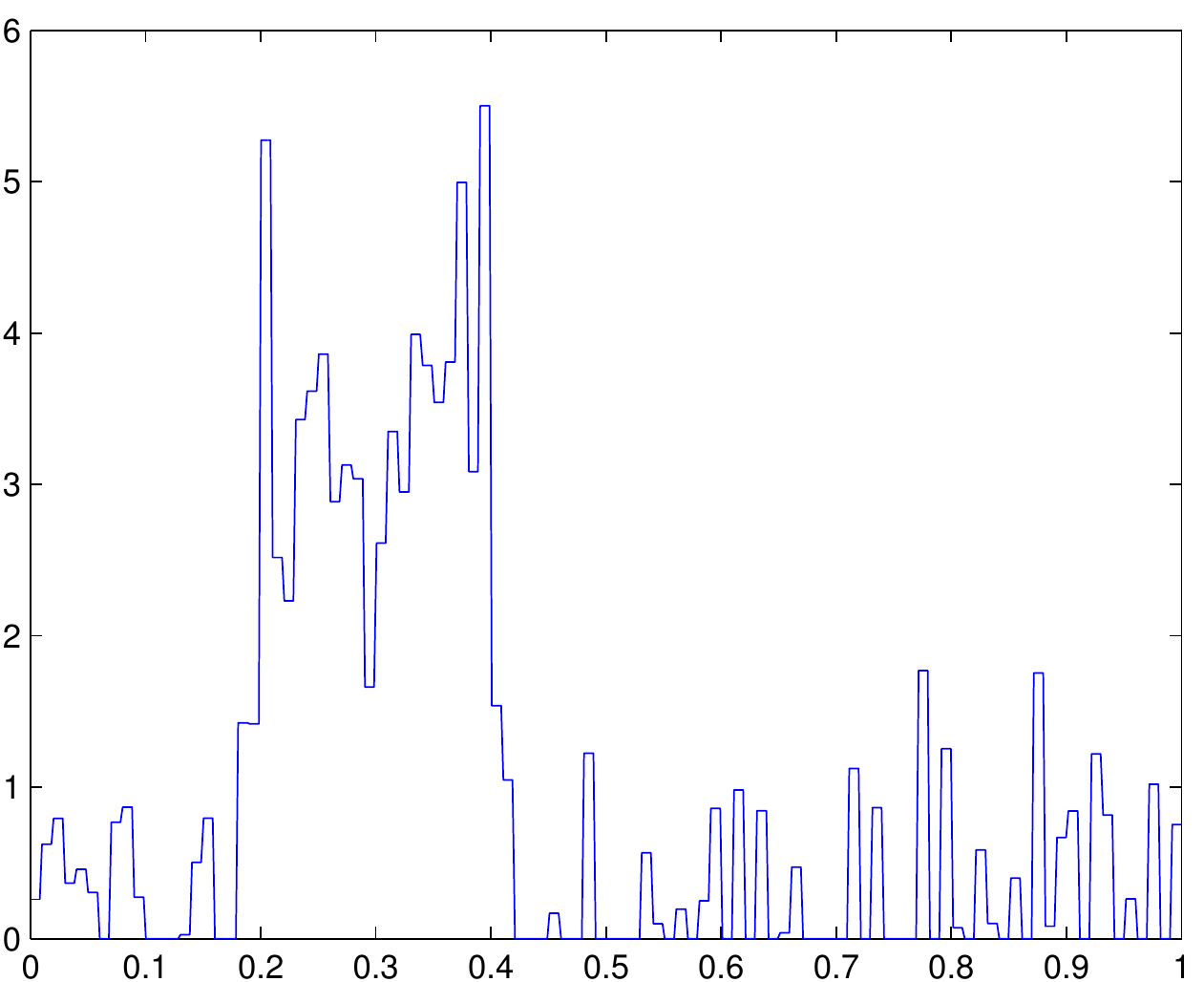}}
  \hspace{0.5cm}
  \subfigure[\( Particles, \, \lambda = 0 \)]{\includegraphics[width=6cm]{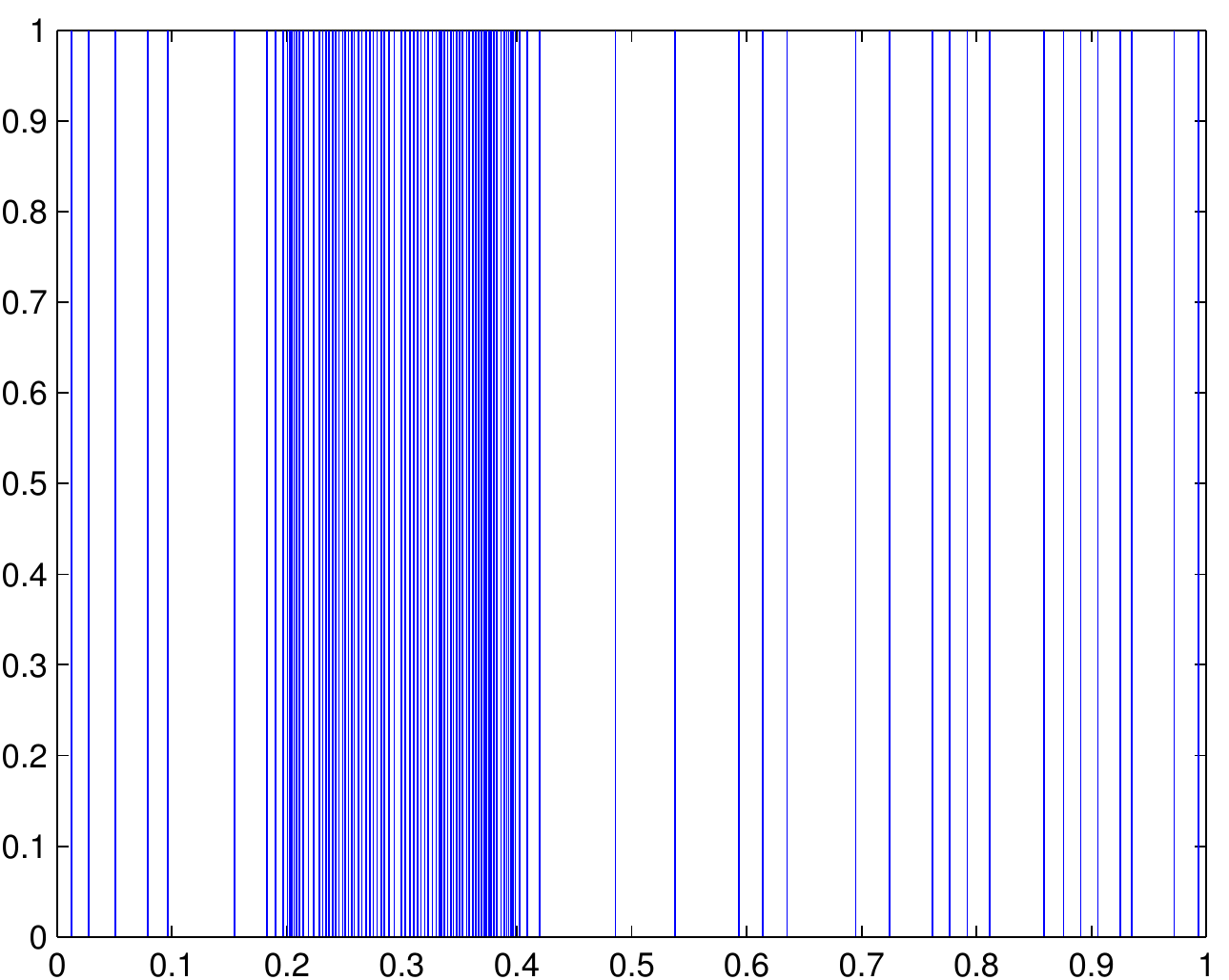}}

  \subfigure[\( \mu \in L^1, \, \lambda = 10^{-5} \)]{\includegraphics[width=6cm]{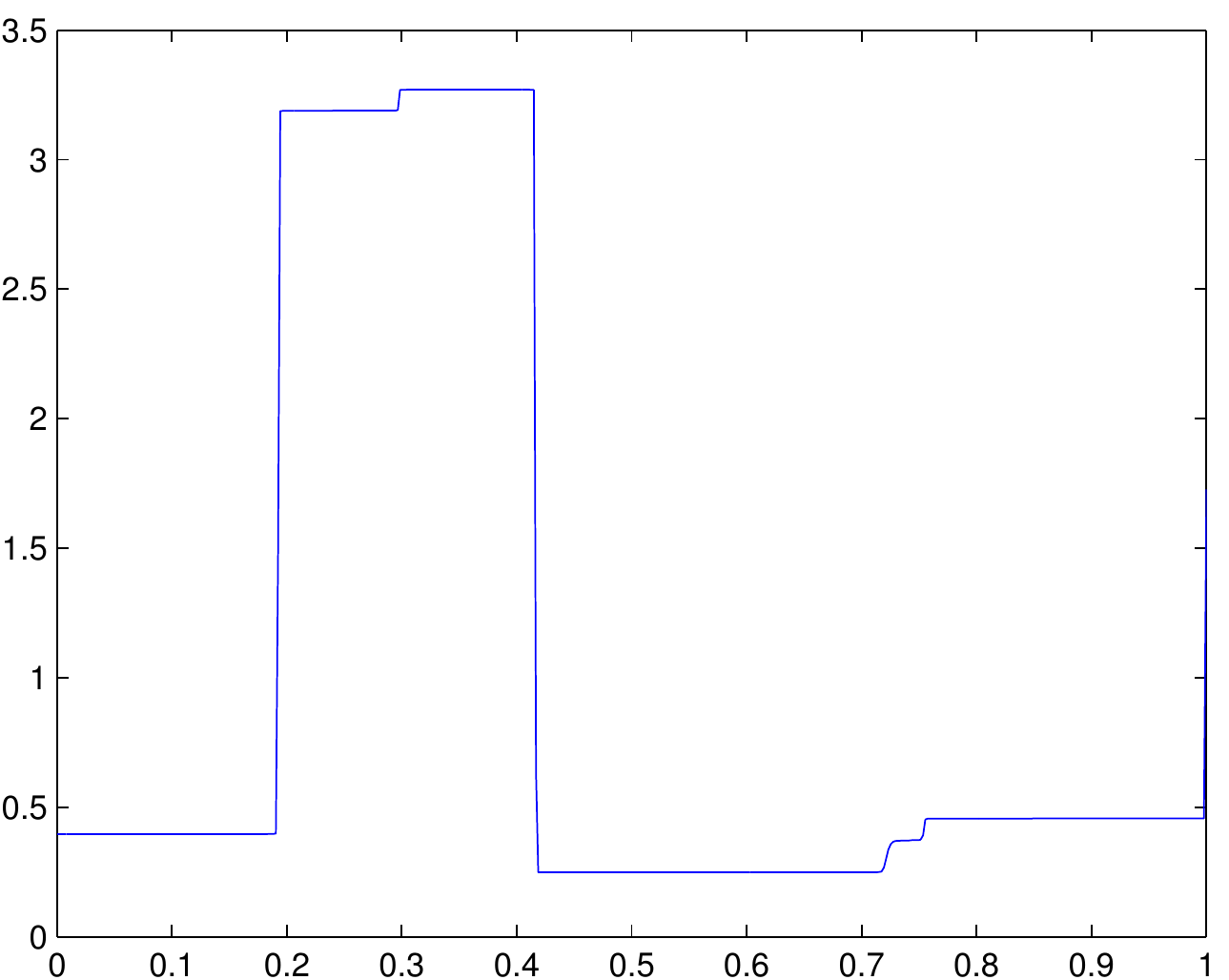}}
  \hspace{0.5cm}
  \subfigure[\( Particles, \, \lambda = 10^{-5} \)]{\includegraphics[width=6cm]{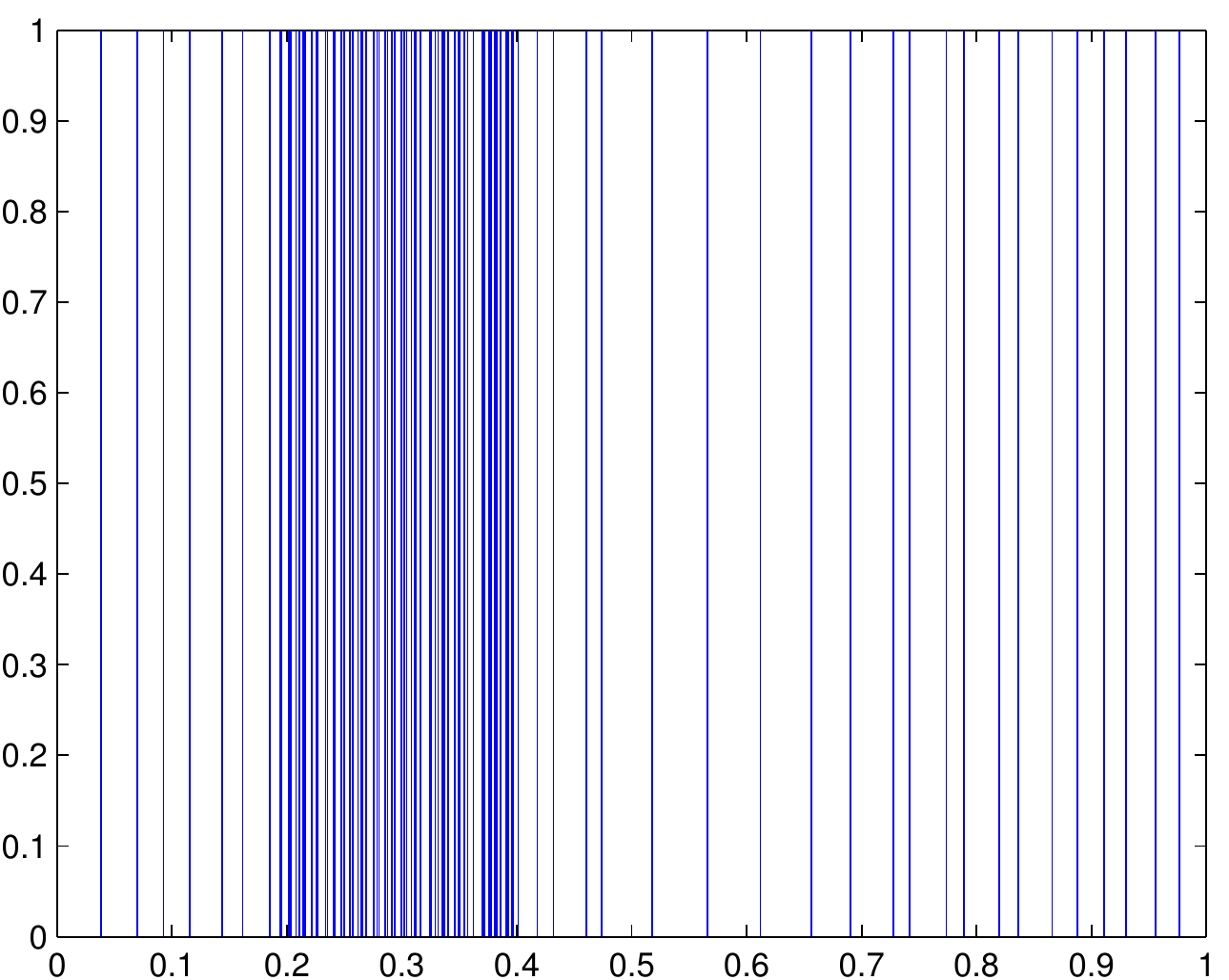}}
  \caption{Minimizers \( \mu \) of \eqref{eq:323} and minimizers \( \mu_N \) of \( \widehat{\mathcal{E}}^\lambda_N \) for a projection of \( \omega_2 + \eta \) as in \ref{fig:data2} and parameters \( q = 1.5, \, N = 100 \)}
  \label{fig:noise}
\end{figure}

The solutions in the particle case were computed by the \textsc{matlab} optimization toolbox \cite{matlab}, \cite{matlab-opt}, in particular the Quasi-Newton method available via the \textsc{fminunc} command. The corresponding function evaluations were computed directly in the case of the repulsion functional and by a trapezoidal rule in the case of the attraction term. For the kernel estimator, we used the one sketched in Figure \ref{fig:discrete-tv},
\begin{equation}
  \label{eq:324}
  K(x) = \left( 1 - \left| x \right| \right) \cdot 1_{[-1,1]}(x), \quad x \in \mathbb{R}.
\end{equation}

\subsection{Results}
\label{sec:num-results}

As for the \( L^1 \) case, we see that the total variation regularization works well and allows us to recover the original profile from a datum disturbed by noise.

When it comes to the particle case, we see the theoretical results of convergence for \( N \rightarrow \infty \) of Section \ref{sec:discrete-version-tv-1} confirmed, since the minimizers of the particle system behave roughly like the minimizers of the problem in \( L^1 \). On the other hand, the latter seems to be far more amenable to an efficient numerical treatment than the former because we lose convexity of the total variation term when passing to the particle formulation and the results there are for reasonably small \( N \) (like \( N = 100 \) here) strictly dependent on the choice of \( h_N \).

\section{Conclusion}
\label{sec:conclusion}

Apart from the easy conclusions for asymmetric exponents \( q_a \neq q_r \) in Section \ref{sec:prel-obs}, the Fourier representation of Section \ref{sec:prop-funct-mathbbrd}, resting upon the theory of Appendix \ref{cha:cond-posit-semi}, proved essential to establish a good formulation of the problem in terms of the lower semi-continuous envelope. This allowed us to use the well-established theory of the calculus of variations, in particular the machinery of \( \Gamma \)-convergence, to prove statements like the consistency of the particle approximation, Theorem \ref{thm:cons-part-appr}, and the moment bound, Theorem \ref{thm:moment-bound-symmetric}, which are otherwise not at all obvious when just considering the original spatial definition of \( \mathcal{E} \).

Moreover, it enabled us to easily analyze the regularized version of the functional in Section \ref{sec:tv-reg}, which on the particle level exhibits an interesting attractive-repulsive behavior, translating the regularizing effect of the total variation in the continuous case into an energy which tries to enforce a configuration of the particles which is as homogeneous as possible, while simultaneously minimizing \( \widehat{\mathcal{E}} \).

\clearpage{}\makeatletter{}
\chapter{Gradient flow in 1D}
\label{cha:gradient-flow-1d}

In this section we shall consider the gradient flow of the functional \eqref{eq:9} in the space \( \mathcal{P}_2(\mathbb{R}) \) endowed with the \( 2 \)-Wasserstein metric (see Definition \ref{def:wasserstein-distance}), which can be written as  
\begin{equation}
  \label{eq:213}
  \partial_t\mu = \nabla \cdot \left[ \left(\nabla \psi_a \ast \omega - \nabla \psi_r \ast \mu \right) \mu\right], \quad \mu(0) = \mu_0 \in \mathcal{P}_2(\mathbb{R}),
\end{equation}
with the notation of \eqref{eq:9}. 

We shall try to answer questions about its existence and its asymptotic behavior for \( t \rightarrow \infty \) after having given a brief overview of previously known results in Section \ref{sec:intro}.  For this, we restrict ourselves to the case \( q_a = q_r \in [1,2] \) and \( \Omega = \mathbb{R} \) in order to be able to use the \emph{pseudo inverse transform} which we briefly introduce in Section \ref{sec:pseudo-inverse-techn} and which renders the involved mathematical objects and thence the asymptotic analysis much easier. What follows is Section \ref{sec:existence-theory-linfty}, which deals with the well-posedness of the pseudo inverse equation, while Section \ref{sec:asymptotics} is concerned with the asymptotic behavior for the limit cases \( q = 1 \) and \( q = 2 \).

\section{Previously known results}
\label{sec:intro}

\subsection{Well-posedness}
\label{sec:well-posedness}

The linear attractive term \( \nabla \cdot \left[ \left(\nabla \psi_a \ast \omega \right ) \mu \right ] \) in equation \eqref{eq:213} does not pose much difficulties when it comes to the question of well-posedness as it corresponds to a Lipschitz flow under mild assumptions on $\omega$ (see Lemma \ref{lem:18} below). However, the repulsive part \( - \nabla \cdot \left[ \left(\nabla \psi_r \ast \mu \right) \mu \right ] \) still presents some problems with respect to well-posedness and particle approximation, despite being studied intensively for its broad range of applications in mathematical modeling.

There, the typical setting is
\begin{equation}
  \label{eq:215}
  \partial_t\mu = \nabla \cdot \left[ \left(\nabla W \ast \mu \right)
    \mu\right], \quad \mu(0) = \mu_0 \in \mathcal{P}(\mathbb{R}^d),
\end{equation}
where $ \mathcal{P}(\mathbb{R}^d)$ is the set of probability measures on $\mathbb R^d$, and $W$ is a suitable kernel associated to the non-local driving potential
 \begin{equation}
  \label{eq:216}
  \mathscr W[\mu]=\frac{1}{2}\int_{\mathbb R^d} \int_{\mathbb R^d} W(x-y) d \mu(x) d\mu (y).
\end{equation}
One standard set of regularity assumptions on \( W \) to ensure existence and uniqueness of a gradient flow solution is for example that of \cite{AGS08,CDFLS}, namely that
\begin{itemize}
\item \( W \) is symmetric, \( W(x) = W(-x) \),
\item \( W \in C^1(\mathbb{R}^d \setminus \left\{ 0 \right\}) \),
\item \( W \) is \( \lambda \)-convex, \ie
  \begin{equation}
    \label{eq:377}
    \exists \lambda < 0 : W(x) - \lambda \left| x \right|^2 \text{ is convex},
  \end{equation}
  which also implies that the singularity of \( \nabla W \) at \( 0 \) is not worse than Lipschitz.
\end{itemize}
Unfortunately, the last condition above fails for the repulsive kernel in question.

One possibility to gain further results is restricting the space in which we are looking for solutions to \( L^p \cap \mathcal{P}_2 \) as in \cite[Theorem 5]{belaro11}. The results there ensure global existence in the case of a repulsive kernel for
\begin{itemize}
\item \( W \) being radially symmetric,
\item \( W \) smooth on \( \mathbb{R}^d \setminus \left\{ 0 \right\} \),
\item the singularity of \( W \) at \( 0 \) not being worse than Lipschitz and \( W \) not exhibiting pathological oscillations there,
\item its derivatives decaying fast enough for \( x \rightarrow \infty \),
\end{itemize}
(together implying \( W \in W^{1, p'}(\mathbb{R}^d) \)) for \( p > d/(d - 1) \). Yet, the case \( d = 1 \) is not included there (or would require to take the formal limit \( p \rightarrow \infty \)).

Another approach is \cite[Theorem 7]{BCLR}, where the existence of strong classical solutions to equation \eqref{eq:215} is shown under integrability assumptions on the first two derivatives of \( W \) and boundedness of the positive part of the Laplacian of \( W \), rendering it applicable to the local behavior of repulsion kernel for \( 1 < q \leq 2 \). A recent result in \cite[Theorem 4.1]{13-Carrillo-Choi-Hauray-MFL} is also applicable in this case, as the repulsive kernel fulfills the growth requirements
\begin{equation}
  \label{eq:334}
  \left| W'(x) \right| \leq \frac{C}{\left| x \right|^\alpha}, \quad \left| W''(x) \right| \leq \frac{C}{\left| x \right|^{1+\alpha}} \quad \forall x \in \mathbb{R}^d \setminus \left\{ 0 \right\},
\end{equation}
with \( -1 \leq \alpha < d-1 = 0 \) for \( \alpha := 1-q \), yielding a local existence result for weak measure solutions in \( \mathcal{P}_1(\mathbb{R}) \cap L^p(\mathbb{R}) \) for \( p' < (2-q)^{-1} \).

To address the remaining case of \( W(x) \sim \pm \left| x \right| \) near \( 0 \), we could employ the arguments of \cite{BS}, where it is shown that while the kernel \( W \) is not \( \lambda \)-convex, the functional \( \mathscr{W} \) for \( d = 1 \) in fact is. This follows an idea in \cite{CDFLS} where the gradient flow selects an appropriate limit of (non-unique) empirical measure solutions of equation \eqref{eq:215}. While not directly applicable, a possible direction towards a generalization for \( d > 2 \) can be found in \cite[Theorem 2.3]{belale11} where well-posedness in \( L^\infty \) is shown for the Newtonian potential for \( d \geq 2 \), while \( W(x) = \pm \left| x \right| \) corresponds to the Newtonian potential for \( d = 1 \).

Summarizing, we could use \cite[Theorem 4.1]{13-Carrillo-Choi-Hauray-MFL} for the case $1<q\leq 2$, and \cite[Theorem 4.3.1]{BS}, for $q=1$ to show well-posedness in our case. However, we want to present a different approach in Section \ref{sec:existence-theory-linfty} where we follow \cite{belale11,BDi08}, in particular the proof of \cite[Theorem 2.9]{BDi08}. We work directly on the level of the pseudo-inverse of \( \mu \), providing unifying arguments for both parameter ranges in question and immediately yielding the formulation of the equation needed for the analysis of the asymptotic behavior, a purpose for which the pseudo-inverse has been very successfully employed (despite its limitation to \( d = 1 \)) as \eg in \cite{lito04,cato04,BDi08,FR1,FR2,ra10}.
 
\subsection{Asymptotic behavior of solutions}
\label{sec:asympt-behav-solut}

In \cite{FR1,FR2,ra10}, it is shown that the asymptotic behavior of equation \eqref{eq:215} depends decisively on the repulsiveness of \( W \). Under strong enough regularity assumptions, convergence can only occur towards sums of Dirac measures, while singular kernels allow uniformly bounded steady states.

In our case, the specific nature of the attraction term
\begin{equation}
  \label{eq:214}
  \nabla \cdot \left[  (\nabla \delta \mathscr V) \mu  \right ] = \nabla \cdot \left[ \left(\nabla \psi_a \ast \omega \right ) \mu  \right ]
\end{equation}
encourages us to look for a more specific description of the steady states. In Section \ref{sec:asymptotics}, we show that for \( q_a = q_r = 2 \), the solution is a traveling wave with \( \mu_0 \) as profile which converges exponentially to match the centers of mass of \( \mu_0 \) and \( \omega \). Moreover, for \( q_a = q_r = 1 \), we are able to confirm the numerical evidence from \cite{FHS12} which suggests that for the whole range \( q_a = q_r \in [1,2) \), there will always be convergence to the given profile \( \omega \). Yet, this being a natural conjecture since we have shown in Corollary \ref{cor:lower-semi-cont} that this is indeed the unique minimizer of the associated energy functional, we did not succeed in adapting our approach in order to prove it.
 
\section{The Pseudo-inverse}
\label{sec:pseudo-inverse-techn}

Here and below $\psi_r(x) = \left| x \right|^{q_r}\), \(\psi_a(x)= |x|^{q_a}$, for $1 \leq q_a,q_r \leq 2$. When considering the symmetric case \( q_a = q_r \), we sometimes just write \( \psi = \psi_a = \psi_r \)

\subsection{Definition and elementary properties}
\label{sec:pseudo-inverse-elem}
In one spatial dimension, we can exploit a special transformation technique which makes equation \eqref{eq:213} much more amenable to estimates in the Wasserstein distance. More precisely, this distance can be explicitly computed in terms of pseudo-inverses.

\begin{definition}[CDF and Pseudo-Inverse]
  \label{def:cdf-pseudo-inverse}
  Given a probability measure $\mu$ on the real line, we define its
  \emph{cumulative distribution function (CDF)} as
  \begin{equation}
    \label{eq:217}
    F_\mu(x) := \mu((-\infty,\,x]), \ x \in \mathbb{R}
  \end{equation}
  and its \emph{pseudo-inverse} as
  \begin{equation}
    \label{eq:218}
    X_\mu(z) := \inf \left\{x \in \mathbb{R} : F_\mu(x) > z\right\}, \ z \in [0,1].
  \end{equation}
\end{definition}

Note that in some cases, $X$ indeed is an inverse of $F$. Namely, if $F$ is \emph{strictly} monotonically increasing, corresponding to \(\mu\) having its support on the whole of $\mathbb{R}$, then \(X \circ F = \mathrm{id}\); if \(F\) is \emph{continuous}, which means that it does not having any point masses, then \(F \circ X = \mathrm{id}\). However, in general we only have
\begin{equation}
  \label{eq:219}
  (X \circ F) (x) \geq x, \ x \in \mathbb{R}, \quad (F \circ X)(z) \geq z, \ z\in[0,1].
\end{equation}

Furthermore, we have the following lemmata:
\begin{lemma}[Substitution formula]
  \label{lem:16} For all $f \in L^1(\mu)$,
  \begin{equation}
    \label{eq:220}
    \int_\mathbb{R} f(x) \, \mathrm{d} \mu(x) = \int_0^1 f(X(z)) \, \mathrm{d} z.
  \end{equation}  
\end{lemma}

\begin{lemma}[Formula for the Wasserstein-distance]
  \label{lem:17} \cite[Section 2.2]{07_Carillo_Toscani_prob-metrics} Let $\mu$ and $\omega$ be two Borel measures with pseudo-inverses $X$ and $Y$, respectively, and \(p \in [1,\infty]\). Then,
  \begin{equation}
    \label{eq:221}
    W_p(\mu,\omega) = \left\| X-Y \right\|_{L^p[0,1]} =
    \begin{cases}
      \left( \int_0^1 \lvert X(z) - Y(z)\rvert^p \, \mathrm{d} z \right)^{1/p} & 1 \leq p < \infty,\\
      \sup_{z \in [0,1]} \left| X(z) - Y(z) \right| & p = \infty.
    \end{cases}
  \end{equation}
\end{lemma}

\subsection{The transformed equation}
\label{sec:pseudo-inverse-trans}

In order to transform equation \eqref{eq:213} in terms of the pseudo-inverse, denote by $\mu$ one of its solutions and by $\omega$ the given datum, as well as by $F$ and $G$ their respective CDFs and by $X$ and $Y$ their pseudo-inverses. Let us further assume for now that equality holds in the inequalities \eqref{eq:219}. Then we can, at least formally, compute the derivatives of these identities.

From $F(t,X(t,z)) = z$, we get by differentiating with respect to time and space, respectively:
\begin{gather}
  \partial_t F(t, X(t,z)) + \partial_x F (t,X(t,z)) \cdot \partial_t X(t,z) = 0, \label{eq:222}\\
  \partial_x F(t, X(t,z)) \cdot \partial_z X(t,z) = 1. \label{eq:223}
\end{gather}
From \eqref{eq:222}, we get
\begin{equation}
  \label{eq:224}
  \partial_t X = \left( -(\partial_x F)^{-1} \cdot \partial_t F\right) \circ X.
\end{equation}
Now we can integrate \eqref{eq:213} in space to get an equation for $\partial_t F$, namely
\begin{equation}
  \label{eq:225}
  \partial_tF = (\psi_a' \ast \omega - \psi_r' \ast \mu)\mu,
\end{equation}
where at the moment we interpret $\mu$ as a density. Using $\partial_x F = \mu$
and combining \eqref{eq:222} and \eqref{eq:223}, we see that
\begin{equation}
  \label{eq:226}
  \partial_t X = - (\psi_a' \ast \omega - \psi_r' \ast \mu) \circ X = - \int_\mathbb{R} \psi_a'(X - y) \, \mathrm{d} \omega(y)
  + \int_\mathbb{R} \psi_r'(X - y) \, \mathrm{d} \mu(y).
\end{equation}
Using the substitution formula of Lemma \ref{lem:16}, we find the formulation which we want to work with:
\begin{equation}
  \label{eq:227}
  \partial_t X(t,z) = - \int_0^1 \psi_a'(X(t,z) - Y(\zeta)) \, \mathrm{d} \zeta + \int_0^1 \psi_r'(X(t,z) - X(t,\zeta)) \, \mathrm{d} \zeta.
\end{equation}
In the case \(q_r = q_a = 1\), where we assume both \(\mu\) and \(\omega\) to be absolutely continuous and \(\psi'(x) = \operatorname{sgn}(x)\), this equation has a particular structure. Namely, by using
\begin{align}
  \label{eq:228}
  \int_{\mathbb{R}} \psi'(x - y) \, \mathrm{d} \mu(y) &= \int_\mathbb{R} \operatorname{sgn}(x - y) \, \mathrm{d} \mu (y) 
  = \mu((-\infty,x]) - \mu((x,\infty))\\
  &= 2\mu((-\infty,x]) - \mu(\mathbb{R}) = 2 F(x) - 1,\label{eq:328}
\end{align}
the equation for \(q = 1\) reads as
\begin{equation}
  \label{eq:229}
  \partial_{t}X(t,z) = 2\left[F(t,X(t,z)) - G(X(t,z))\right],
\end{equation}
where \(F(t,x)\) denotes the CDF of $\mu(t)$.

Note that these formal computations can sometimes be made rigorous, as we shall do in the next section. There, in the case \(q_a, q_r \in [1,2]\), we construct under certain additional assumptions a solution to equation \eqref{eq:227} and prove that for all \(t \in [0,\infty)\) there is an associated measure $\mu(t)$ with pseudo-inverse $X(t,.)$ which fulfills \eqref{eq:213} in a distributional sense.

\section{Existence of solutions}
\label{sec:existence-theory-linfty}

Let \(q_a, q_r \in [1,2]\). Under certain further restrictions on $\omega$ and $\mu_0$, we can employ a fixed-point iteration for the pseudo-inverse in $L^\infty([0,1])$ to find solutions to equation \eqref{eq:227}, corresponding to distributional solutions of \eqref{eq:213}. We also want to allow the mass of \( \omega \) to be different from \( 1 \).

\begin{theorem}[Existence of solutions]
  \label{thm:linfty-ext}
  Let $\omega,\mu_0 \in L^\infty_c(\mathbb{R})$, the space of functions in \( L^\infty(\mathbb{R}) \) which are compactly supported, such that \( \omega,\mu_0 \geq 0 \) almost everywhere and
  \begin{equation}
    \label{eq:361}
    \int_{\mathbb{R}} \mu_0(x) \diff x = 1, \quad \int_{\mathbb{R}} \omega(x) \diff x \in (0,\infty)
  \end{equation}
  Then there is a unique curve
  \begin{equation}
    \label{eq:231}
    X(.,.) \in C^{1}([0,\infty),L^\infty([0,1]))
  \end{equation}
  such that
 \begin{enumrom}
 \item $X(0,.)$ is the pseudo-inverse of $\mu_0$;
 \item for every $t \in [0,\infty)$, $X(t,.)$ is the pseudo-inverse of a probability measure \(\mu(t)\);
 \item for almost all $t \in [0,\infty)$ and every $z \in[0,1]$, the curve $X(t,.)$ fulfills the pseudo-inverse formulation \eqref{eq:227} if we interpret \( \psi_{a,r}' = \sgn(x) \) if \( q_{a,r} = 1 \);
 \item the curve \(\mu(t)\) is a distributional solution of the original equation \eqref{eq:213}, i.e.\@ for all $\varphi \in C^{\infty}_{c}([0,\infty) \times \mathbb{R})$, it fulfills the weak formulation
   \begin{align}
     \label{eq:232}
     \leadeq{-\int_{0}^{\infty} \int_{\mathbb{R}} \partial_{t}\varphi(t,x) \, \mathrm{d} \mu(t,x) \, \mathrm{d} t - \int_{\mathbb{R}} \varphi(0,x) \, \mathrm{d} \mu(0,x)}\\
     ={}&\int_{0}^{\infty} \int_{\mathbb{R}} \partial_{x}\varphi(t,x) \cdot \left(\psi_r' \ast \mu(t,.)\right)(x) \, \mathrm{d} \mu(t,x) \, \mathrm{d} t \nonumber\\
     &- \int_{0}^{\infty} \int_{\mathbb{R}} \partial_{x} \varphi(t,x) \cdot \left(\psi_a' \ast \omega\right)(x)\, \mathrm{d} \mu(t,x) \, \mathrm{d} t.
   \end{align}
 \end{enumrom}
\end{theorem}

This result may appear of relative novelty, as it may also be obtained by taking advantage of the smoothness and confining properties of the linear attractive term $ \nabla \cdot \left[ \left(\nabla \psi_a \ast \omega \right ) \mu \right ]$, combined with the well-posedness of the repulsive term from \cite[Theorem 7]{BCLR}, for the case $1<q_r\leq 2$, and \cite[Theorem 4.3.1]{BS}, for $q_r=1$ respectively.

One ingredient for the proof of Theorem \ref{thm:linfty-ext} is the following lemma.

\begin{lemma}
  \label{lem:18}
  Let \(\omega \in L^\infty(\mathbb{R}) \cap L^1(\mathbb{R})\) such that \( \omega \geq 0 \). Then, \(\psi_a' \ast \omega\) is Lipschitz-continuous.
\end{lemma}

\begin{proof}
  For \(q_a = 1\), remember that we arbitrarily set \(\psi'(x) = \sgn(x)\) and that we explicitly computed the convolution \(\psi_a' \ast \omega\) in \eqref{eq:228}, namely
  \begin{equation}
    \label{eq:253}
    \psi_a' \ast \omega (x) = 2 \int_{-\infty}^{x} \omega(y) \, \mathrm{d} y - \left\| \omega \right\|_{1},
  \end{equation}
  which is obviously Lipschitz-continuous if \(\omega \in L^{\infty}(\mathbb{R})\).

  For \(q_a \in (1,2)\), we consider $\psi_a''(x) =q(q-1)|x|^{q-2}$ and its convolution with $\omega$, and we show that it is uniformly bounded, hence $x \to \psi_a' \ast \omega(x)$ is Lipschitz continuous. As $\psi_a''$ is integrable on $[-1,1]$ and bounded by $1$ on $\mathbb R \setminus [-1,1]$, one gets
  \begin{align}
    \label{eq:384}
    \left| \psi_a'' \ast \omega(x) \right| = {} & \int_{-1}^{1} q_a(q_a-1) \left| y \right|^{q_a-2} \omega (x-y) \diff y\\
    & + q_a(q_a-1) \int_{\mathbb{R} \setminus [-1,1]} \left| y \right|^{q_a-2} \omega(x-y) \diff y\\
    \leq {} &  q_a(q_a-1) \left( \frac{2}{q_a-1} \left\| \omega \right\|_\infty + \left\| \omega \right\|_1 \right).
  \end{align}
\end{proof}

We follow the lines of the proof of \cite[Theorem 2.9]{BDi08}, which means that below, we define a suitable operator whose fixed point will be a solution of \eqref{eq:227}, and then we show that this determines a solution to \eqref{eq:213}. As elements of novelty, two major differences with respect to \cite[Theorem 2.9]{BDi08} are in order:

\begin{itemize}
\item We implement a suitable time rescaling, adapted to the lack of smoothness in $0$ of $W$ for gaining contractivity of the operator, see the exponential term in \eqref{eq:256} below; in particular, as $\psi_{a,r}'$ is {\it not} Lipschitz we need to establish contractivity of the operator by a more careful analysis which requires some technicalities, see Step \ref{item:linfty-contractive} below.
\item Our way to return to solutions of the original equation \eqref{eq:213} is more direct and it does not go through a smooth approximation argument, see Step \ref{item:linfty-dist-form} below.
\end{itemize}

\begin{proof}[Proof (Theorem \ref{thm:linfty-ext})]
  For now, we assume \(q_r \in (1,2]\). The arguments in the case \(q_r = 1\) are in fact even simpler and we elaborate on them afterwards in Step \ref{item:linfty-changes-q-1}.

  In the following, let \( \alpha > 0 \) such that
  \begin{equation}
    \label{eq:230}
    \omega(x) \leq \alpha^{-1}, \ \mu_0(x) \leq \alpha^{-1}, \ \text{for a.e.\@ } x \in \mathbb{R}.
  \end{equation}
    
  \begin{enumpara}[start=0]
     
  \item \emph{(Definition of the operator)} Let \(T > 0\) and $V(x) = \psi_a \ast \omega(x), \ x \in \mathbb{R}$. By the \( L^\infty \) assumption on \( \omega \) together with Lemma \ref{lem:18}, \(V'\) is Lipschitz-continuous. Denote its Lipschitz-constant by $\lambda$ and set
    \begin{equation}
      \label{eq:255}
      \widetilde{V}(x) := V(x) - \frac{\lambda}{2}|x|^2, \ x \in \mathbb{R}.
    \end{equation}
    Now, define the operator
    \begin{align}
      \label{eq:256}
      S[X](t,z) := {} &\exp(-\lambda t)X_0(z) \notag\\ &+\int_0^t \exp(-\lambda(t-s)) \left[ \int_0^1
        \psi_r'(X(s,z) - X(s,\zeta)) \, \mathrm{d} \zeta - \widetilde{V}'(X(s,z)) \right] \, \mathrm{d} s,
    \end{align}
    on the set
    \begin{equation*}
      \label{eq:257}
      \mathcal{B} := \left\{ X(.,.) \in C\big([0,T], L^\infty([0,1])\big) :
      \begin{gathered}
        X(t,.) \text{ has a right-continuous representative}\\
        \text{and fulfills \eqref{eq:258}}
      \end{gathered}
      \right\},
    \end{equation*}
    where
        \begin{equation}
      \label{eq:258}
      \begin{gathered}
        \frac{1}{h} \left(X(t,z+h) - X(t,z) \right) \geq \alpha\exp(-\lambda t) \\
        \text{for all } h \in (0,1) \text{ and } z \in [0,1-h],
      \end{gathered}  \tag{SL}
    \end{equation}
        endowed with the norm 
    \begin{equation}
      \label{eq:259}
      \lVert X \rVert_{\mathcal{B}} := \sup \left\{ \exp(\lambda t) \lVert X(t,.)\Vert
        _{L^\infty} : t \in [0,T]\right\}.
    \end{equation}
    Notice that \(\widetilde{V}\) is concave and hence \(\widetilde{V}'\) decreasing. $\mathcal B$ is actually closed in $L^\infty([0,1])$: Given a convergent sequence \( X_n \xrightarrow{\mathcal{B}} X \), we first remark that despite the exponential rescaling, we still have uniform convergence of \( X_n(t,.) \) and therefore that \( X(t,.) \) is continuous with values in \( L^\infty([0,1]) \). Now, convergence in \( L^\infty \) at each point \( t \) means that right-continuity is preserved via an \( \varepsilon/3 \) argument. Finally, the expression \eqref{eq:258} is continuous in \( X(t,z) \) and \( X(t,z+h) \) for each \( h \), whence we can also pass to the limit there.
    \item \emph{($S$ maps $\mathcal{B}$ into $\mathcal{B}$)} \label{item:s-b-into-b} Firstly, for \(X \in \mathcal{B}\), the continuity of \(t \mapsto S[X](t,.)\) from \([0,T]\) to \(L^{\infty}([0,1])\) follows from the continuity of the integral defining \(S\) and by the continuity of the functions involved, which attain their maximum on the set \([0,1]\).
    
    Secondly, ad slope condition: Let $X \in \mathcal{B}$ (in particular non-decreasing), $h > 0$, $z \in [0, 1 - h]$. By using the slope condition on $X_0$, the fact that $\psi_r'$ is increasing, and that $\widetilde{V}'$ is decreasing one obtains
    \begin{equation}
      \frac{1}{h}\left[ S[X](t,z + h) - S[X](t,z) \right] \geq\alpha\exp(-\lambda t).\label{eq:261}
    \end{equation}
    
  \item \emph{($S$ is contractive)} \label{item:linfty-contractive} Let
    $X,\widetilde{X} \in \mathcal{B}$. Then,
    \begin{align}
      \leadeq{\exp(\lambda t) \cdot \left| S[\widetilde{X}](t,z) - S[X](t,z)\right|} \\
       \leq {} & \int_0^t \exp(\lambda s) \int_0^1
      \left| \psi_r'(\widetilde{X}(s,z) - \widetilde{X}(s,\zeta)) - \psi_r'(X(s,z) - X(s,\zeta))\right|
      \, \mathrm{d} \zeta \, \mathrm{d} s\label{eq:262}\\
       &+ \int_0^t \exp(\lambda s) \left| \widetilde{V}'(\widetilde{X}(s,z)) -
        \widetilde{V}'(X(s,z))\right| \, \mathrm{d} s. \label{eq:263}
    \end{align}

    For the term \eqref{eq:263} we can simply use the Lipschitz-continuity of $\widetilde{V}'$ with Lipschitz-constant \(2\lambda\), which yields
    \begin{align}
      \label{eq:264}
      \int_0^t \exp(\lambda s) \left| \widetilde{V}'(\widetilde{X}(s,z)) -
        \widetilde{V}'(X(s,z))\right| \, \mathrm{d} s \leq 2\lambda \int_0^t \lVert \widetilde{X} - X
      \rVert_{\mathcal{B}} \, \mathrm{d} s \leq 2\lambda t \lVert \widetilde{X} - X
      \rVert_{\mathcal{B}}.
    \end{align}

    For the term \eqref{eq:262}, we observe that $\psi_r'$ is \emph{not} Lipschitz-continuous. However, the slope assumption allows us to see that the part of the integral where this gets critical, i.e.\@ where $X(s,z) - X(s,\zeta)$ is near zero, is small: Let us first assume that $\zeta \leq z$ and without loss of generality
    \begin{equation}
      \label{eq:265}
      \widetilde{X}(s,z) - \widetilde{X}(s,\zeta) \geq X(s,z) - X(s,\zeta).
    \end{equation}
    Then both evaluations of $\psi'$ lie on the positive branch of $\psi'$, while we can bound the difference of the operands by
    \begin{equation}
      \label{eq:266}
      \left| \big|\widetilde{X}(s,z) - \widetilde{X}(s,\zeta)\big| - \big|X(s,z) -
          X(s,\zeta)\big|\right| \leq 2 \underbrace{\sup_{\zeta \in [0,1]}
          \left| X(s,\zeta)
        - \widetilde{X}(s,\zeta)\right|}_{:=\delta(s)},
    \end{equation}
    while for each of them, by \eqref{eq:258}, we have
    \begin{align}
      \label{eq:267}
      \widetilde{X}(s,z) - \widetilde{X}(s,\zeta) &\geq \alpha\e^{-\lambda t}(z - \zeta), \quad
      X(s,z) - X(s,\zeta) \geq \alpha\e^{-\lambda t}(z - \zeta).
    \end{align}
    Hence, the integrand can be estimated by
    \begin{equation}
      \label{eq:268}
      \Big|\psi_r'(\underbrace{\widetilde{X}(s,z) - \widetilde{X}(s,\zeta)}_{\sim x +
        \eta}) - \psi_r'(\underbrace{X(s,z) - X(s,\zeta)}_{\sim x})\Big|
      \leq  \sup_{\substack{\widetilde{x} \leq x\\0 \leq \eta \leq \widetilde{\eta}}}
      \left(\psi_r'(x + \eta) - \psi_r'(x)\right),
    \end{equation}
    where $\widetilde{x} := \alpha\exp(-\lambda t)(z-\zeta)$ and $\widetilde{\eta} := 2\delta(s)$, and we used the monotonicity of \(\psi_r'\) to leave out the modulus.
    \begin{figure}[t]
      \centering
      \includegraphics{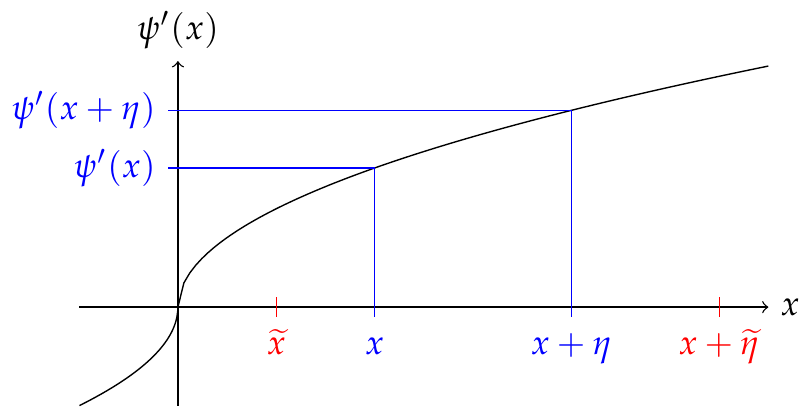}
      \caption{\label{fig:optimise-psi} The biggest distance in \eqref{eq:268}
        is attained for \(\psi'(\widetilde{x} + \widetilde{\eta}) - \psi'(\widetilde{x})\)}
    \end{figure}
    To visualise where this supremum is attained, one might have a look at Figure \ref{fig:optimise-psi}, which is actually for \(x = \widetilde{x}\), and \(\eta = \widetilde{\eta}\): Since \( \psi_r'(x) = q_r \left| x \right|^{q_r-2}x \) is strictly monotonically increasing and positive for \( x > 0 \), it is clear that for fixed \( x \), \( \eta = \widetilde{\eta} \) maximizes the expression. Furthermore, setting \( f(x) := \psi_r'(x+\eta) - \psi_r'(x) \), by \( f'(x) = \psi_r''(x+\eta) - \psi_r''(x) \) and \( \psi_r''(x) = q_r (q_r-1) \left| x \right|^{q-2} \), we see that \( f \) is monotonically decreasing for \( x > 0 \), hence the maximum \wrt \( x \) is attained for the leftmost point \( x = \widetilde{x} \).

So, inserting \(\widetilde{x} = \alpha\e^{-\lambda t}(z - \zeta)\) and \(\widetilde{\eta} = \delta(s)\) in \eqref{eq:268}, and using a similar argument also for $\zeta > z$, we eventually obtain
    \begin{align}
      \leadeq{\left| \psi_r'(\widetilde{X}(s,z) - \widetilde{X}(s,\zeta)) - \psi_r'(X(s,z) - X(s,\zeta))\right|} \nonumber\\
      \leq {} &\psi_r'(\alpha\e^{-\lambda t}(|z-\zeta|) + 2\delta(s) - \psi_r'(\alpha\e^{-\lambda t}(|z - \zeta|)),
    \end{align}
    for all $\zeta \neq z$. We can now use the mean value theorem to get a linear estimate for all $\zeta \neq z$. If for example $\zeta < z$, then
    \begin{align}
      \leadeq{\left| \psi_r'(\widetilde{X}(s,z) - \widetilde{X}(s,\zeta)) - \psi_r'(X(s,z) - X(s,\zeta))\right|} \nonumber \\ 
      \leq {} & 2q_r(q_r-1) \, \delta(s) \, \sup \left\{
        \eta^{q_r - 2} : \eta \in
        [\alpha\e^{-\lambda t}(z-\zeta),\alpha\e^{-\lambda t}(z-\zeta) + 2\delta(s)]
      \right\} \nonumber \\ 
      = {} & 2q_r(q_r - 1) (\alpha\e^{-\lambda t}(z-\zeta))^{q_r - 2} \delta(s)\label{eq:269}
    \end{align}
    and similarly for $\zeta > z$. Integrating \eqref{eq:269} with respect to $\zeta$ yields
    \begin{align}
      \leadeq{\int_0^1 \left| \psi_r'(\widetilde{X}(s,z) - \widetilde{X}(s,\zeta)) -
          \psi_r'(X(s,z) -
          X(s,\zeta))\right|\, \mathrm{d} \zeta} \nonumber \\
      \leq {} & 2q_r(q_r-1) \delta(s)
      \int_0^1\left| \alpha\e^{-\lambda t}(z-\zeta)\right|^{q-2}\, \mathrm{d} \zeta\nonumber\\
      = {} & 2q_r  \delta(s) (\alpha\e^{-\lambda t})^{q_r-2}\left[z^{q_r-1} + (1 - z)^{q_r-1} \right]\nonumber\\
      \leq {} & C \delta(s)
    \end{align}
    with a suitable $C > 0$, as the factors apart from \(\delta(s)\) are bounded for $z \in [0,1]$ and $t \in [0,T]$. Thence,
    \begin{equation}
      \label{eq:270}
      \int_0^t \exp(\lambda s) \int_0^1 \left| \psi_r'(\widetilde{X}(s,z) -
        \widetilde{X}(s,\zeta)) - \psi_r'(X(s,z) - X(s,\zeta))\right| \, \mathrm{d} \zeta \, \mathrm{d} s \leq TC \lVert X
      - \widetilde{X} \rVert_{\mathcal{B}},
    \end{equation}
    in total implying that for $T$ small enough, $S$ is a contraction.

    Combining the previous steps, we find a unique fixed point $X$ of $S$ using the Banach fixed point theorem, i.e.\@ an \(X \in \mathcal{B}\) such that
    \begin{align}
      \label{eq:271}
        X(t,z) = {} & \exp(-\lambda t)X_0(z) \notag\\ & +\exp(-\lambda t)\int_0^t \underbrace{\exp(\lambda s)) \left[ \int_0^1
          \psi'(X(s,z) - X(s,\zeta)) \, \mathrm{d} \zeta - \widetilde{V}'(X(s,z)) \right]}_{\text{integrand}} \, \mathrm{d} s,
    \end{align}
    where the integrand is continuous as a mapping \([0,t]\rightarrow L^{\infty}([0,1])\), again by the continuity of the involved functions and the $L^\infty$-property of $X$. Hence, the right-hand side has the desired \(C^{1}\)-regularity on \([0,T]\) and by equality in \eqref{eq:271}, so has \(X\).
  \item \emph{(Global existence)}\label{item:linfty-global-existence} Differentiating \eqref{eq:271} with respect to time directly yields
    \begin{align}
      \partial_tX(t,z) 
      = \int_0^1 \psi_r'(X(t,z) - X(t,\zeta)) \, \mathrm{d} \zeta - V'(X(t,z)),\label{eq:272}
    \end{align}
    hence $X$ fulfills also the desired equation. Global existence is achieved by preventing a blowup of the $L^\infty$-norm, which we rule out by estimating the growth: by Lipschitz-continuity of $V'$, the estimate $\left| \psi_r'(x)\right| \leq q_r \cdot (1 + \left| x\right|)$, and by Gronwall‘s inequality, we get that
    \begin{equation}
      \label{eq:273}
      \lVert X(t,.) \rVert_{L^\infty} \leq \left(\lVert X_0 \rVert_{L^\infty} + C_{1}t\right) \, \exp(C_{2}t).
    \end{equation}
  \item \label{item:linfty-dist-form} \emph{(Distributional formulation)} Firstly, for every \(t \in [0,\infty)\), \(X(t,.)\) is a right continuous increasing function and hence can be used to construct a probability measure on \(\mathbb{R}\): For this, apply the pseudo-inverse transform to get a right-continuous increasing function on \(\mathbb{R}\) and then use the well-known correspondence between probability measures and CDFs.

    Secondly, let \(\varphi \in C^{\infty}_{c}([0,\infty)\times \mathbb{R})\). As we have \(C^{1}\)-regularity of the solution curve, combining this with the fundamental theorem of calculus, Fubini’s theorem and the compactness of the support of \(\varphi\), we see that
    \begin{align}
      \label{eq:274}
      \int_{0}^{\infty} \int_{0}^{1} \frac{\mathrm{d}}{\mathrm{d} t} \left[\varphi(t, X(t,z))\right] \, \mathrm{d} z \, \mathrm{d} t {} & =
      -\int_{0}^{1}
      \varphi(0,X(0,z)) \, \mathrm{d} z &&\\
            & = -\int_{\mathbb{R}} \varphi(0,x) \, \mathrm{d} \mu(0,x) &&\text{by Lemma \ref{lem:16}},
    \end{align}
    where the use of Lemma \ref{lem:16} is justified because \(\varphi(0,.)\) is bounded and therefore in \(L^{1}(\mu(0))\).

    On the other hand, again by the regularity of the curves and the chain rule, for all $t \in [0,\infty)$ and almost all $z \in [0,1]$,
    \begin{equation}
      \label{eq:275}
      \frac{\mathrm{d}}{\mathrm{d} t} \left[\varphi(t, X(t,z))\right] = \partial_{t}\varphi(t,X(t,z)) +
      \partial_{x}\varphi(t,X(t,z)) \cdot \partial_{t}X(t,z).
    \end{equation}
    The integral over the first term in \eqref{eq:275} yields
    \begin{equation}
      \label{eq:276}
      \int_{0}^{\infty} \int_{0}^{1} \partial_{t}\varphi(t,X(t,z)) \, \mathrm{d} z \, \mathrm{d} t = \int_{0}^{\infty} \int_{\mathbb{R}}
      \partial_{t}\varphi(t,x) \, \mathrm{d} \mu(t,x) \, \mathrm{d} t,
    \end{equation}
    where we again used Lemma \ref{lem:16} as above. By inserting equation \eqref{eq:272} for $\partial_{t}X$, the integral over the second term in \eqref{eq:275} becomes
      \begin{align}
        \leadeq{\int_{0}^{\infty} \int_{0}^{1} \partial_{x}\varphi(t,X(t,z)) \cdot \partial_{t}X(t,z) \, \mathrm{d}
          z \, \mathrm{d} t}\\
        = {} &\int_{0}^{\infty} \int_{0}^{1} \partial_{x}\varphi(t,X(t,z)) \int_{0}^{1} \psi_r'(X(t,z)
        - X(t,\zeta)) \, \mathrm{d} \zeta - V'(X(t,z))
        \, \mathrm{d} z \, \mathrm{d} t\\
        = {} & \int_{0}^{\infty} \int_{0}^{1} \partial_{x}\varphi(t,X(t,z))\Big[(\psi_r' \ast
        \mu(t,.))(X(t,z)) - V'(X(t,z)) \Big] \, \mathrm{d} z \, \mathrm{d} t\\
         = {} & \int_{0}^{\infty} \int_{\mathbb{R}} \partial_{x}\varphi(t,x) \cdot (\psi_r' \ast \mu(t,.))(x) \, \mathrm{d}
        \mu(t,x) \, \mathrm{d} t \\
         & - \int_{0}^{\infty} \int_{\mathbb{R}} \partial_{x}\varphi(t,x) \cdot V'(x) \, \mathrm{d}
        \mu(t,x) \, \mathrm{d} t,\label{eq:277}
      \end{align}
      which is the desired equation. The use of Lemma \ref{lem:16} here is justified because the involved measures are compactly supported, yielding a bound on their second moment; this results in \(\psi_r' \ast \mu(t) \in L^{1}(\mu(t))\) and \(V' \in L^{1}(\mu(t))\), which we then combine with \(\partial_{x}\varphi(t,.) \in L^{\infty}(\mathbb{R})\) to see that the integrands in the last line of \eqref{eq:277} are in \(L^{1}(\mu(t))\).

    \item \label{item:linfty-changes-q-1} \emph{(Adjustments for $q=1$)} We first simplify the pseudo-inverse formulation \eqref{eq:229}: Note that for a strictly increasing pseudo-inverse \(X(t,.)\) with associated measure \(\mu(t)\) and CDF \(F(t,.)\), \(X(t,.)\) is the right-inverse of \(F(t,.)\), which means that we can write \eqref{eq:229} as
      \begin{equation}
        \label{eq:278}
        \partial_{t}X(t,z) = 2 F(t,X(t,z)) - 1 - V'(X(t,z)) = 2 z - 1 - V'(X(t,z)).
      \end{equation}
      We now apply the previous arguments to find a solution to this equation and afterwards justify that \(X(t,.)\) stays strictly increasing, allowing us to go back in the above equation \eqref{eq:278}.

      As \(\omega\) was assumed to be absolutely-continuous with its density belonging to \(L^{\infty}(\mathbb{R})\), the attraction potential \(V'(.)\) is Lipschitz-continuous (see Lemma \ref{lem:18}). Therefore, again denoting its Lipschitz-constant by \(\lambda\), we can define the operator \(S\) analogously to \eqref{eq:256} using the simplified form of \eqref{eq:278} as the right-hand side, i.e.\@
      \begin{align}
        \label{eq:279}
        S[X](t,z) := {} &\exp(-\lambda t)X_0(z) \notag\\ &+ \int_0^t \exp(-\lambda(t-s)) \left[ 2z - 1 - \Big(V'(X(s,z)) - \lambda X(s,z)\Big)\right] \, \mathrm{d} s.
      \end{align}
      Step \ref{item:s-b-into-b} may then still be applied as the integrand in \eqref{eq:279} is continuous and the monotonicity arguments used in \eqref{eq:261} remain true, as well. This provides us with the strict monotonicity of \(X(t,.)\) for all \(t\), so we can reverse the simplification \eqref{eq:278} as intended.

      Now, Step \ref{item:linfty-contractive} is actually much easier, since the mapping
      \begin{equation}
        \label{eq:280}
        X \mapsto 2z - 1 - \left(V'(X) - \lambda X\right), \ X \in \mathbb{R}
      \end{equation}
      is obviously Lipschitz-continuous in \(X\). Finally, Steps \ref{item:linfty-global-existence} and \ref{item:linfty-dist-form} work analogously and can be followed verbatim. \qedhere
    \end{enumpara}
\end{proof}

\section{Asymptotic behavior}
\label{sec:asymptotics}

\subsection{The case $q_r = q_a = 2$}
\label{sec:case-q-2}

Let us first look at the equation for $q = q_a = q_r = 2$. In that case, the corresponding potentials are \(\lambda\)-convex along generalised geodesics, so we could use the extensive theory of solutions of \cite{AGS08}. As we want to work with the pseudo-inverse formulation \eqref{eq:227}, we shall however use the notion of solutions provided by Theorem \ref{thm:linfty-ext}. Here, the solutions of \eqref{eq:227} exhibit a traveling wave behavior: The profile of the initial value \(\mu_{0}\) moves along the real line in the direction of the first moment of \(\omega\), which can be justified as follows:

Assume that the assumptions of Theorem \ref{thm:linfty-ext} are satisfied. We compute $\psi'(x) = 2x$ and remark that equation \eqref{eq:227} simplifies significantly due to the linearity of $\psi'$:
\begin{align}
  \partial_t X(t,z) &= - \int_0^1 2(X(t,z) - Y(\zeta)) \, \mathrm{d} \zeta + \int_0^1 2(X(t,z) - X(t,\zeta))
  \, \mathrm{d} \zeta \\
  &= 2\int_0^1 Y(\zeta) \, \mathrm{d} \zeta - 2\int_0^1 X(t,\zeta)
  \, \mathrm{d} \zeta \\
  &= 2 \left( \rho_{\omega} - \rho_{\mu(t)} \right),\label{eq:233}
\end{align}
where by $\rho_{\mu}$ we denote the first moment of $\mu$, which we can identify with $\int_0^1 X(z) \, \mathrm{d} z$ by Lemma \ref{lem:16}. The expression \eqref{eq:233} is independent of $z$, so \(X\) is left untouched by the evolution except for a vertical translation, i.e.
\begin{equation}
  \label{eq:234}
  X(t,z) = X(0,z) + 2 \int_{0}^{t} \int_0^1 Y(\zeta) - X(t,\zeta)
  \, \mathrm{d} \zeta. 
\end{equation}
To see where our initial distribution is pulled, consider the translated quantity
\begin{equation}
  \overline{\mu} (t,.) := \mu(t,.) - \omega(.)\label{eq:235}
\end{equation}
and integrate \eqref{eq:233} in space to get
\begin{equation}
  \label{eq:236}
  \frac{\mathrm{d}}{\mathrm{d} t} \rho_{\overline{\mu}(t)} = -2 \rho_{\overline{\mu}(t)}.
\end{equation}
As a result, we get exponential convergence of the first moment of $\mu(t)$ to that of \(\omega\), while the initial profile $\mu_0$ is preserved.

\subsection{The case $q_r = 1$}
\label{sec:case-q-1}

\begin{figure}[t]
  \centering
  \subfigure[\( \psi_a' \ast \omega \) determines \( x_0,a,b \)]{\includegraphics[]{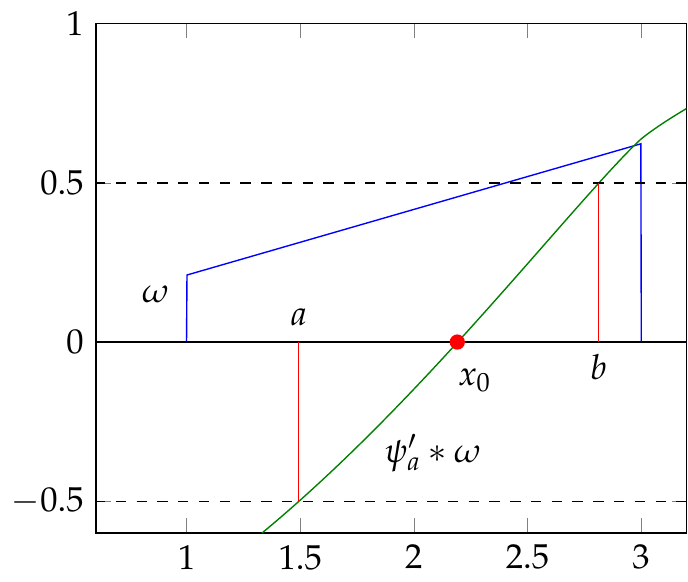}}
  \hspace{.5cm}
  \subfigure[\( \psi_a'' \ast \omega \) determines the shape]{\includegraphics[]{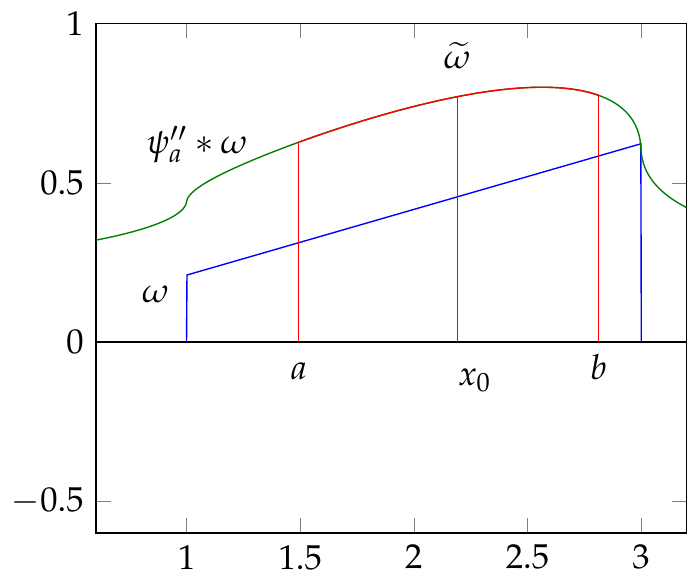}}
  \caption{Example of an \( \widetilde{\omega} \) for \( q_a = 1.5 \)}
  \label{fig:q1plot}
\end{figure}

Now, let us consider the case \( q_r = 1 \), corresponding to \( \psi_r \) being a multiple of the Newtonian potential in 1D, \ie the solution to \( f'' = 2\delta_0 \). Here, we again want to allow \( \omega \) to have a mass  different from \( 1 \) and just ask for \( \int_{\mathbb{R}} \omega(x) \diff x =: m > 0 \). The steady state will now be a suitable cut-off of \( \psi_a'' \ast \omega \) (which stays also valid for \( \psi_a'' = 2\delta_0 \) if \( q_a = 1 \)), which we shall firstly compute in Proposition \ref{prp:steady-stat-equat} and then show convergence of the gradient flow towards it for \( t\rightarrow\infty \) in Theorem \ref{thm:conv-given-datum}. See Figure \ref{fig:q1plot} for an example of the resulting \( \widetilde{\omega} \), where we used the notation of the following Proposition \ref{prp:steady-stat-equat}.

\begin{proposition}[Steady states of the equation]
  \label{prp:steady-stat-equat}
  Let \( \widetilde{\omega}, \omega \in L^\infty_c(\mathbb{R}) \) such that \( \widetilde{\omega}, \omega \geq 0 \) and
  \begin{equation}
    \label{eq:359}
    \int_{\mathbb{R}}  \diff \widetilde{\omega}(x) = 1, \quad \int_{\mathbb{R}}  \diff \omega(x) = m > 0.
  \end{equation}
  and denote their associated pseudo-inverses by \( \widetilde{Y} \) and \( Y \), respectively.

  If \( \widetilde{Y} \) fulfills the steady state equation
  \begin{equation}
    \label{eq:360}
    0 = - \int_0^1 \psi_a'(\widetilde{Y}(t,z) - Y(\zeta)) \, \mathrm{d} \zeta + \int_0^1 \psi_r'(\widetilde{Y}(z) - \widetilde{Y}(\zeta)) \, \mathrm{d} \zeta,
  \end{equation}
  and we define \( x_0, a, b \in \mathbb{R} \) such that
  \begin{equation}
    \label{eq:363}
    \psi_a' \ast \omega (x_0) = 0, \quad 
    \int_{a}^{x_0} \psi_a'' \ast \omega(x) \diff x = 1, \quad 
    \int_{x_0}^{b} \psi_a'' \ast \omega(x) \diff x = 1,
  \end{equation}
  then
  \begin{align}
    & \text{for } q_a > 1, \quad && \widetilde{\omega}(x) = \frac{1}{2} \psi_a'' \ast \omega(x) \, 1_{[a,b]}(x) \text{ a.e.}, \label{eq:364}
\\
    & \text{for } q_a = 1 \text{ and } m \geq 1, \quad && \widetilde{\omega}(x) = \frac{1}{2} \psi_a'' \ast \omega(x) \, 1_{[a,b]}(x) = \omega(x) \, 1_{[a,b]}(x) \text{ a.e.} \label{eq:365}\\
    & \text{for } q_a = 1 \text{ and } m < 1, && \text{ there is no } \widetilde{\omega} \text{ fulfilling the steady state equation \eqref{eq:360}.}\label{eq:366}
  \end{align}
\end{proposition}

\begin{proof}
  Let \( V(x) := \psi_a \ast \omega (x) \), so \( V'(x) = \psi_a' \ast \omega(x) \). For \( q_a > 1 \), \( V' \) is strictly monotonically increasing, continuous and fulfills \( V'(x) \rightarrow \pm \infty \) for \( x \rightarrow \pm \infty \). For \( q_a = 1 \), by the calculation \eqref{eq:328},
  \begin{equation}
    \label{eq:329}
    V'(x) = 2G(x) - m, \quad x \in \mathbb{R},
  \end{equation}
  where
  \begin{equation}
    \label{eq:330}
    G(x) := \int_{-\infty}^{x} \omega(x) \diff x,
  \end{equation}
  so \( V' \) is monotonically increasing, continuous and fulfills \( V'(x) \rightarrow \pm m \) for \( x \rightarrow \pm \infty \). 

  Again using \eqref{eq:328} and the fact that \( \widetilde{Y} \) by assumption has no point masses, the steady state equation \eqref{eq:360} now reads as
  \begin{equation}
    \label{eq:331}
    0 = 2z - 1 - V'(\widetilde{Y}(z)) \Leftrightarrow V'(\widetilde{Y}(z)) = 2z - 1, \quad z \in [0,1].
  \end{equation}
  In the case \( q_a > 1 \), the inverse \( (V')^{-1} \) is well-defined, and its application yields
  \begin{equation}
    \label{eq:332}
    \widetilde{Y}(z) = (V')^{-1}(2z - 1).
  \end{equation}
  Since \( z \mapsto (V')^{-1}(2z-1) \) here is strictly monotonically increasing as well, we can compute the associated CDF \( \widetilde{G} \) to \( \widetilde{Y} \) as its inverse, yielding \begin{equation}
    \label{eq:333}
    \widetilde{G}(x) = 
    \begin{cases}
      0, & V'(x) < -1,\\
      \frac{1}{2} \left( V'(x) + 1 \right), & V'(x) \in [-1,1],\\
      1, & V'(x) > 1.
    \end{cases}
  \end{equation}
  So the steady state will have its median where \( V'(x) \) has its unique zero and its density will then coincide with \( V''(x)/2 \), extending in both directions from the median until mass \( 1 \) is reached, which proves \eqref{eq:364}.

  For \( q_a = 1 \), the steady state equation is
  \begin{equation}
    \label{eq:335}
    2G(\widetilde{Y}(z)) - m = 2z - 1 \Leftrightarrow G(\widetilde{Y}(z)) = z + \frac{m - 1}{2}.
  \end{equation}
  We see that for \( m < 1 \), due to \( G(x) \in [0, m] \), this will not have a solution for \( \widetilde{Y} \) in the whole range \( z \in [0,1] \), proving \eqref{eq:366} (we will see in Remark \ref{rem:van-mass} that indeed for those \( z \) where it has, there is attraction towards that profile and the rest of the mass travels towards \( \pm \infty \)). For \( m \geq 1 \), there is a unique solution due to the fact that \( G \) only fails to be invertible for those right-hand side values where the pseudo inverse \( Y \) of \( \omega \) has a jump, which corresponds to a hole in the support of \( \omega \). Since this can only occur at at most countably \( z \), meaning its complementary set is dense, \( \widetilde{Y} \) is uniquely determined by the condition of being right-continuous, resulting in the CDF \( \widetilde{G}(x) \)
  \begin{equation}
    \label{eq:336}
    \widetilde{G}(x) = 
    \begin{cases}
      0, & G(x) < (m - 1)/2,\\
      G(x) - \frac{m - 1}{2}, & G(x) \in [(m-1)/2, \, 1 + (m-1)/2],\\
      1, & G(x) > 1 + (m-1)/2,
    \end{cases}
  \end{equation}
  which corresponds to the density of the steady state being a cut-off of \( \omega \) such that the medians coincide and again extending from there to ensure that the steady state has mass \( 1 \), which is proves \eqref{eq:365}.
\end{proof}

The monotonicity of \( V' \) now also allows us to show asymptotic stability of these steady states:

\begin{theorem}[Asymptotic stability]
  \label{thm:conv-given-datum}
  Let \( \mu_0, \omega \in L^\infty_c(\mathbb{R}) \) such that \( \mu_0, \omega \geq 0 \) and
  \begin{equation}
    \label{eq:367}
    \int_{\mathbb{R}}  \diff \mu_0(x) = 1, \quad \int_{\mathbb{R}}  \diff \omega(x) > 0.
  \end{equation}
  Denote by $X_{0}$, \( Y \) the pseudo-inverses of \( \mu_0 , \,\omega \), respectively. By Theorem \ref{thm:linfty-ext}, there is a solution
  \begin{equation}
    \label{eq:237}
    X \in C^{1}\left([0,\infty), L^{\infty}([0,1])\right)
  \end{equation}
  of \eqref{eq:229} with \(X(0,.) = X_{0}(.)\) and an associated curve of probability measures \(\mu(t)\) fulfilling the distributional formulation \eqref{eq:232}.

  For \( z \in [0,1] \), let 
  \begin{equation}
    \label{eq:338}
    \widetilde{Y}(z) := 
    \begin{cases}
      (V')^{-1}(2z - 1), &q_a > 1,\\
      Y\left( z + \frac{m-1}{2} \right), &q_a = 1 \text{ and } m \geq 1,\\
    \end{cases}
  \end{equation}
  be the pseudo-inverses of the measures denoted by \( \widetilde{\omega} \) in  \eqref{eq:364} and \eqref{eq:365}, respectively.
  
  The solution \( X(t,.) \) then fulfills
  \begin{equation}
    \label{eq:238}
    X(t,.) \xrightarrow{L^{2}} \widetilde{Y}(.) \,\ \text{ for } t\rightarrow\infty,
  \end{equation}
  which for the associated measures means
  \begin{equation}
    \label{eq:239}
    W_{2}(\mu(t), \widetilde{\omega} ) \rightarrow 0 \,\ \text{ for } t\rightarrow\infty.
  \end{equation}
\end{theorem}

\begin{proof}
  \thmenumhspace{-1em}
  \begin{enumpara}
  \item \emph{(Simplified equation)} Let us begin with the observation that because Theorem \ref{thm:linfty-ext} guarantees that \(X(t,.)\) stays strictly increasing for all times, \(X(t,.)\) is the right-inverse of \(F(t,.)\) for all times \(t \geq 0\), hence \( F(X) = \operatorname{id} \) and by \eqref{eq:338}, equation \eqref{eq:227} can be written as
    \begin{equation}
      \label{eq:339}
      \partial_t X(t,z) = 2z - 1 - V'(X(t,z)) = V'(\widetilde{Y}(z)) - V'(X(t,z)).
    \end{equation}

  \item \emph{(\(L^2\)-norm and pointwise distance decrease)} \label{item:wasserstein-decr} We can compute the derivative of the \(L^{2}\)-norm between \(X(t,.)\) and \(\widetilde{Y}(.)\):
    \begin{align}
      \label{eq:242}
      \leadeq{\frac{\mathrm{d}}{\mathrm{d} t} \int_0^1 \left| X(z)-\widetilde{Y}(z)\right|^2 \, \mathrm{d} z} \\
      = {} &2 \int_0^1 (X(t,z) - \widetilde{Y}(z)) \cdot \left[V'(\widetilde{Y}(z)) - V'(X(t,z))\right] \, \mathrm{d} z \leq 0
    \end{align}
    by the monotonicity of $V'$ and we conclude that
    \begin{equation}
      \label{eq:243}
      \lVert X(t,.) - \widetilde{Y}(.) \rVert_{L^{2}} \leq \lVert X_{0}(.) - \widetilde{Y}(.)
      \rVert_{L^{2}} \,\ \text{ for all } t \geq 0.
    \end{equation}
    
    The same argument can be used to see that the quadratic distance \( | X(t,z) - \widetilde{Y}(z) |^2 \) also decreases monotonically for every $z$, by computing
    \begin{equation}
      \label{eq:244}
      \frac{\mathrm{d}}{\mathrm{d} t} \left| X(t,z) - \widetilde{Y}(z)\right|^2 = 2 \left(X(t,z) - \widetilde{Y}(z)\right) \cdot \left[(V' \circ \widetilde{Y})(z) - (V' \circ X)(t,z)\right] \leq 0.
    \end{equation}
  \item \emph{(Vanishing dissipation)} \label{item:van-diss} Denote the right-hand side of \eqref{eq:242} by $-2 I[X(t,.)]$, i.e.,
    \begin{equation}
      \label{eq:245}
      I[X(t,.)] := \int_0^1 (X(t,z) - \widetilde{Y}(z)) \left[(V' \circ X)(t,z) - (V' \circ \widetilde{Y})(z)\right] \, \mathrm{d} z.
    \end{equation}
    By integrating \eqref{eq:242} in time, we can conclude from
    \begin{equation}
      \label{eq:246}
      \int_0^1\left(X(t,z) - \widetilde{Y}(z)\right)^2 \, \mathrm{d} z + 2 \int_0^t I[X(\tau,.)]\, \mathrm{d} \tau = \int_0^1 \left(X(0, z) - \widetilde{Y}(z)\right)^2 \, \mathrm{d} z
    \end{equation}
    that
    \begin{equation}
      \label{eq:247}
      \int_0^\infty I[X(\tau,.)]\, \mathrm{d} \tau < \infty.
    \end{equation}
    So there is a sequence $(t_k)_k \subseteq [0,\infty)$ with
    $t_k\rightarrow\infty$ and \( I[X(t_k,.)] \rightarrow 0. \)

    Furthermore, as the integrand in the definition of $I$ is non-negative, this convergence can be interpreted as $L^1$-convergence of the integrand. Therefore we can extract a subsequence along which the integrand converges almost everywhere in $[0,1]$, i.e.
    \begin{equation}
      \label{eq:249}
      (X(t,z) - \widetilde{Y}(z)) \left[(V' \circ X)(t,z) - (V' \circ \widetilde{Y})(z)\right] \rightarrow 0 \,\ 
      \text{ for a.e.\@ } z
    \end{equation}
  \item \emph{(Convergence of the pseudo-inverse a.e.)} Since by Step \ref{item:wasserstein-decr}, \( \left| X(t_{k},z) - Y(z) \right| \) is monotonically decreasing for all \(z \in [0,1]\) and it is obviously bounded from below, it is a convergent sequence. Towards a contradiction, assume that for some \(z\), it was not converging to \(0\). Then for those \(z\) (except possibly for a null set), we would have at least
    \begin{equation}
      \label{eq:250}
      V'(X(t_k,z)) \rightarrow V'(\widetilde{Y}(z)) \ \text{ for } k\rightarrow\infty,
    \end{equation}
    as otherwise the convergence in \eqref{eq:249} would not hold.

    In the case \( q_a > 1 \), we can simply continuously invert \( V' \), yielding a contradiction to the assumption that \(\left| X(t_{k},z) - Y(z)\right|\) is bounded away from zero.

    In the case \( q_a = 1 \), the continuous invertibility of \( V' \) only fails at its jump points, which on the right-hand side of \eqref{eq:250} occur at at most countably many points, meaning almost nowhere. For all other \( z \), there is a small neighborhood around \( \widetilde{Y}(z) \) such that \( V' \) is continuous there, resulting in convergence almost everywhere as well.

    Summarising, we conclude that
    \begin{equation}
      \label{eq:251}
      X(t_k,z) \rightarrow \widetilde{Y}(z) \ \text{ for } k\rightarrow\infty \text{ and a.e.\@ } z.
    \end{equation}
  \item \emph{(Convergence in \(L^2\))} By the assumptions, $(X(0,z) - \widetilde{Y}(z))^2$ is integrable and the sequence $(X(t_k,z) - \widetilde{Y}(z))^2$ is monotonically decreasing, therefore the former is a dominating function for the latter, and by the Dominated Convergence Theorem we get convergence in $L^{2}([0,1])$ along the subsequence $t_k$. Due to the monotonicity proven in Step \ref{item:wasserstein-decr}, this carries over to any sequence \(t_{k}\rightarrow\infty\). The convergence \eqref{eq:239} in \(W_{2}\) of the associated measures then follows from Lemma \ref{lem:17}. \qedhere
  \end{enumpara}
\end{proof}

\begin{remark}[Special case \( m = 1 \), no \( W_\infty \)-convergence]
  If \( q_a = m = 1 \), then \( \widetilde{Y} = Y \) and we have proved \( \mu(t) \rightarrow \omega \) in \( \mathcal{P}_2(\mathbb{R}) \) above.

  Note that in this case, we cannot in general expect \( \infty \)-Wasserstein-convergence: Let \( \inf \supp \mu_0 < \inf \supp \omega \), then
  \begin{equation}
    \label{eq:248}
    \partial_t X(0,0) = 2 \cdot 0 - 1 - V'(X(0,0)) = 2G(X(0,0)) = 0,
  \end{equation}
  so the left edge of the support will stay stationary throughout the evolution (while the mass to the right of it will be pulled towards \( \omega \)).
\end{remark}

\begin{remark}[Traveling mass for \( m < 1 \)]
  \label{rem:van-mass}
  Let \( q_a = 1 \) and \( m < 1 \). For the left tail, \( z \in [0, (1-m)/2)) \), the pseudo-inverse equation is
  \begin{equation}
    \label{eq:240}
     \partial_t X(t,z) = 2z - 1 - V'(X(t,z)) = 2z - 1 + m \leq 0
  \end{equation}
  which means that the left tail of \( \mu \) of mass \( (1-m)/2 \) will travel towards \( -\infty \) and similarly, the right tail of the same mass will travel to \( \infty \).

  As for \( z \in \left[ (1-m)/2, 1 - (1-m)/2 \right] \), we are able to define the pseudo-inverse \( \widetilde{Y} \) there as in \eqref{eq:338} and can thence apply the same arguments as in the proof of Theorem \ref{thm:conv-given-datum}, restricting all integrals, norms and pointwise evaluations to the interval \( [ (1-m)/2,\, 1 - (1-m)/2 ] \), to see that the restricted profile of \( \mu \) there converges to \( \omega \).

  In total, this means that \( \mu(t) \) converges vaguely to \( \widetilde{\omega} \) (\ie, in the duality with functions \( C_0(\mathbb{R}) \) which vanish at \( \pm \infty \)): Let \( f \in C_0(\mathbb{R}) \), then
  \begin{align}
    \label{eq:369}
    \int_{\mathbb{R}} f(x) \diff \mu(t,x) = {} & \int_{0}^{1} f(X(t,z)) \diff z\\
    = {} & \int_{(1-m )/2}^{1-(1-m)/2} f(X(t,z)) \diff z \nonumber \\
    {} & + \int_{0}^{(1-m)/2} f(X(t,z)) \diff z + \int_{1-(1-m)/2}^1 f(X(t,z)) \diff z\\
    \rightarrow {} & \int_{(1-m )/2}^{1-(1-m)/2} f(\widetilde{Y}(z)) \diff z = \int_{\mathbb{R}} f(x) \diff \widetilde{\omega}(x), \quad t\rightarrow\infty,
  \end{align}
  by the substitution formula (Lemma \ref{lem:16}) and the Dominated Convergence Theorem.
\end{remark}

\subsection{Convergence of a subsequence to a steady state}
\label{sec:conv-against-steady}

In Sections \ref{sec:case-q-2} and \ref{sec:case-q-1}, we clarified the asymptotic behavior of solutions in the cases \( q_a = q_r = 2 \) and \( q_r = 1 \leq q_a \in [1,2]  \). In this section we address the problem of establishing convergence to steady states in the range \( 1 < q_r \leq q_a < 2 \). For such a purpose, we shall employ both an energy-energy-dissipation inequality combined with the moment bounds \eqref{eq:358} (Theorem \ref{thm:exist-min-strong}) and \eqref{eq:206} (Theorem \ref{thm:moment-bound-symmetric}) in order to derive both compactness of trajectories and continuity of the dissipation.

It will turn out that such a technique is actually sharply bounded to succeed only for the range of parameters
\begin{align}
  \label{eq:362}
    q_r &\in [1,2), \quad q_r < q_a \leq 2, \quad \text{or}\\
    q_r &= q_a \in \left[ 1,\,4/3 \right),
\end{align}
leaving still open the harder problem to describe the asymptotics for \( q_r > q_a \) or \( 4/3 \leq q_r = q_a < 2 \). Notice that the case \eqref{eq:362} actually applies to all parameters in the range \( [1,2] \) as soon as the attraction is stronger than the repulsion, thanks to the additional compactness given by the confinement property of the attraction (see Section \ref{sec:exist-minim-strong}, in particular the proof of Theorem \ref{thm:exist-min-strong}).

\begin{lemma}[Dissipation formula]
  \label{lem:28}
  Let \( q_a, q_r \in (1,2] \) and \( X(.,.) \) be a solution curve to \eqref{eq:227} as in Theorem \ref{thm:linfty-ext}. Denoting the associated measures by \( \mu(.,.) \), we have
  \begin{align}
    \label{eq:340}
    \frac{\diff}{\diff t} \mathcal{E}[\mu(t,.)] = {} & - \int_{\mathbb{R}} \left| \psi_a' \ast \omega - \psi_r' \ast \mu(t) \right|^2 \diff \mu \\
    = {} & -\int_{0}^{1} \left| \int_{0}^{1} \left[ \psi_a(X(t,z) - Y(\zeta)) - \psi_r(X(t,z) - X(t,\zeta)) \right] \diff \zeta \right|^2 \diff z \\
    = : {} & - \mathcal{D}[\mu(t)]. \label{eq:368}
  \end{align}
\end{lemma}

\begin{proof}
  By Lemma \ref{lem:16}, we can write
  \begin{align}
    \label{eq:341}
    \mathcal{E}[\mu(t)] = {} & \int_{\mathbb{R}\times\mathbb{R}} \psi_a(x-y) \diff \omega(x) \diff \mu(y) - \frac{1}{2} \int_{\mathbb{R}\times\mathbb{R}} \psi_r(x-y)  \diff \mu(x) \diff \mu(y) \\
    = {} & \int_{0}^{1} \int_{0}^{1} \psi_a(X(t,z) - Y(\zeta)) \diff \zeta \diff z \nonumber \\
    &- \frac{1}{2} \int_{0}^{1} \int_{0}^{1} \psi_r(X(t,z) - X(t, \zeta)) \diff \zeta \diff z.
  \end{align}
  Since \( \psi_r' \) and \( \psi_a' \) are continuous functions and \( X(.,.) \in C^1([0,T], L^\infty[0,1]) \), for \( t \in [0,T] \), the appearing derivatives will be bounded uniformly in \( t \), so differentiating under the integral sign is justified by the Dominated Convergence Theorem, yielding
  \begin{align}
    \label{eq:342}
    \frac{\diff}{\diff t} \mathcal{E}[\mu(t)] = {} & \int_{0}^{1} \int_{0}^{1} \psi_a'(X(t,z) - Y(\zeta)) \partial_t X(t,z) \diff \zeta \diff z \nonumber \\
    & \quad - \frac{1}{2} \int_{0}^{1} \int_{0}^{1} \psi_r'(X(t,z) - X(t, \zeta)) \left( \partial_t X(t,z) - \partial_t X(t,\zeta) \right) \diff \zeta \diff z \\
     {} = & \int_{0}^{1} \int_{0}^{1} \int_{0}^{1} \left[ \psi_a'(X(t,z) - Y(\zeta)) - \psi_r'(X(t,z) - X(t, \zeta)) \right] \nonumber\\
     & \quad \cdot \left[ - \psi_a'(X(t,z) - Y(\xi)) + \psi_r'(X(t,z) - X(t,\xi)) \right] \diff \xi \diff \zeta \diff z \label{eq:343}\\
     = {} & -\int_{0}^{1} \left| \int_{0}^{1} \left[ \psi_a(X(t,z) - Y(\zeta)) - \psi_r(X(t,z) - X(t,\zeta)) \right] \diff \zeta \right|^2 \diff z,\\
  \end{align}
  where in equation \eqref{eq:343}, we inserted \eqref{eq:227} and used the anti-symmetry of \( \psi_r' \).
\end{proof}

\begin{theorem}[Convergence up to a subsequence]
  \label{thm:conv-steady-state}
  Let either
  \begin{equation}
    \label{eq:346}
    1 < q_r < 2, \quad q_r < q_a < 2
  \end{equation}
  or 
  \begin{equation}
    \label{eq:347}
    1 < q_a = q_r < \frac{4}{3}
  \end{equation}
  and \( \omega, \mu \) as in Theorem \ref{thm:linfty-ext}. Then, there is a probability measure \( \mu^\ast \) and a sequence \( (t_k)_{k \in \mathbb{N}} \) with \( 0 \leq t_k \rightarrow \infty \) for \( k \rightarrow\infty \) such that
  \begin{equation}
    \label{eq:344}
    \mu(t_k) \rightarrow \mu^\ast \text{ narrowly}
  \end{equation}
  and
  \begin{equation}
    \label{eq:345}
    \psi_a' \ast \omega(x) - \psi_r' \ast \mu^\ast(x) = 0, \quad x \in \supp(\mu^\ast),
  \end{equation}
  \ie, \( \mu^\ast \) fulfills the steady state equation of the gradient flow.
\end{theorem}

\begin{proof}
  We shall use a compactness argument to find the subsequence and then have to justify the continuity of \( \mathcal{D} \) as defined in \eqref{eq:368} to gain \( \mathcal{D}[\mu^\ast] = 0 \) and hence equation \eqref{eq:345}:

  By Lemma \ref{lem:28}, we know that
  \begin{equation}
    \label{eq:348}
    \mathcal{E}[\mu(t)] = \mathcal{E}[\mu(0)] - \int_{0}^{t} \mathcal{D}[\mu(\tau)] \diff \tau,
  \end{equation}
  where \( D[\mu(\tau)] \geq 0 \) for all \( \tau \), so the energy is decreasing and therefore bounded from above. Moreover, by its Fourier representation \eqref{eq:31} for \( q_a = q_r \) and Theorem \ref{thm:exist-min-strong} for \( q_r < q_a \), respectively, we know that \( \mathcal{E} \) is also bounded from below, yielding convergence of the integral \( \int_{0}^{t} \mathcal{D}[\mu(\tau)] \diff \tau \) for \( t\rightarrow\infty \) whence we can select a subsequence \( (t_k)_k \) for which
  \begin{equation}
    \label{eq:349}
    \mathcal{D}[\mu(t_k)] \rightarrow 0, \quad k \rightarrow\infty.
  \end{equation}

  In the case \eqref{eq:346}, by the proof of Theorem \ref{thm:exist-min-strong} we have that the sub-levels of \( \mathcal{E} \) have a uniformly bounded \( q_a \)th moment, while in the case \eqref{eq:347}, by Theorem \ref{thm:moment-bound-symmetric}, the \( r \)th moment is uniformly bounded for all \( r < q_r/2 \), yielding tightness of \( (\mu(t_k))_k \) and uniform integrability of all smaller moments by Lemma \ref{lem:24}. Hence there exist a \( \mu^\ast \in \mathcal{P}(\mathbb{R}) \) and a further subsequence, again denoted by \( (t_k)_k \), for which
  \begin{equation}
    \label{eq:350}
    \mu(t_k) \rightarrow \mu^\ast \quad \text{narrowly}.
  \end{equation}
  Now, we want to deduce \( \mathcal{D}[\mu^\ast] = 0 \) by the continuity of \( \mathcal{D} \) along the sequence \( (\mu(t_k))_k \). For this, expand \( D \) into a sum of triple integrals \wrt probability measures, , \ie
\begin{align}
  \label{eq:352}
  \mathcal{D}[\mu] = {} & - \int_{\mathbb{R}^3} \psi_a'(x-y) \psi_a'(x-z) \diff \omega (y) \diff \omega(z) \diff \mu(x) \nonumber\\
  & + \int_{\mathbb{R}^3} \psi_a'(x-y) \psi_r'(x-z) \diff \omega (y) \diff \mu(z) \diff\mu(x) \nonumber\\
  & + \int_{\mathbb{R}^3} \psi_r'(x-y) \psi_a'(x-z) \diff \mu (y) \diff \omega(z) \diff\mu(x) \nonumber\\
  & - \int_{\mathbb{R}^3} \psi_r'(x-y) \psi_r'(x-z) \diff \mu (y) \diff \mu(z) \diff\mu(x).
\end{align}
To deduce the continuity of \( \mathcal{D} \), we want to use a tensorization argument as in Theorem \ref{thm:exist-min-strong}. Hence we consider narrow convergence of sequences of the type \( \mu(t_k) \otimes \mu(t_k) \otimes \omega \) or \( \mu(t_k) \otimes \omega \otimes \omega \) (which is true by \eqref{eq:350} and Lemma \ref{lem:27}) and the uniform integrability of the integrands of the integrals in \eqref{eq:352}. For the latter it will be enough to show that they are bounded by uniformly integrable functions: Using \( \left| x \right|^r \leq \left| x \right|^q + 1 \) for all \( x \in \mathbb{R} \) and \( 0 \leq r \leq q \), each of the occurring integrands can be estimated (up to a constant) by
  \begin{align}
    \label{eq:351}
    \leadeq{\left| x - y \right|^{q_a-1} \left| x - z \right|^{q_a-1} + \left| x-y \right|^{q_a-1} + \left| x - z \right|^{q_a-1} + 1} \\
    \leq {} & C \left[ \left| x - y \right|^{2q_a - 2} + \left| x-z \right|^{2q_a - 2} + 1\right] \\
    \leq {} & C \left[ \left| x \right|^{2q_a - 2} + \left| y \right|^{2q_a - 2} + \left| z \right|^{2q_a - 2} \right].
  \end{align}
  The uniform integrability of the moments is now enough to deduce continuity of \( \mathcal{D} \) by the Lemmata \ref{lem:24} and \ref{lem:4}, yielding the claim: In the asymmetric case \eqref{eq:347} of a stronger attraction, \( 2q_a - 2 < q_a \) by \( q_a < 2 \) and in the symmetric case \eqref{eq:344}, we have \( 2q_a - 2 < q_a/2 \) by \( q_a < 4/3 \).
\end{proof}

\begin{remark}[Sharpness of \eqref{eq:351}]
  \label{rem:sharpness}
  For our purposes, \ie, using only bounds on the moments, we cannot do better than \eqref{eq:351} and hence not better as condition \eqref{eq:346} as well: Formally assume \( \omega = \delta_0 \) (this particular choice is excluded by the \( L^\infty_c \)-assumption, but can easily be approximated with convergence in \( \mathcal{D} \)). Then, one of the occurring summands in \eqref{eq:351} is
  \begin{equation}
    \label{eq:370}
    \int_{\mathbb{R} \times \mathbb{R}} \left| x-y \right|^{q_a - 1} \left| x-z \right|^{q_a - 1} \diff \omega(y) \diff \omega(z) = \left| x \right|^{2q_a - 2},
  \end{equation}
  so this moment has to be uniformly integrable for our argumentation to work.
\end{remark}

\section{Conclusion}
\label{sec:conclusion2}

For the first two main results, the analysis of the asymptotic behavior in Section \ref{sec:case-q-1} and Section \ref{sec:conv-against-steady}, we remark that the pseudo-inverse technique proved very helpful in understanding the equation. It also revealed the special structure for \( q_r = 1 \) and \( q_r = 2 \) which we exploited, namely that in terms of the pseudo-inverse, there is a certain locality in this case which is lost for \( 1 < q_r < 2 \), where the equation becomes highly non-local. On the other hand, we remark that for the arguments in Section \ref{sec:conv-against-steady}, as in the first part of this work, Section \ref{cha:vari-prop-symm}, the Fourier representation was indispensable to understand the compactness properties for a wider range of parameter combinations.

Finally, there is a large range of parameters for which the asymptotic behavior remains still open, in particular the specific characterization of the steady states, which is likely to necessitate additional or completely different techniques compared to the ones used here.

\appendix
\clearpage{}\makeatletter{}\chapter{Conditionally positive definite functions}
\label{cha:cond-posit-semi}

In order to compute the Fourier representation of the energy functional \( \mathcal{E} \) in Section \ref{sec:four-repr-gener}, we used the notion of \emph{generalized Fourier transforms} and \emph{conditionally positive definite functions} from \cite{Wend05}, which we would like to briefly introduce here.

Our representation formula \eqref{eq:39} is a consequence of Theorem \ref{thm:repr-thm-cond-semi} below, which serves as a characterization theorem in the theory of conditionally positive definite functions:

\begin{definition}
  \label{def:cond-definit}
  \cite[Definition 8.1]{Wend05}
  Let \( \mathbb{P}_{k}(\mathbb{R}^d) \) denote the set of polynomial functions on \( \mathbb{R}^d \) with degree less or equal than \( k \). We call a continuous function \( \Phi \colon \mathbb{R}^d \rightarrow \mathbb{C} \) \emph{conditionally positive semi-definite of order} \( m \) if for all \( N\in\mathbb{N} \), pairwise distinct points \( x_1,\ldots,x_N \in \mathbb{R}^d \) and \( \alpha \in \mathbb{C}^N \) with
  \begin{equation}
    \label{eq:281}
    \sum_{j=1}^{N} \alpha_j p(x_j) = 0, \quad \text{for all } p \in \mathbb{P}_{m - 1}(\mathbb{R}^d),
  \end{equation}
  the quadratic form given by \( \Phi \) is non-negative, i.e.
  \begin{equation}
    \label{eq:282}
    \sum_{j,k=1}^{N} \alpha_j \overline{\alpha_k} \Phi(x_j - x_k) \geq 0.
  \end{equation}
  Moreover, we call \( \Phi \) \emph{conditionally positive definite of order} \( m \) if the above inequality is strict for \( \alpha \neq 0 \).
\end{definition}

\section{Generalized Fourier transform}
\label{sec:gener-four-transf}

When working with distributional Fourier transforms, which can serve to characterize the conditionally positive definite functions defined above, it can be opportune to further reduce the standard Schwartz space \( \mathcal{S} \) to functions which in addition to the polynomial decay for large arguments also exhibit a certain decay for small ones. This way, one can elegantly neglect singularities in the Fourier transform which could otherwise arise.

\begin{definition}
  [Restricted Schwartz class \( \mathcal{S}_m \)]
  \label{def:restr-schwartz}
  \cite[Definition 8.8]{Wend05}
  Let \( \mathcal{S} \) be the space of functions in \( C^\infty(\mathbb{R}^d) \) which for \( \left| x \right| \rightarrow \infty \) decay faster than any fixed polynomial. Then, for \( m \in \mathbb{N} \), we denote by \( \mathcal{S}_m \) the set of those functions in \( S \) which additionally fulfill
  \begin{equation}
    \label{eq:283}
    \gamma(\xi) = O(\left| \xi \right|^m) \quad \text{for } \xi \rightarrow 0.
  \end{equation}

  Furthermore, we shall call an (otherwise arbitrary) function \( \Phi \colon \mathbb{R}^d \rightarrow \mathbb{C} \) \emph{slowly increasing} if there is an \( m \in \mathbb{N} \) such that
  \begin{equation}
    \label{eq:284}
    \Phi(x) = O \left( \left| x \right|^m \right) \quad \text{for } \left| x \right| \rightarrow \infty.
  \end{equation}
\end{definition}

\begin{definition}
  [Generalized Fourier transform]
  \label{def:gen-fourier-transform}
  \cite[Definition 8.9]{Wend05}
  For \( \Phi \colon \mathbb{R}^d \rightarrow \mathbb{C} \) continuous and slowly increasing, we call a measurable function \(\widehat{\Phi} \in L_{\mathrm{loc}}^2(\mathbb{R}^d \setminus \left\{ 0 \right\})\) the \emph{generalized Fourier transform} of \( \Phi \) if there exists an integer \( m \in \mathbb{N}_0 \) such that
  \begin{equation}
    \label{eq:285}
    \int_{\mathbb{R}^d} \Phi(x)\widehat{\gamma}(x) \diff x = \int_{\mathbb{R}^d} \widehat{\Phi}(\xi) \gamma(\xi) \diff \xi \quad \text{for all } \gamma \in \mathcal{S}_{2m}.
  \end{equation}
  Then, we call \( m \) the \emph{order} of \( \widehat{\Phi} \).
\end{definition}

Note that the order here is defined in terms of \( 2m \) instead of \( m \).

The consequence of this definition is that we ignore additive polynomial factors in \( \Phi \) which would translate to Dirac distributions in the Fourier transform:

\begin{proposition}
  \label{prp:poly-vanish}
  \cite[Proposition 8.10]{Wend05} If \( \Phi \in \mathbb{P}_{m-1}(\mathbb{R}^d) \), then \( \Phi \) has the generalized Fourier transform \( 0 \) of order \( m/2 \). Conversely, if \( \Phi \) is a continuous function which has generalized Fourier transform \( 0 \) of order \( m/2 \), then \( \Phi \in \mathbb{P}_{m-1} \left( \mathbb{R}^d \right) \).
\end{proposition}

\begin{proof}[Sketch of proof]
  The first claim follows from the fact that multiplication with polynomials corresponds to computing derivatives of the Fourier transform: By condition \eqref{eq:283}, all derivatives of order less than \( m \) of a test function \( \gamma \in \mathcal{S}_m \) have to vanish.

  The second claim follows from considering the coupling \( \int_{\mathbb{R}^d} \Phi(x) \widehat{g}(x) \diff x \) for a general \( g \in \mathcal{S} \) and projecting it into \( \mathcal{S}_m \) by setting
  \begin{equation}
    \label{eq:286}
    \gamma(x) := g(x) - \sum_{\left| \beta \right| < m} \frac{D^{\beta}g(0)}{\beta!}x^\beta \chi(x), \quad x \in \mathbb{R}^d,
  \end{equation}
  with a \( \chi \in C_0^\infty(\mathbb{R}^d) \) which is \( 1 \) close to \( 0 \).
\end{proof}

\section{Representation formula for conditionally positive definite functions}
\label{sec:repr-form-cond}

Before proceeding to prove Theorem \ref{thm:repr-thm-cond-semi}, we need two lemmata. The first one is the key to applying the generalized Fourier transform in our case, namely that functions fulfilling the decay condition \eqref{eq:283} can be constructed as Fourier transforms of point measures satisfying condition \eqref{eq:281}. The second one recalls some basic facts about the Fourier transform of the Gaussian, serving to pull the exponential functions in Lemma \ref{lem:19} into \( \mathcal{S}_m \).

\begin{lemma}
  \cite[Lemma 8.11]{Wend05}
  \label{lem:19}
  Given pairwise distinct points \( x_1,\ldots,x_N \in \mathbb{R}^d \) and \( \alpha \in \mathbb{C}^N \setminus \left\{ 0 \right\} \) such that
  \begin{equation}
    \label{eq:287}
    \sum_{j=1}^{N} \alpha_j p(x_j) = 0, \quad \text{for all } p \in \mathbb{P}_{m-1}(\mathbb{R}^d),
  \end{equation}
  then
  \begin{equation}
    \label{eq:288}
    \sum_{j=1}^{N} \alpha_j \e^{\i x_j \cdot \xi} = O \left( \left| \xi \right|^m \right) \quad \text{for } \left| \xi \right| \rightarrow 0.
  \end{equation}
\end{lemma}

\begin{proof}
  Expanding the exponential function into its power series yields
  \begin{equation}
    \label{eq:289}
    \sum_{j=1}^{N} \alpha_j \e^{\i x_j \cdot \xi} = \sum_{k = 0}^{\infty} \frac{\i^k}{k!} \sum_{j=1}^{N} \alpha_j \left( x_j \cdot \xi \right)^k,
  \end{equation}
  and by condition \eqref{eq:287} its first \( m \) terms vanish, giving us the desired behavior.
\end{proof}

\begin{lemma}
  \label{lem:20}
  \cite[Theorem 5.20]{Wend05}
  Let \( l > 0 \) and \( g_l (x) := (l/\pi)^{d/2} \e^{-l \left| x \right|^2} \). Then,
  \begin{enumparaalph}
  \item \label{lem:21} \( \widehat{g}_l(\xi) = \e^{-\left| \xi \right|^2/(4l)} \);
  \item \label{lem:22} for \( \Phi \colon \mathbb{R}^d \rightarrow \mathbb{C}\) continuous and slowly increasing, we have
    \begin{equation}
      \label{eq:290}
      \Phi(x) = \lim_{l \rightarrow \infty} (\Phi \ast g_l)(x).
    \end{equation}
  \end{enumparaalph}
\end{lemma}

\begin{theorem}
  \label{thm:repr-thm-cond-semi}
  \cite[Corollary 8.13]{Wend05}
  Let \( \Phi \colon \mathbb{R}^d \rightarrow \mathbb{C} \) be a continuous and slowly increasing function with a non-negative, non-vanishing generalized Fourier transform \( \widehat{\Phi} \) of order \( m \) that is continuous on \( \mathbb{R}^d \setminus \left\{ 0 \right\} \). Then, we have
  \begin{equation}
    \label{eq:291}
    \sum_{j,k=1}^N \alpha_j \overline{\alpha}_k \Phi \left( x_j-x_k \right) = \int_{\mathbb{R}^d} \left| \sum_{j=1}^{N} \alpha_j \e^{\i x_j \cdot \xi}\right|^2 \widehat{\Phi}(\xi) \diff \xi.
  \end{equation}
\end{theorem}

\begin{proof}
  Let us start with the right-hand side of the claimed identity \eqref{eq:291}: By Lemma \ref{lem:19}, the function
  \begin{equation}
    \label{eq:292}
    f(\xi) := \left| \sum_{j=1}^{N} \alpha_j \e^{\i x_j \cdot \xi} \right|^2 \widehat{g}_l(\xi)
  \end{equation}
  is in \( \mathcal{S}_{2m} \) for all \( l > 0 \). Moreover, by the Monotone Convergence Theorem,
  \begin{align}
    \label{eq:293}
    \int_{\mathbb{R}^d} \left| \sum_{j=1}^{N} \alpha_j \e^{\i x_j \cdot \xi}\right|^2 \widehat{\Phi}(\xi) \diff \xi = {} &\lim_{l \rightarrow \infty} \int_{\mathbb{R}^d} \left| \sum_{j=1}^{N} \alpha_j \e^{\i x_j \cdot \xi}\right| \widehat{g}_l(\xi)\, \widehat{\Phi}(\xi) \diff \xi\\
    = {} &\lim_{l \rightarrow\infty} \int_{\mathbb{R}^d} \left( \left| \sum_{j=1}^{N} \alpha_j \e^{\i x_j \cdot .}\right|^2 \widehat{g}_l(.) \right)^{\wedge}(x)\, \Phi(x) \diff x.
  \end{align}
  Now, by Lemma \ref{lem:20}\ref{lem:21},
  \begin{align}
    \label{eq:294}
    \left( \left| \sum_{j=1}^{N} \alpha_j \e^{\i x_j \cdot .}\right|^2 \widehat{g}_l(.) \right)^{\wedge}(x) = {} &\widehat{\widehat{g\,}}_l \ast \left( \sum_{j=1}^{N} \alpha_j \delta_{x_j} \right) \ast \left( \sum_{j=1}^{N} \overline{\alpha_j} \delta_{-x_j} \right)(x)\\
    = {} & g_l\ast \left( \sum_{j=1}^{N} \alpha_j \delta_{x_j} \right) \ast \left( \sum_{j=1}^{N} \overline{\alpha_j} \delta_{-x_j} \right) (x)
  \end{align}
  and therefore
  \begin{align}
    \label{eq:295}
    \leadeq{\lim_{l \rightarrow\infty} \int_{\mathbb{R}^d} \left( \left| \sum_{j=1}^{N} \alpha_j \e^{\i x_j \cdot .}\right|^2 \widehat{g}_l(.) \right)^{\wedge}(x)\, \Phi(x) \diff x}\\
    = {} & \lim_{l \rightarrow\infty} \int_{\mathbb{R}^d}  \Phi(x) \, g_l \ast \left( \sum_{j=1}^{N} \alpha_j \delta_{x_j} \right) \ast \left( \sum_{j=1}^{N} \overline{\alpha_j} \delta_{-x_j} \right)(x) \diff x\\
    = {} & \lim_{l \rightarrow\infty} \sum_{i,j=1}^{N} \int_{\mathbb{R}^d} \alpha_i \overline{\alpha_j}\,  \Phi(x) \, g_l(x-(x_i-x_j)) \diff x\\
    = {} & \lim_{l \rightarrow\infty} \sum_{i,j=1}^{N} \int_{\mathbb{R}^d} \alpha_i \overline{\alpha_j}\,  \Phi(x-(x_i-x_j)) \, g_l(x) \diff x\\
    = {} & \sum_{i,j=1}^{N} \alpha_i \overline{\alpha_j} \, \Phi(x_i-x_j)
  \end{align}
   by Lemma \ref{lem:20}\ref{lem:22}.
\end{proof}

\section{Computation for the power function}
\label{sec:comp-power-funct}

Given Theorem \ref{thm:repr-thm-cond-semi}, we are naturally interested in the explicit formula for the power function \( x \mapsto \left| x \right|^q \) for \( q \in [1,2) \). It is a nice example of how to pass from an ordinary Fourier transform to the generalized Fourier transform by extending a formula by means of complex analysis. Our starting point will be the multiquadric \( x \mapsto \left( c^2 + \left| x \right|^2 \right)^{\beta} \) for \( \beta < -d/2 \), whose Fourier transform involves the modified Bessel function of the third kind:

\begin{definition}
  [Modified Bessel function]
  \cite[Definition 5.10]{Wend05}
  For \( \nu \in \mathbb{C} \), \( z \in \mathbb{C} \) with \( \left| \operatorname{arg} z \right| < \pi/2 \), set
  \begin{equation}
    \label{eq:298}
    K_\nu(z) := \int_{0}^\infty \exp(-z \cosh(t)) \cosh(\nu t)  \diff t,
  \end{equation}
  the \emph{modified Bessel function of the third kind of order} \( \nu \in \mathbb{C} \).
\end{definition}

\begin{theorem}
  \label{poly:cond-ft-multquad1}
  \cite[Theorem 6.13]{Wend05}
    For \( c > 0 \) and \( \beta < -d/2 \),
    \begin{equation}
      \label{eq:299}
      \Phi(x) = (c^2+\left| x \right|^2)^{\beta}, \quad x \in \mathbb{R}^d,
    \end{equation}
    has the Fourier transform
    \begin{equation}
      \label{eq:300}
      \widehat{\Phi}(\xi) =(2\pi)^{d/2}  \frac{2^{1+\beta}}{\Gamma(-\beta)} \left( \frac{\left| \xi \right|}{c} \right)^{-\beta-d/2}K_{d/2+\beta}(c \left| \xi \right|).
    \end{equation}
\end{theorem}

The next lemma provides the asymptotic behavior of the involved Bessel function for large and small values, which we need for the following proof.

\begin{lemma}
  [Estimates for \( K_\nu \)]
  \label{lem:23}
  \thmenumhspace{-1em}
  \begin{enumparaalph}
  \item \cite[Lemma 5.14]{Wend05}
    For \( \nu \in \mathbb{C}, r > 0 \),
    \begin{equation}
      \label{eq:301}
      \left| K_\nu(r) \right| \leq
      \begin{cases}
        2^{\left| \Re (\nu) \right| - 1}\Gamma \left( \left| \Re(\nu) \right| \right) r^{-\left| \Re(\nu) \right|}, &\Re(\nu) \neq 0,\\
        \frac{1}{\e}-\log \frac{r}{2},&r < 2, \Re(\nu) = 0.
      \end{cases}
    \end{equation}

  \item For large \( r \), \( K_\nu \) has the asymptotic behavior
    \begin{equation}
      \label{eq:302}
      \left| K_\nu(r) \right| \leq \sqrt{\frac{2\pi}{r}} \e^{-r} \e^{\left| \Re(\mu) \right|^2/(2r)}, \quad r > 0.
    \end{equation}

  \end{enumparaalph}
\end{lemma}

\begin{theorem}
  \label{thm:cond-ft-power}
  \thmenumhspace{-1em}
  \begin{enumparaalph}
  \item \label{itm:cond-ft-multquad2} \cite[Theorem 8.15]{Wend05}
    \( \Phi(x) = (c^2 + \left| x \right|^2)^\beta \), \( x \in \mathbb{R}^d \) for \( c > 0 \) and \( \beta \in \mathbb{R} \setminus \frac{1}{2} \mathbb{N}_0 \) has the generalized Fourier transform
    \begin{equation}
      \label{eq:303}
      \widehat{\Phi}(\xi) = (2\pi)^{d/2} \frac{2^{1+\beta}}{\Gamma(-\beta)} \left( \frac{\left| \xi \right|}{c} \right)^{-\beta-d/2} K_{d/2+\beta}(c \left| \xi \right|), \quad \xi \neq 0
    \end{equation}
    of order \( m = \max(0, \lceil 2\beta \rceil/2) \).
  \item \label{itm:cond-ft-power} \cite[Theorem 8.16]{Wend05}
    \( \Phi(x) = \left| x \right|_2^\beta \), \( x \in \mathbb{R}^d \) with \( \beta \in \mathbb{R}_+ \setminus \mathbb{N} \) has the generalized Fourier transform
    \begin{equation}
      \label{eq:304}
      \widehat{\Phi}(\xi) = (2\pi)^{d/2}\frac{2^{\beta+d/2}\Gamma((d+\beta)/2)}{\Gamma(-\pi/2)} \left| \xi \right|^{-\beta-d}, \quad \xi \neq 0.
    \end{equation}
    of order \( m = \lceil \beta \rceil/2 \).
  \end{enumparaalph}
\end{theorem}

\begin{proof}
  \thmenumhspace{-1em}
  \begin{enumparaalph}
  \item We can pass from formula \eqref{eq:300} to \eqref{eq:303} by analytic continuation, where the increasing \( m \) serves to give us the needed dominating function.

    Let \( G = \left\{ \lambda \in \mathbb{C} : \Re(\lambda) < m \right\} \) and
    \begin{align}
      \label{eq:305}
      \varphi_\lambda(\xi) := {} &(2\pi)^{d/2} \frac{2^{1+\lambda}}{\Gamma(-\lambda)} \left( \frac{\left| \xi \right|}{c} \right)^{-\lambda-d/2} K_{d/2+\lambda}(c \left| \xi \right|)\\
      \Phi_\lambda(\xi) := {} & \left( c^2 + \left| \xi \right|^2 \right)^\lambda.
    \end{align}
    We want to show
    \begin{equation}
      \label{eq:306}
      \int_{\mathbb{R}^d} \Phi_\lambda(\xi)\widehat{\gamma}(\xi) \diff \xi = \int_{\mathbb{R}^d} \varphi_\lambda(\xi) \gamma(\xi) \diff \xi, \quad \text{for all } \gamma \in \mathcal{S}_{2m},
    \end{equation}
    which is true for \( \lambda < d/2 \) by \eqref{eq:300}. As the integrands \( \Phi_\lambda \) and \( \varphi_\lambda \) are analytic, the integral functions are also analytic by Cauchy’s Integral Formula and Fubini’s Theorem if we can find a uniform dominating function for each of them on an arbitrary compact set \( \mathcal{C} \subseteq G \). As this is clear for \( \Phi_\lambda \) by the decay of \( \gamma \in \mathcal{S} \), it remains to consider \( \varphi_\lambda \).

    Setting \( b := \Re(\lambda) \), for \( \xi \) close to \( 0 \) we get by estimate \eqref{eq:301} of Lemma \ref{lem:23} that
    \begin{equation}
      \label{eq:307}
      \left| \varphi_\lambda(\xi) \gamma(\xi) \right| \leq C_\gamma \frac{2^{b+\left| b + d/2 \right|}\Gamma(\left| b + d/2 \right|)}{\left| \Gamma(-\lambda) \right|}c^{b+d/2-\left| b+d/2 \right|}\left| \xi \right|^{-b-d/2-\left| b+d/2 \right|+2m}
    \end{equation}
    for \( b \neq -d/2 \) and
    \begin{equation}
      \label{eq:308}
      \left| \varphi_\lambda(\xi)\gamma(\xi) \right|\leq C_\lambda \frac{2^{1-d/2}}{\left| \Gamma(-\lambda) \right|}\left( \frac{1}{\e} - \log \frac{c \left| \xi \right|}{2} \right).
    \end{equation}
    for \( b = -d/2 \). Taking into account that \( \mathcal{C} \) is compact and \( 1/\Gamma \) is an entire function, this yields
    \begin{equation}
      \label{eq:309}
      \left| \varphi_\lambda(\xi)\gamma(\xi) \right| \leq C_{\lambda,m,c,\mathcal{C}} \left( 1 + \left| \xi \right|^{-d+2\varepsilon}-\log \frac{c \left| \xi \right|}{2} \right),
    \end{equation}
    with \( \left| \xi \right| < min \left\{ 1/c,1 \right\} \) and \( \varepsilon := m-b \), which is locally integrable.

    For \( \xi \) large, we similarly use estimate \eqref{eq:302} of Lemma \ref{lem:23} to obtain
    \begin{equation}
      \label{eq:310}
      \left| \varphi_\lambda(\xi)\gamma(\xi) \right| \leq C_\lambda \frac{2^{1+b}\sqrt{2\pi}}{\left| \Gamma(-\lambda) \right|}c^{b+(d-1)/2} \left| \xi \right|^{-b-(d+1)/2} \e^{-c \left| \xi \right|} \e^{\left| b+d/2 \right|^2/(2c \left| \xi \right|)}
    \end{equation}
    and consequently
    \begin{equation}
      \label{eq:311}
      \left| \varphi_\lambda(\xi)\gamma(\xi) \right| \leq C_{\gamma,m,\mathcal{C},c} \e^{-c \left| \xi \right|},
    \end{equation}
    which certainly is integrable.
  \item We want to pass to \( c \rightarrow 0 \) in formula \eqref{eq:303}. This can be done by applying the Dominated Convergence Theorem in the definition of the generalized Fourier transform \eqref{eq:285}. Writing \( \Phi_c(x) := \left( c^2+ \left| x \right|^2 \right)^{\beta/2} \) for \( c>0 \), we know that
    \begin{equation}
      \label{eq:312}
      \widehat{\Phi}_c(\xi) = \varphi_c(\xi) := (2\pi)^{d/2} \frac{2^{1+\beta/2}}{\left| \Gamma(-\beta/2) \right|} \left| \xi \right|^{-\beta-d}(c \left| \xi \right|)^{(\beta+d)/2}K_{(\beta+d)/2}(c \left| \xi \right|).
    \end{equation}
    By using the decay properties of a \( \gamma \in \mathcal{S}_{2m} \) in the estimate \eqref{eq:307}, we get
    \begin{equation}
      \label{eq:313}
      \left| \varphi_c(\xi)\gamma(\xi) \right| \leq C_\gamma \frac{2^{\beta+d/2}\Gamma((\beta+d)/2}{\left| \Gamma(-\beta/2) \right|} \left| \xi \right|^{2m-\beta-d} \quad \text{for } \left| \xi \right| \rightarrow 0
    \end{equation}
    and
    \begin{equation}
      \label{eq:314}
      \left| \varphi_c(\xi)\gamma(\xi) \right| \leq C_\gamma \frac{2^{\beta+d/2}\Gamma((\beta+d)/2)}{\left| \Gamma(-\beta/2) \right|} \left| \xi \right|^{-\beta-d},
    \end{equation}
    yielding the desired uniform dominating function. The claim now follows by also taking into account that
    \begin{equation}
      \label{eq:315}
      \lim_{r\rightarrow 0} r^\nu K_\nu(r) = \lim_{r\rightarrow 0} 2^{\nu-1} \int_{0}^{\infty} \e^{-t} \e^{-r^2/(4t)} t^{\nu-1} \diff t = 2^{\nu-1} \Gamma(\nu).
    \end{equation}
  \end{enumparaalph}
\end{proof}

\begin{remark}[Fractional orders]
  \label{rem:frac-ord}
  In Theorem \ref{thm:cond-ft-power}, we have slightly changed the statement compared to the reference in order to allow orders which are a multiple of \( 1/2 \) instead of just integers. This makes sense because the definition of the order involves the space \( \mathcal{S}_{2m} \) due to its purpose in the representation formula of Theorem \ref{thm:repr-thm-cond-semi}, involving a quadratic functional. However, in Section \ref{sec:moment-bound-symm} we needed the generalized Fourier transform in a linear context.

  Fortunately, the proofs we repeated here from \cite{Wend05}, still apply: All integrability arguments remain true when permitting multiples of \( 1/2 \), in particular the estimates of \eqref{eq:307} and \eqref{eq:313}.
\end{remark}
			     
\clearpage{}

\bibliographystyle{alpha}
\bibliography{Bibliography}

\end{document}